\documentclass[10pt]{article}

\usepackage{amssymb,amsthm,amsmath,amsfonts}
\usepackage{graphicx}
\usepackage{verbatim}
\usepackage{algorithm,algorithmic}
\usepackage[usenames]{color}
\graphicspath{{./pics/}}
\usepackage[colorlinks,linktocpage,linkcolor=blue]{hyperref}

\newtheorem{lem}{Lemma}[section]

\newtheorem{prop}{Proposition}[section]

\newtheorem{exam}{Example}

\title{A Direct Sampling Method for \\ Inverse
 Electromagnetic Medium Scattering}
\author{Kazufumi Ito\footnote{Department of Mathematics and Center for
Research in Scientific Computation, North Carolina State University,
Raleigh, North Carolina (kito@unity.ncsu.edu).} \and Bangti
Jin\footnote{Department of Mathematics and Institute for Applied Mathematics
and Computational Science, Texas A\&M University, College Station, Texas
77843-3368, USA (btjin@math.tamu.edu).} \and Jun Zou\footnote{Department
of Mathematics, Chinese University of Hong Kong, Shatin, N.T., Hong Kong
(zou@math.cuhk.edu.hk).}}
\date{\today}
\begin{document}
\maketitle

\begin{abstract}
In this paper, we study the inverse electromagnetic
medium scattering problem of estimating the support and shape of
medium scatterers from scattered electric or magnetic near-field data.
We shall develop a novel direct sampling method based on an analysis of
electromagnetic scattering and the behavior of the fundamental solution.
The method is applicable even with one incident field
and needs only to compute inner products of the measured scattered
field with the fundamental solutions located at sampling points.
Hence it is strictly direct, computationally very efficient, and highly
tolerant to the presence of noise in the data. Two- and three-dimensional
numerical experiments indicate that it can provide reliable support
estimates of one single and multiple scatterers in case of both exact and highly noisy data.

\end{abstract}

\section{Introduction}

Inverse electromagnetic scattering represents an important noninvasive imaging technology for interrogating
material properties, and it arises in many practical applications such as
biomedical diagnosis \cite{Semenov:1999,Bulyshev:2004}, nondestructive testing \cite{WangChew:1989},
and geophysical exploration \cite{CuiChewAydinerChen:2001}. In this work we are
concerned with the inverse medium problem of determining the electrical/magnetic properties
of unknown inhomogeneous objects embedded in a homogeneous background from noisy measurements of the
scattered electric/magnetic field corresponding to one or several
incident fields impinged on the objects.
Mathematically, the medium scattering problem
is described by the time-harmonic Maxwell system:
\begin{equation*}
  \begin{aligned}
  \mathrm{i}\omega\epsilon\, E+\nabla\times H& =0\quad\mbox{in } ~\mathbb{R}^\mathrm{d},\\
  -\mathrm{i}\omega\mu\, H+\nabla \times E&=0\quad\mbox{in } ~\mathbb{R}^\mathrm{d},
\end{aligned}
\end{equation*}
where the vectorial fields $H$ and $E$ denote the magnetic and electric fields, respectively.
Here the constant $\omega$ is the angular frequency, the functions $\epsilon$ and $\mu$ refer to the
electrical permittivity and magnetic permeability, respectively.

Let the domain $\Omega\subset\mathbb{R}^\mathrm{d}\, (\mathrm{d}=2,3)$ be the space
occupied by the inhomogeneous medium objects within the homogeneous background $\mathbb{R}^\mathrm{d}$.
We are interested in either electrical or magnetic inhomogeneities, but shall focus our discussions on the
case of electrical inhomogeneities since the case of magnetic inhomogeneities follows analogously.
So we shall assume $\mu=\mu_0$, i.e., the magnetic permeability of the background.
Then by taking the $\mathrm{curl}$ of the second equation of the system and
eliminating $H$ in the first equation, we obtain
the following vector Helmholtz equation for the electric field $E$
\begin{equation}\label{eqn:maxwell}
   \nabla\times(\nabla \times E)-k^2n^2(x)\,E=0
\end{equation}
where $k$ is the wavenumber, with $k^2=\omega^2\epsilon_0\mu_0$,
and $n=\frac{1/\sqrt{\epsilon_0\mu_0}}{1/\sqrt{\epsilon\mu_0}}$ is the refractive index function,
i.e., the ratio of the wave velocity in the homogeneous background medium to that in the
inhomogeneities. The refractive index $n$ completely characterizes the electrical inhomogeneities, and the support
is filled by the inhomogeneous media, i.e.,
$\mathrm{supp}(n^2-1)=\Omega$. Further, we assume that the medium scattering is excited by an incident plane
wave $E^{i}$:
\begin{equation*}
  E^{i} = pe^{\mathrm{i}kd\cdot x},
\end{equation*}
where $d\in\mathbb{S}^{\mathrm{d}-1}$ and $p\in\mathbb{S}^{\mathrm{d}-1}$ are
the incident and polarization directions, respectively. Since the incident field $E^i$ is solenoidal, i.e.,
$\nabla\cdot E^i=0$, the polarization $p$ should be chosen such that
it is perpendicular to the incident direction $d$. Then the incident field
$E^i$ satisfies the Maxwell system \eqref{eqn:maxwell}
in the entire homogeneous background. The forward scattering
problem is to find the total electric field $E$ given the refractive index $n^2$,
under the following Silver-M\"{u}ller radiation condition
for  the scattered field $E^s=E-E^i$:
\begin{equation*}
  \lim_{|x|\rightarrow\infty}|x|^\frac{\mathrm{d}-1}{2}\left(\nabla\times E^s\times\hat{x}-\mathrm{i}kE^s\right)=0
\end{equation*}
uniformly for all directions $\hat{x}=x/|x|\in\mathbb{S}^{\mathrm{d}-1}$.

The inverse problem of our interest is the inverse medium scattering problem, which is
to reconstruct the inhomogeneous media from the scattered electric
field $E^s$ corresponding to one (or several) incident field $E^i$, measured over a certain
closed curve/surface $\Gamma$.
In view of the practical significance of the inverse problem, there has been considerable
interest in designing efficient and stable inversion techniques. However, this is very challenging
because of a number of complicating factors: strong nonlinearity of the map from the
refractive index to the scattered field, severe ill-posedness of the inverse problem,
complexity of the forward model, and the limited available and noisy data. Nonetheless,
a number of imaging algorithms have been developed in the literature,
which can be roughly divided into two categories: direct and indirect methods. The former
aims at detecting the scatterer support and shape, and includes linear sampling method
(LSM) \cite{ColtonKirsch:1996,CakoniColtonMonk:2011},
multiple signal classification (MUSIC) \cite{Devaney:2006}, and asymptotic analysis
\cite{AmmariIakovlevaLesselierPerrusson:2007}. In contrast,
the latter provides a distributed estimate of the refractive index by applying
regularization techniques. We refer interested readers to
the adjoint-based method \cite{DornBerteteBerrymanPapanicolaou:1999,Vogeler:2003,Lakhal:2010},
the recursive linearization (with continuation in frequency) \cite{BaoLi:2005siam,BaoLi:2009jcp},
the Gauss-Newton method \cite{Hohage:2006,ZaeytijdFranchois:2007,HohageLanger:2011},
the contrasted source inversion \cite{AbubakarHabashy:2005}, subspace regularization \cite{ChenZhong:2009} and level set method \cite{DornLesselier:2006} for an
incomplete list. Generally, the estimates by the method of the latter category
can provide more details of the inclusions/inhomogeneities,
but at the expense of much increased computational efforts, especially when the
forward model is the full Maxwell system.

In this work, we shall develop a novel direct method, called the direct sampling method,
for stably and accurately detecting the support of the scatterers. These sampling-type methods
are direct in the sense that no optimizations or solutions of linear systems are involved, so they are
computationally very cheap. They have attracted some attentions in very recent years;
see \cite{Potthast:2010} and \cite{ItoJinZou:2012} for inverse acoustic scattering problems
using far-field data and near-field data, respectively.
A different derivation of the method due to Potthast \cite{Potthast:2010}
was given in \cite{LiZou:2012}, where the performance of the methods
using near-field and far-field data was compared in detail and the effectiveness
of these methods was also investigated for several important scattering scenarios:
obstacles, inhomogeneous media, cracks and their combinations.

The major goal of this work is to extend
the direct sampling method developed in \cite{ItoJinZou:2012}
for the acoustic scattering to the electromagnetic scattering.
Due to the much increased complexity of the Maxwell's system relative to
its scalar counterpart, the Helmholtz equation, the extension is nontrivial and
requires several innovations.
It is based on an integral representation of the scattered field,
a careful analysis on electromagnetic scattering
and the behavior of the fundamental solutions. Numerically, it involves only computing inner products of the
measured scattered field $E^s$ with fundamental solutions to the Maxwell system
located at sampling points over the measurement surface $\Gamma$. Hence it
is strictly direct and does not use any matrix operations or minimizations, and its implementation is
very straightforward. Our extensive numerical experiments indicate that it can provide an
accurate and reliable estimate of the scatterer support, even in the presence of a fairly large amount
of noises in the data. Hence, it represents an effective yet simple computational tool for
reliably detecting the scatterer support. In practice, a rough estimate of the scatterer support
may be sufficient for many purposes \cite{TortelMicolauSaillard:1999}. And
if desired, one can obtain a more refined estimate of the medium scatterers by, e.g.,
any aforementioned indirect imaging methods or the one from \cite{ItoJinZou:2012jcp},
using the estimate from the direct sampling method to provide an initial computational domain,
which is usually much smaller than the originally selected sampling domain.
Since indirect imaging methods often involve highly nonlinear optimization processes,
an accurate and smaller initial domain can essentially reduce the entire computational
efforts.

The direct sampling method developed below
uses a sampling strategy, and its flavor closely resembles MUSIC and the LSM
(see \cite{Potthast:2006,CakoniColtonMonk:2011} for overviews).
However, it differs significantly from these
two techniques. Firstly, it works with a few incident fields, whereas the latter
two require the full map (multi-static response matrix/far-field operator) or data
from sufficiently many incidents. Secondly,
our method does not perform any matrix operations, e.g.
eigendecomposition in MUSIC or solving ill-posed integral equations
in the LSM. Hence, our method is computationally efficient. Lastly, the noise is
treated directly via the inner product, which automatically filters out the noise
contribution, and thus the method is highly tolerant to noises.

The rest of the paper is organized as follows. In Section 2, we recall
an integral reformulation of the Maxwell system, cf. \eqref{eqn:maxwell},
recently derived in \cite{LakhalLouis:2008}, which plays an essential role in the
derivation of the direct sampling method. Then we develop the
method in Section 3 in detail, where a preliminary analysis of its theoretical performance
is also provided. In Section 4, we provide two-
and three-dimensional numerical experiments to
illustrate its distinct features, i.e., accuracy and robustness, for both exact and noisy data.
Technical details of the implementation are provided in the appendix.

\section{Integral representation of Maxwell System}
In this part, we recall an equivalent formulation of the Maxwell system \eqref{eqn:maxwell},
which is fundamental to the derivation of the direct sampling
method. We begin with the definition of the fundamental solution
$G(x,y)$ to the scalar Helmholtz equation, i.e.,
\begin{equation*}
  (-\Delta -k^2)G(x,y) = \delta(x-y)
\end{equation*}
where $\delta(x-y)$ is the Dirac delta function with the singularity located at
$y\in\mathbb{R}^\mathrm{d}$. We know that $G(x,y)$ has the following representations
(see, e.g., \cite{ColtonKress:1998})
\begin{equation*}
  G(x,y) = \left\{\begin{aligned}
    \frac{\mathrm{i}}{4}H_0^{(1)}(k|x-y|),&\quad\mathrm{d}=2;\\
    \frac{1}{4\pi}\frac{e^{\mathrm{i}k|x-y|}}{|x-y|}, &\quad \mathrm{d}=3,
  \end{aligned}\right.
\end{equation*}
where the function $H_0^{(1)}$ refers to Hankel's function of the first kind and zeroth order.
Using the scalar function $G(x, y)$ we can define a matrix-valued function $\Phi(x,y)$ by
\begin{equation}\label{eq:phi}
\Phi(x,y)=k^2G(x,y)\,I+D^2G(x,y)
\end{equation}
where $I\in\mathbb{R}^{\mathrm{d}\times\mathrm{d}}$ is the identity
matrix and $D^2$ denotes the Hessian of $G$. Then we can verify by some direct calculations
that $\nabla \cdot \Phi(x,y)=0$ and
\begin{equation}
  \nabla\times\nabla\times \Phi(x,y) - k^2\Phi(x,y) = \delta(x-y)I, \label{eq:Green}
\end{equation}
where (and in the sequel) the actions of the operators $\nabla \cdot$ and $\nabla\times$ on a matrix-valued
function are always understood to be operated columnwise. Hence,
the matrix $\Phi(x,y)$ defined by \eqref{eq:phi} is a divergence-free fundamental solution to
the Maxwell system \eqref{eqn:maxwell} in the homogeneous space $\mathbb{R}^\mathrm{d}$.
Using the fundamental solution $\Phi(x,y)$,
the total electric field $E(x)$ can be represented
by the following integral equation
\begin{equation}\label{eqn:int}
  E(x) = E^{i} + \int_{\mathbb{R}^\mathrm{d}}\Phi(x,y)(n^2-1)E(y)dy.
\end{equation}
We note that the fundamental solution $\Phi(x,y)$ involves a non-integrable singularity at
 $x=y$. Hence care must be exerted when interpreting the integral, in the case that the
point $x$ lies within the domain $\Omega$ \cite{HabashyGroomSpies:1993}.
Next we let $\eta=n^2-1$, which precisely profiles the inhomogeneities of the media. In particular,
the support of $\eta$ coincides with the scatterer support $\Omega$. Furthermore,
we introduce the function $J=(n^2-1)E$, that is the induced electrical current caused
by the medium inhomogeneities. Then, the total electric field $E(x)$ satisfies
\cite[Theorem 9.1]{ColtonKress:1998}
\begin{equation*}
  E(x)=E^{i}(x)+k^2\int_{\mathbb{R}^\mathrm{d}}G(x,y)J(y)\,dy+\nabla_x \int_{\mathbb{R}^\mathrm{d}}
  G(x,y)\mbox{div}_y J(y)\,dy.
\end{equation*}
Upon noting the reciprocity relation $\nabla_x G(x,y)=-\nabla_yG(x,y)$ and applying integration
by parts to the second integral term on the right hand side, we arrive at the following equivalent integral equation
\begin{equation*}
E(x) - \int_{\mathbb{R}^\mathrm{d}} G(x,y)PJ(y)\,dy=E^{i}(x)
\end{equation*}
where the operator $P$ is defined by
\begin{equation*}
   P\phi=k^2 I\phi+\mbox{grad}(\mbox{div}\phi).
\end{equation*}
Thus, by multiplying the equation with the coefficient $\eta$, we obtain an integral equation for the
generalized electric current $J$
\begin{equation} \label{Jeq}
J(x)-\eta\int_{\Omega}G(x,y)PJ(y)\,dy=\eta E^i(x),\quad x\in\Omega.
\end{equation}
The above equation was rigorously justified in suitable function spaces in
\cite{LakhalLouis:2008}, and it is very convenient for
solving inverse problems; see \cite{LakhalLouis:2008} for
the inverse source problem and \cite{Lakhal:2010} for inverse medium scattering.
Compared with the whole-space Maxwell system, the integral
equation \eqref{Jeq} is defined over the scatterer support $\Omega$ since the induced current $J$ vanishes
identically outside the scatterer support $\Omega=\mathrm{supp}(\eta)$.
This reduces greatly the computational domain, and hence brings significant computational
convenience. We shall adopt the integral equation \eqref{Jeq} for the forward scattering simulation, which
can be discretized numerically by the mid-point quadrature rule (cf. Appendix \ref{app:int} for details).

\section{Direct sampling method}

In this section, we develop a novel direct sampling method
to determine the locations,
the number and the shape of the scatterers/inhomogeneities in
electromagnetic wave propagation. It is based on an analysis
for electromagnetic scattering. Methodologically, it extends our earlier work
on the acoustic scattering \cite{ItoJinZou:2012}. We shall also provide an analysis of the
theoretical performance of the method by examining the behaviors of the
fundamental solution.
For the sake of convenience, we introduce the domain $\Omega_\Gamma$, which is the domain
enclosed by the circular/spherical measurement surface $\Gamma$, and use
$(\cdot,\cdot)$ for the real inner product on $\mathbb{C}^\mathrm{d}$ and the overbar
for the complex conjugate.

\subsection{Direct sampling method}
In this part, we develop a novel direct sampling method (DSM). The derivation relies essentially
on the following two basic facts. The first is the representation of the scattered electric field
$E^s$ using the fundamental solution $\Phi$ by  (cf.\,\eqref{eqn:int}):
\begin{equation}
E^s(x)=\int_{\Omega} \Phi(x,y)J(y)dy\quad \forall\,x\in\Gamma.
\label{eq:sf} 
\end{equation}

The second fact is an important relation for the fundamental solution
$\Phi(x,y)$ \eqref{eq:phi} to the Maxwell system \eqref{eqn:maxwell}.
For any two arbitrary sampling points $x_p$ and  $x_q$ which lie inside
domain $\Omega_\Gamma$ but are away from its boundary $\Gamma$, we have
the following approximation
\begin{equation}\label{eqn:xpq}
\int_{\Gamma}(\Phi(x,x_p)p,\overline{\Phi}(x,x_q)q)ds \approx  k^{-1}
\Im (p, \Phi(x_p,x_q)q)
\quad \forall\,p, \,q\in\mathbb{R}^\mathrm{d}\,.
\end{equation}

Next, we derive this crucial correlation. To do so, we first show the following lemma.
\begin{lem}
Let $x_p$ and $x_q$ be two distinct points  lying inside the domain $\Omega_\Gamma$. Then
for any constant vectors $p\in \mathbb{C}^\mathrm{d}$ and $q\in\mathbb{R}^\mathrm{d}$, there holds that
\begin{equation}\label{eqn:intid}
  \begin{aligned}
   \int_\Gamma(\nabla\times \overline{\Phi}(x,x_q)q\times n,\Phi(x,x_p)p)-(\nabla \times\Phi(x,x_p)p\times n,\overline{\Phi}(x,x_q)q)ds
      = - 2\mathrm{i}(p,\Im(\Phi(x_p,x_q))q).
  \end{aligned}
\end{equation}
\end{lem}
\begin{proof}
It is easy to verify directly that
the identify $(\nabla\times\,\nabla \times \Phi(x,x_p)) p =
\nabla\times\,\nabla \times \Phi(x,x_p) p$ holds
for any constant vector $p\in\mathbb{R}^\mathrm{d}$.
Using this identify and (\ref{eq:Green}) we have
\begin{eqnarray}
  \nabla\times\,\nabla \times \Phi(x,x_p) p - k^2\Phi(x,x_p) p&=& \delta(x-{x_p}) p \quad
   \forall\, p\in\mathbb{R}^\mathrm{d}\,, \label{eqn:fdsp}\\
  \nabla \times\nabla \times \Phi(x,x_q)q - k^2\Phi(x,x_q)q&=& \delta(x-{x_q})q \quad
   \forall\, q\in\mathbb{R}^\mathrm{d}\,. \label{eqn:fdsq1}
\end{eqnarray}
Taking the conjugate of equation \eqref{eqn:fdsq1} yields
\begin{equation}\label{eqn:fdsq}
    \nabla\times\nabla\times \overline{\Phi}(x,x_q)q - k^2\overline{\Phi}(x,x_q)q= \delta(x-{x_q})q.
\end{equation}
By taking (real) inner products of equation \eqref{eqn:fdsp} with $\overline{\Phi}(x,x_q)q$  and of
equation \eqref{eqn:fdsq} with $\Phi(x,x_p)p$, then subtracting the two identities, we arrive at
\begin{equation}\label{eq:xpq}
  \begin{aligned}
    \int_{\Omega_\Gamma}\big\{(\overline{\Phi}(x,x_q)q,\nabla\times\nabla\times \Phi(x,x_p)p)-
       (\Phi(x,x_p)p,&\nabla\times\nabla\times \overline{\Phi}(x,x_q)q) \big\}dx\\
    &= (p,\overline{\Phi}(x_p,x_q)q)-(q,\Phi(x_p,x_q)p).
   \end{aligned}
\end{equation}
Next we apply the following integration by parts formula
\begin{equation*}
  \begin{aligned}
   \int_{\Omega_\Gamma}\big\{(\overline{\Phi}(x,x_q)q,&\nabla\times\nabla\times \Phi(x,x_p)p) - (\Phi(x,x_p)p,\nabla\times\nabla\times \overline{\Phi}(x,x_q)q)\big\}dx\\
   =& \int_\Gamma(\nabla\times \overline{\Phi}(x,x_q)q\times n,\Phi(x,x_p)p)-(\nabla\times {\Phi}(x,x_p)p\times n, \overline{\Phi}(x,x_q)q)ds
  \end{aligned}
\end{equation*}
to the left hand side of (\ref{eq:xpq}) to obtain
\begin{equation*}
   \begin{aligned}
   \int_\Gamma(\nabla\times \overline{\Phi}(x,x_q)q\times n,\Phi(x,x_p)p) - (\nabla\times &\Phi(x,x_p)p\times n,\overline{\Phi}(x,x_q)q)ds\\
    &= (p,\overline{\Phi}(x_p,x_q)q) - (q,\Phi(x_p,x_q)p).
   \end{aligned}
\end{equation*}
Now the real symmetry of the fundamental solution $\Phi(x,y)$ leads directly to the identify
\eqref{eqn:intid}.
\end{proof}

%

Next, note that on a circular curve/spherical surface $\Gamma$, we can approximate the left
hand side of identity \eqref{eqn:intid} by means of the
Silver-M\"{u}ller radiation condition for the outgoing fundamental solution $\Phi(x,y)$ to the Maxwell system, i.e.,
\begin{equation*}
\nabla\times \Phi(x,x_p)p \times n =\mathrm{i}k\Phi(x,x_p)p + \mathrm{h.o.t.}
\end{equation*}
Thus we have the following approximations:
\begin{equation*}
   \begin{aligned}
     \nabla\times \Phi(x,x_p)p\times n &\approx \mathrm{i}k\Phi(x,x_p)p,\\
     \nabla\times \overline{\Phi}(x,x_q)q\times n&\approx -\mathrm{i}k\overline{\Phi}(x,x_q)q,
   \end{aligned}
\end{equation*}
which are valid if the points $x_p$ and $x_q$ are not close to the boundary $\Gamma$.
Consequently, we arrive at the following important approximate relation:
\begin{equation*}
-\int_{\Gamma}\Big\{(\mathrm{i}k\,\Phi(x,x_p)p,\overline{\Phi}(x,x_q)q)+(\mathrm{i}k\,\overline{\Phi}(x,x_q)q,\Phi(x,x_p)p)\Big\}ds
\approx - 2\mathrm{i}(p,\Im(\Phi(x_p,x_q))q).
\end{equation*}
Upon simplifying the relation, we arrived at the desired relation \eqref{eqn:xpq}.

The relation \eqref{eqn:xpq} leads us to a very important observation:
the following integral over the measurement surface $\Gamma$
\begin{equation} \label{cross}
   \langle\Phi(\cdot,x_p)p,\Phi(\cdot,x_q)q\rangle_{L^2(\Gamma)}=\int_{\Gamma}(
    \Phi(x,x_p)p,\overline{\Phi}(x,x_q)q)ds
\end{equation}
has a maximum if $x_p=x_q$ and decays to $0$ as $|x_p-x_q|$
tends to $\infty$, in view of the decay property of the fundamental solution $\Phi(x,y)$.

These two basic facts \eqref{eq:sf} and \eqref{eqn:xpq} are crucial to the derivation of 
our direct sampling method. 
Next we consider a sampling domain $\widetilde \Omega\subset\Omega_\Gamma$
that contains the scatterer support $\Omega$,  and divide it
into a set of small elements $\{\tau_j\}$. Then by using the rectangular quadrature rule,
we arrive at the following discrete sum representation
\begin{equation}\label{eqn:sum}
E^s(x)=\int_{\widetilde{\Omega}}\Phi(x,y)J(y)dy\approx \sum_{j}\Phi(x,y_j)J(y_j)\,|\tau_j|,
\end{equation}
where the point $y_j$ lies within the $j$th element $\tau_j$,
and $|\tau_j|$ denotes
the volume/area of the element $\tau_j$. Noting the fact that the
generalized induced current $J$ vanishes identically outside the support
$\Omega$, the summation in \eqref{eqn:sum} is actually only over
those elements that intersect with $\Omega$. We
point out that, in practice, the scatterer support $\Omega$ may consist of several disjoint
subregions, each being occupied by a (possibly different) physical medium. By elliptic
regularity theory \cite{ColtonKress:1998,WeberWerner:1981}, the induced
current $J=\eta\,E$ is smooth in each subregion, and thus according to classical
approximation theory, the approximation in \eqref{eqn:sum} can be made arbitrarily
accurate by refining the elements $\{\tau_j\}$. Nonetheless, we reiterate that
the relation \eqref{eqn:sum} serves only the goal of motivating our method, and
it is not needed in the implementation of our method. Physically,  the relation
\eqref{eqn:sum} can be  interpreted as follows: the scattered electric field $E^s$ at any fixed point
$x\in\Gamma$ is a weighted average of that due to the point scatterers located at
$\{y_j\}$ lying within the true scatterer $\Omega$.

The approximate relations \eqref{eqn:xpq} and \eqref{eqn:sum} together give that
for any sampling point $x_p\in \widetilde{\Omega}$ and any constant vector $q\in\mathbb{R}^\mathrm{d}$, there holds
\begin{equation}\label{eq:crucial}
  \begin{aligned}
    \langle E^s, \Phi(\cdot,x_p)q\rangle_{L^2(\Gamma)} &\approx \left\langle \sum_j\Phi(\cdot,y_j)J(y_j)
    |\tau_j|,\Phi(\cdot,x_p)q\right\rangle_{L^2(\Gamma)}\\
      & =\sum_j|\tau_j| \langle\Phi(\cdot,y_j)J(y_j),\Phi(\cdot,x_p)q\rangle_{L^2(\Gamma)}\\
      & \approx k^{-1}\sum_j |\tau_j|(J(y_j), \Im(\Phi(x_p,y_j))q).
  \end{aligned}
\end{equation}
Due to the behavior of the fundamental solution (see the discussions
in Section 3.2), the approximation (\ref{eq:crucial}) indicates
that the inner product $\langle E^s, \Phi(\cdot,x_p)q\rangle$ may take significant
values if the sampling point $x_p$ is close to some physical point
scatterer $y_j$ and small values otherwise, thereby
acting as an indicator function to the presence of scatterers and
equivalently providing an estimate of the scatterer support $\Omega$.
In summary, these observations lead us to the following index function:
\begin{equation}\label{eqn:ind}
  \Psi(x_p;q) = \frac{|\langle E^s,\Phi(\cdot,x_p)q\rangle_{L^2(\Gamma)}|}
  {\|E^s\|_{L^2(\Gamma)}\|\Phi(\cdot,x_p)q\|_{L^2(\Gamma)}}\,, \quad
  \forall\, x_p\in\widetilde{\Omega}.
\end{equation}
Here $\Phi(\cdot,x_p)q$ acts as a probe/detector
for the scatterers. In principle, the choice of the polarization vector $q$ in the index
$\Psi(x_p;q)$ can be quite arbitrary. Naturally, it is expected that the choice
of $q$ will affect the capability of the function $\Phi(x,x_p)q$ for probing scatterers.
The analysis in Section 3.2 will shed insights into the probing
mechanism and will provide useful guidelines for the choice of $q$.
Our experimental experiences indicate that the choice $q=p$, i.e., the
polarization of the incident field $E^{i}$, works very well in practice.
We note that, apart from providing an estimate of the scatterer support,
the index $\Psi(x_p;q)$ provides
a likelihood distribution of the inhomogeneities in $\widetilde{\Omega}$, which up to
a multiplicative constant may be used as the an initial estimate of the distributed
coefficient $\eta$ \cite{ItoJinZou:2012jcp}.

The DSM is one sampling type imaging technique. There are several existing sampling
techniques, e.g., MUSIC and the LSM. The DSM differs from these techniques
in several aspects. First, the DSM involves computing only inner products of
the measured scattered field  with the probing functions $\Phi(x,x_p)q$, i.e., the fundamental
solutions with singularity concentrated at sampling points. Therefore, it does not
involve any expensive linear algebraic operations, like
solving first-kind ill-posed integral equations in the LSM and eigenvalue decomposition
for computing the signal space (noise space) in the MUSIC.
Secondly, the DSM is applicable to a few (e.g., one or two) incident fields, whereas
the LSM requires the far field map and the MUSIC requires the multi-statistic matrix, which
necessitates multiple measurements. Therefore, the DSM is particularly attractive in
the case of a few scattered data.
Thirdly, the DSM treats noise directly via the inner product. The noise roughly distributes
the energy in all frequencies equally, whereas the fundamental solutions are smooth with their energy
concentrated on low frequency modes. Therefore, the high-frequency modes in the noise
are roughly orthogonal to the (smooth) fundamental solutions.
In other words, the presence of noise does not cause
much changes in the inner product, but the normalization. As a consequence, the
DSM  is highly tolerant to data noise.

\subsection{Analysis of the index function $\Psi$}
In this part, we analyze the theoretical performance and the mechanism of the DSM by
analytically and numerically studying the fundamental solution $\Phi(x,y)$ in (\ref{eq:phi})
and the index $\Psi$ in (\ref{eqn:ind}) for one single point scatterer.
First, we recall the crucial role of the approximate
relation \eqref{eqn:xpq}: the fundamental solution $\Phi(x_p,x_q)$ is
nearly singular, and assumes very large values for $x_p$ close to $x_q$.
More precisely, the extremal property of $\Im(\Phi(x_p,x_q))$ paves the way for accurate
support detection. This observation is evident for the scalar Helmholtz equation
in view of the identity $\Im(G(x_p,x_q))=\frac{1}{4}J_0(k|x_p-x_q|)$ (in
two-dimension), where $J_0$ is the Bessel function of the first kind
of order zero. However, for the Maxwell system, the fundamental solution $\Phi(x,y)$
contains 4 or 9 different entries, and each term exhibits drastically different behavior.
So with each different polarization vector $q$, the probing function $\Phi(x,x_p)q$, as well as
the index function $\Psi(x_p;q)$ in \eqref{eqn:ind}, mixes different components together and may yield
quite different behaviors. Next we shall investigate more closely the properties of $\Phi(x,y)$ in order to have a better understanding
and interpretation of the index function $\Psi(x_p;q)$ in \eqref{eqn:ind}.

We begin with an important observation on the trace of the fundamental solution $\Phi(x,y)$.
\begin{prop}
For $\mathrm{d}=2,3$, the fundamental solution $\Phi(x,y)$ satisfies the following relation
\begin{equation}\label{eq:phi2}
\mathrm{tr}\Phi(x,y) = (\mathrm{d}-1)k^2G(x,y).
\end{equation}
\end{prop}
\begin{proof}
Let us begin with the two-dimensional case, i.e., transverse electric mode,
i.e., $E=(E_1,E_2,0)$ and $H=(0,0,H_3)$. Then
the Maxwell system \eqref{eqn:maxwell} should be understood as
\begin{equation*}
  \begin{aligned}
     \mathrm{i}\omega\mu_0\, H_3&=(\nabla \times E)_3= \tfrac{\partial E_2}{\partial x_1}-\tfrac{\partial E_1}{\partial x_2}\\
     -\mathrm{i}\omega \epsilon \, E&=\nabla \times H= \left(\tfrac{\partial H_3}{\partial x_2},-\tfrac{\partial H_3}{\partial x_1}\right).
  \end{aligned}
\end{equation*}
The fundamental solution $G(x,y)$ is given by $G(x,y)=\frac{\mathrm{i}}{4}H^{(1)}_0(k\,|x-y|)$.
To evaluate the Hessian $D^2G(x,y)$, we first recall the recursive relation for the derivative
of Hankel functions \cite{AbramowitzStegun:1968}
\begin{equation*}
  \tfrac{d}{dz}H_n^{(1)}(z) = \tfrac{nH_n^{(1)}(z)}{z}-H_{n+1}^{(1)}(z).
\end{equation*}
Hence, by chain rule and product rule, we deduce that
\begin{equation*}
  \tfrac{\partial^2}{\partial x_i\partial x_j} H_0^{(1)}(k|x-y|) %
   = k^2H_2^{(1)}(k|x-y|)\tfrac{(x-y)_i(x-y)_j}{|x-y|^2}-k\tfrac{1}{|x-y|}H_1^{(1)}(k|x-y|)\delta_{i,j},
\end{equation*}
where $\delta_{ij}$ is the Kronecker delta function.
%
Consequently, the fundamental solution $\Phi(x,y)$ is given by
\begin{equation*}
\Phi(x,y)_{i,j}=\tfrac{\mathrm{i}k^2}{4}\left[\left(H_0^{(1)}(k|x-y|)-\tfrac{H_1^{(1)}(k|x-y|)}{k|x-y|}\right)\delta_{i,j}+
   H^{(1)}_2(k|x-y|)\tfrac{(x-y)_i(x-y)_j}{|x-y|^2}\right].
\end{equation*}
Upon noting the recursive relation $\frac{2H_1^{(1)}(z)}{z}=H_0^{(1)}(z) + H_2^{(1)}(z)$
for Hankel functions \cite{AbramowitzStegun:1968}, we deduce that the trace of the fundamental solution satisfies
\begin{equation*}
  \begin{aligned}
    \mathrm{tr}(\Phi(x,y))&=\tfrac{\mathrm{i}k^2}{4}\left[2H_0^{(1)}(k|x-y|)-2\tfrac{H_1^{(1)}(k|x-y|)}{k|x-y|}+H^{(1)}_2(k|x-y|)\right]\\
       & = \tfrac{\mathrm{i}k^2}{4}H_0^{(1)}(k|x-y|)=k^2G(x,y).
  \end{aligned}
\end{equation*}

Next we turn to the three-dimensional case, i.e.,
$G(x,y)=\frac{1}{4\pi}\frac{e^{\mathrm{i}k|x-y|}}{|x-y|}$.
A direct calculation yields
\begin{equation*}
  \begin{aligned}
    \Phi(x,y)_{ij} =& G(k|x-y|)\left[k^2\left(\delta_{ij}-\tfrac{(x-y)_i(x-y)_j}{|x-y|^2}\right)
     -\left(\delta_{i,j}-\tfrac{3(x-y)_i(x-y)_j}{|x-y|^2}\right)\tfrac{1}{|x-y|^2}\right.\\
     &\left.+\mathrm{i}k\left(\delta_{i,j}-\tfrac{3(x-y)_i(x-y)_j}{|x-y|^2}\right)\tfrac{1}{|x-y|}\right].
  \end{aligned}
\end{equation*}
Consequently, $\mbox{tr}(\Phi(x,y))=2k^2G(x,y)$.
\end{proof}

It follows from \eqref{eq:phi2} that
the sum of the diagonal components of the fundamental solution
$\Phi(x,y)$ to the Maxwell system behaves
like the scalar fundamental solution $G(x,y)$, which, according to the experimental
evidences in \cite{ItoJinZou:2012}, can provide an excellent probing function
of the scatterer support. To be more precise, this yields
\begin{equation} \label{Id}
  k^{-1}\Im(\mbox{tr}\,\Phi(x,y))=
  \left\{  \begin{aligned}
  \frac{k^2}{4}\frac{J_0(k|x-y|)}{k},&\quad \mathrm{d}=2,\\
  2k^2\frac{\sin(k|x-y|)}{k|x-y|}, &\quad \mathrm{d}=3.
  \end{aligned}\right.
\end{equation}
Therefore, the crucial quantity $k^{-1}\Im(\mbox{tr}\Phi(x_p,x_q))$ attains its maximum at
$x_p=x_q$,  indicating the presence of the scatterers. We note that apart from
the global maximum, there are also a number of local maxima, which introduce ``ripples'' into the
resulting index. However, we note that the positions of the local
maxima depend only on the known wavenumber $k$, and can be strictly calibrated.

Next, we examine combinations of the components of the fundamental solution $\Phi(x,y)$.
We first note that each independent component $\Im\Phi_{ij}(x,x_q)$ does have
a significant maximum at $x=x_q$ like \eqref{Id}. However, the magnitude
of each component differs dramatically. In Fig.\,\ref{fig:fds1}, we show the
inner products of the components $\Phi_{11}$, $\Phi_{22}$, $\Phi_{12}$ with
themselves (termed cross product for short) and the diagonal
sum as a function of the sampling point $x_p\in\widetilde{\Omega}$, i.e., $|\langle
\Phi_{11}(x_p,x),\Phi_{11}(x_q,x)_{L^2(\Gamma)}\rangle|$ etc.
We note that these cross products form the building blocks of the index function $\Psi(x_p;q)$,
and their behaviors completely determine the performance of the DSM.
The cross products for $\Phi_{11}$ and $\Phi_{22}$ exhibit significant
directional resonance effects but their sum, i.e., $\langle\Phi_{11}(x,x_q),
\Phi_{11}(x,x_p)\rangle_{L^2(\Gamma)}+\langle\Phi_{22}(x,x_q),\Phi_{22}(x,x_p)
\rangle_{L^2(\Gamma)}$
does only assume one significant maximum at $x_q=x_p$ without much resonance
effects. In practice, due to the mixing of these different
components in the index $\Psi(x_p;q)$, the performance of the DSM
may not be as good as that for the scalar Helmholtz case
in the case of the scattered electric field $E^s$
of one polarization $p$. Nonetheless, with multiple polarizations $p_k$
and all components of the respective scattered field $E^s$, one might be able to
approximate the trace $\mathrm{tr}\Phi(x,y)$, and \eqref{Id} can be applied.
If this is indeed the case, then the performance of the sampling method would equal that for
the scalar case. 

\begin{figure}[h] \centering
  \begin{tabular}{cc}
     \includegraphics[trim = 2cm 1cm 2cm 1cm, clip=true,width=6cm]{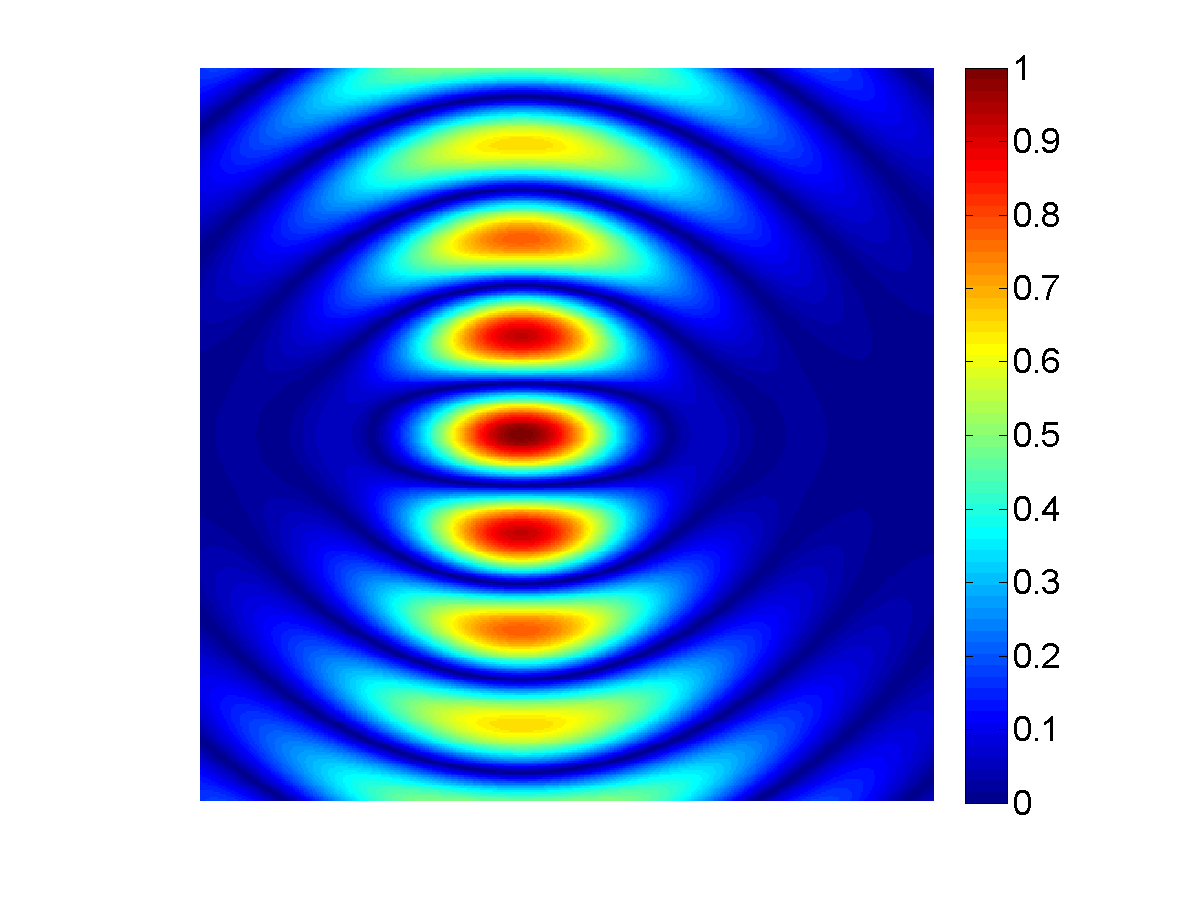} & \includegraphics[trim = 2cm 1cm 2cm 1cm, clip=true,width=6cm]{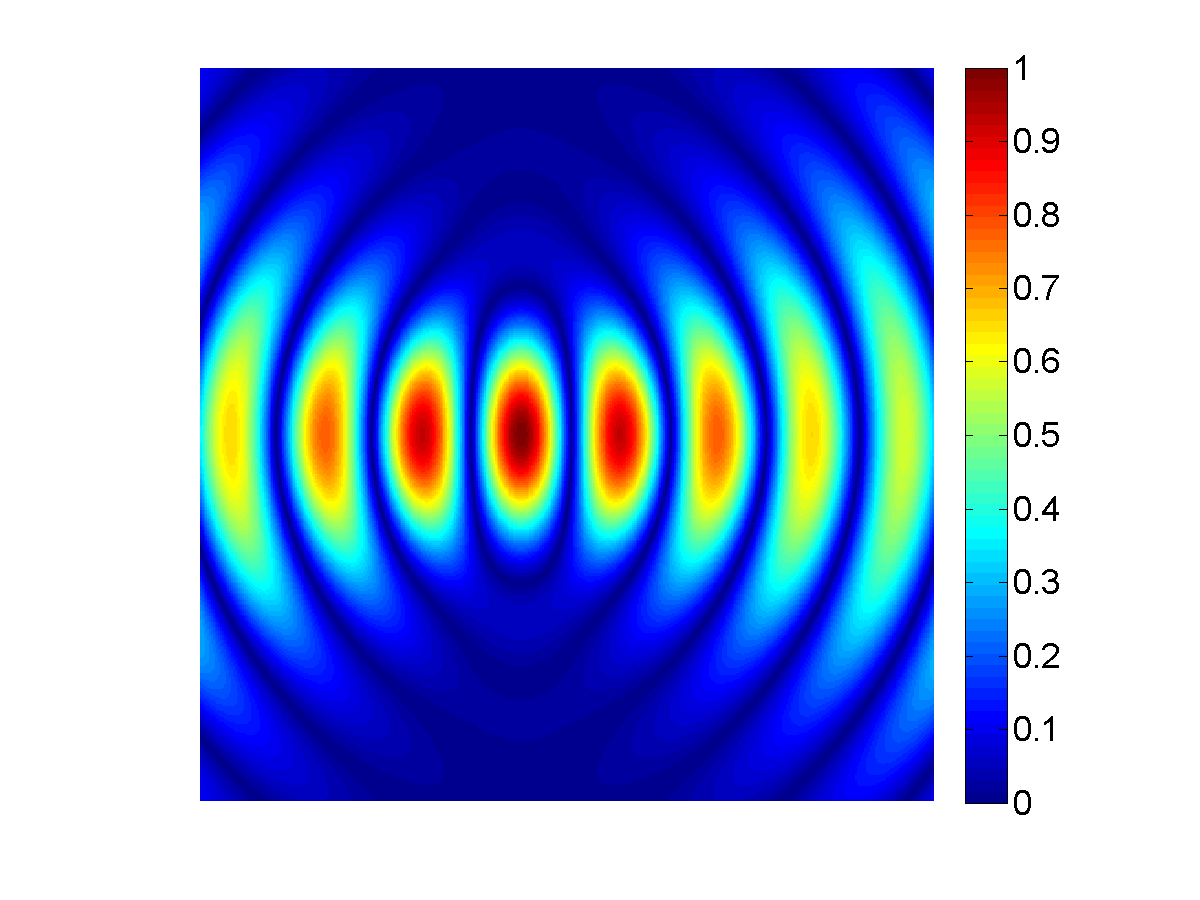}\\
     $|\langle\Phi_{11}(x,x_p),\Phi_{11}(x,x_q)\rangle_{L^2(\Gamma)}|$ & $|\langle\Phi_{22}(x,x_p),\Phi_{22}(x,x_q)\rangle_{L^2(\Gamma)}|$\\
     \includegraphics[trim = 2cm 1cm 2cm 1cm, clip=true,width=6cm]{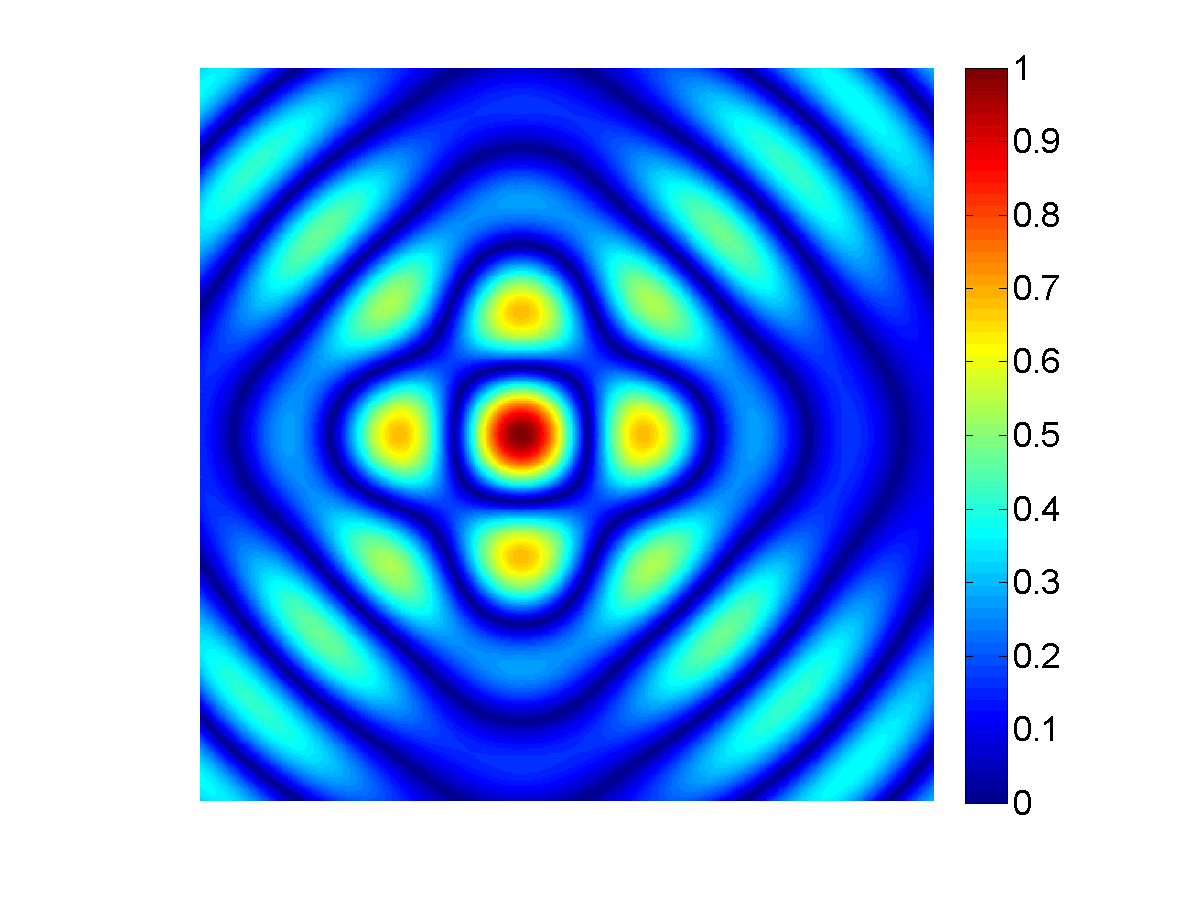} & \includegraphics[trim = 2cm 1cm 2cm 1cm, clip=true,width=6cm]{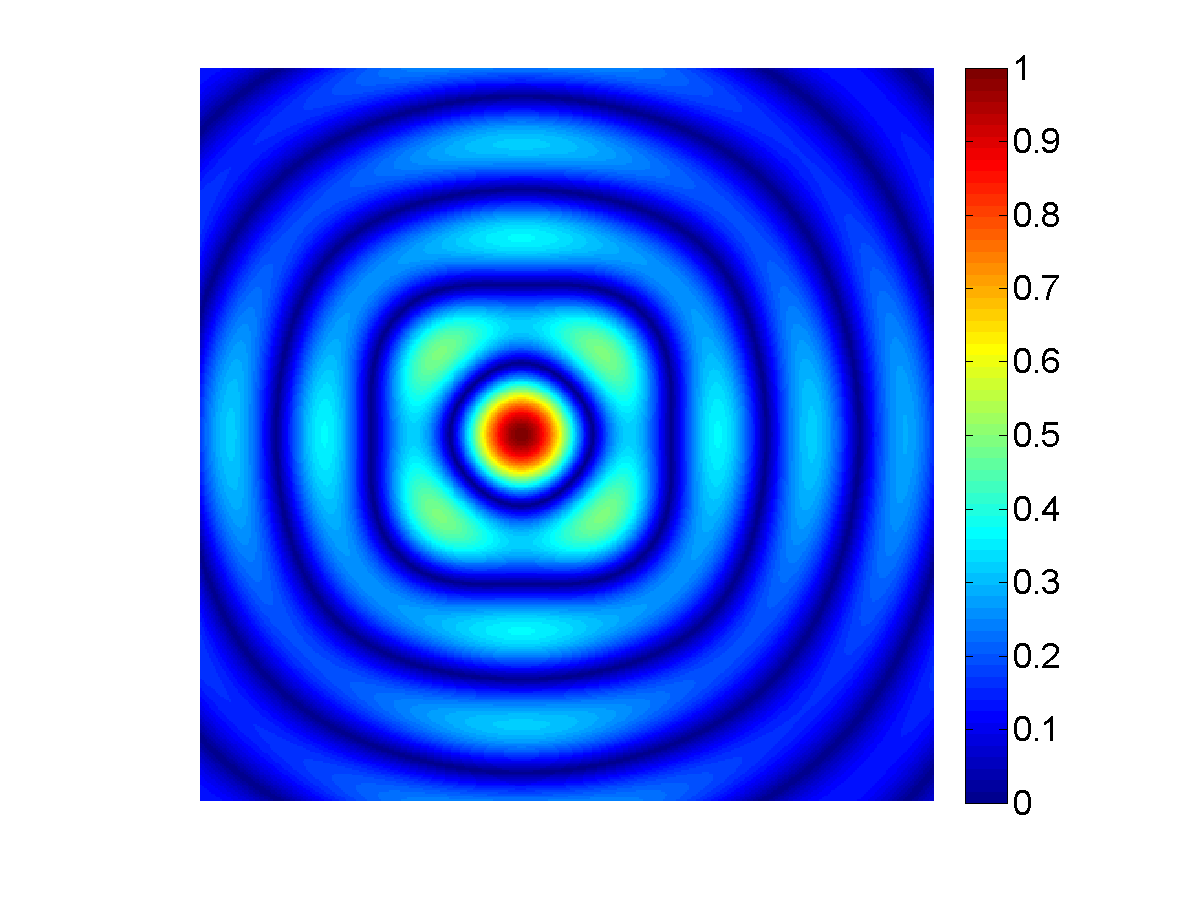}\\
     $|\langle\Phi_{12}(x,x_p),\Phi_{12}(x,x_q)\rangle_{L^2(\Gamma)}|$ & diagonal sum
  \end{tabular}
\caption{Cross products of $\Phi_{11}$, $\Phi_{22}$, $\Phi_{12}$, with the point scatterer
located at the $x_q=(-\frac{1}{4},0)$, and the diagonal sum, over the sampling domain
$\widetilde{\Omega}=[-2,2]^2$. The quantities are normalized.} \label{fig:fds1}
\end{figure}

\paragraph{Index function for multiple incident fields.} To arrive at an applicable effective
index function for multiple incidents, we
examine the behavior of the two-dimensional fundamental
solution $\Phi(x,y)$ more closely. We first consider the incident direction
$d_1=\frac{\sqrt{2}}{2}(1,1)$ and the polarization $p_1=\frac{\sqrt{2}}{2}(1,-1)$, and
assume that the local induced electrical current $J$ is proportional to
the polarized incident wave $E^{i}$ within the inhomogeneity located at $x_q$. Upon
ignoring the normalization, the
index function $\Psi(x_p;q)$ with the choice $q=p_1$ is roughly determined by
\begin{equation*}
   \begin{aligned}
        \langle \Phi(x,x_q)p_1,\Phi(x,x_p)p_1\rangle_{L^2(\Gamma)}  = & \langle
        \Phi_{11}(x,x_q)-\Phi_{12}(x,x_q),\Phi_{11}(x,x_p)-\Phi_{12}(x,x_p)\rangle_{L^2(\Gamma)}\\
             &    +\langle\Phi_{12}(x,x_q)-\Phi_{22}(x,x_q),\Phi_{12}(x,x_p)-\Phi_{22}(x,x_p)\rangle_{L^2(\Gamma)},
   \end{aligned}
\end{equation*}
where each term in the inner products, like $\Phi_{11}(x,x_p)-\Phi_{12}(x,x_q)$ and so on,
corresponds to one component of the vector field $E^s$.
Similarly, for the incident direction $d_2=\frac{\sqrt{2}}{2}(1,-1)$ and
the polarization $p_2=\frac{\sqrt{2}}{2}(1,1)$, the index function $\Phi(x_p;q)$
with the choice $q=p_2$ is roughly determined by
\begin{equation*}
  \begin{aligned}
    \langle\Phi(x,x_q)p_2,\Phi(x,x_p)p_2\rangle_{L^2(\Gamma)}=&
    \langle \Phi_{11}(x,x_q)+\Phi_{12}(x,x_q),\Phi_{11}(x,x_p)+\Phi_{12}(x,x_p)\rangle_{L^2(\Gamma)}\\
          &+\langle\Phi_{12}(x,x_q)+\Phi_{22}(x,x_q),\Phi_{12}(x,x_p)+\Phi_{22}(x,x_p)\rangle_{L^2(\Gamma)}.
  \end{aligned}
\end{equation*}
The quantities  $|\langle\Phi(x,x_q)p_1,\Phi(x,x_p)p_1\rangle_{L^2(\Gamma)}|$  and $|\langle
\Phi(x,x_q)p_2,\Phi(x,x_p)p_2\rangle_{L^2(\Gamma)}|$ after normalization are shown in Fig.\,\ref{fig:fds2}: they
exhibit  strong resonances. This is attributed to the ripples
predicted from the relation \eqref{Id} and mixing of different components of
the fundamental solution $\Phi$. Note that the resonances show up at different
regions, with their locations depending on the incident direction/polarization. We reiterate that
in case of one single point scatterer, the resonance
is completely predictable and one might be able to remove them
from the reconstructed images with suitable postprocessing. However, in the presence of multiple scatterers,
the discrimination of the resonances from the true scatterers can be highly nontrivial.
Therefore, it is highly desirable  to
design a sampling method that is free from significant resonances ab initio, by suitably adapting
the choice of $q$ in the probing function $\Phi(x,x_p)q$ and combining the indices
for multiple polarizations. Let us consider
the two polarizations $p_1=\frac{\sqrt{2}}{2}(1,-1)$ and $p_2=
\frac{\sqrt{2}}{2}(1,1)$, and measure all components of
the scattered electric field $E^s$, in the hope of focusing on
the genuine inhomogeneities. Then the performance of
the combined index with $p_1=\frac{\sqrt{2}}{2}(1,-1)$ and
$p_2=\frac{\sqrt{2}}{2}(1,1)$ is approximately given by
\begin{eqnarray*}
&&\langle\Phi(x,x_q)p_1,\Phi(x,x_p)p_1\rangle_{L^2(\Gamma)} +
        \langle\Phi(x,x_q)p_2,\Phi(x,x_p)p_2\rangle_{L^2(\Gamma)}\\
      &=& \langle \Phi_{11}(x,x_q),\Phi_{11}(x,x_p)\rangle_{L^2(\Gamma)}+2\langle\Phi_{12}(x,x_q),
             \Phi_{12}(x,x_p)\rangle_{L^2(\Gamma)}
             +\langle\Phi_{22}(x,x_q),\Phi_{22}(x,x_p)\rangle_{L^2(\Gamma)}\,.
\end{eqnarray*}
Even though the indices for the polarizations $p_1$ and $p_2$ separately
both have their own shadows of the point scatterer at $x_q=(-\frac{1}{4},0)$
(see Fig.\,\ref{fig:fds2}(a)-(b)), the combination
can isolate the  point scatterer very well (see Fig.\,\ref{fig:fds2}(c)).
Hence, the combination of multiple polarizations can remedy the
undesirable rippling phenomena. This suggests us the following combined
index $\Psi_c(x_p)$ for the multiple polarizations $\{p_l\}_{l=1}^L$\,:
\begin{equation*}
  \Psi_c(x_p) = \frac{1}{L}\sum_{l=1}^L\Psi(x_p;p_l)
\end{equation*}
where $\Psi(x_p;p_l)$ refers to the index function for the polarization $p_l$ (hence the $l$th incident field).

\begin{figure}[h] \centering
\begin{tabular}{ccc}
  \includegraphics[trim = 2cm 1cm 2cm 1cm, clip=true,width=5cm]{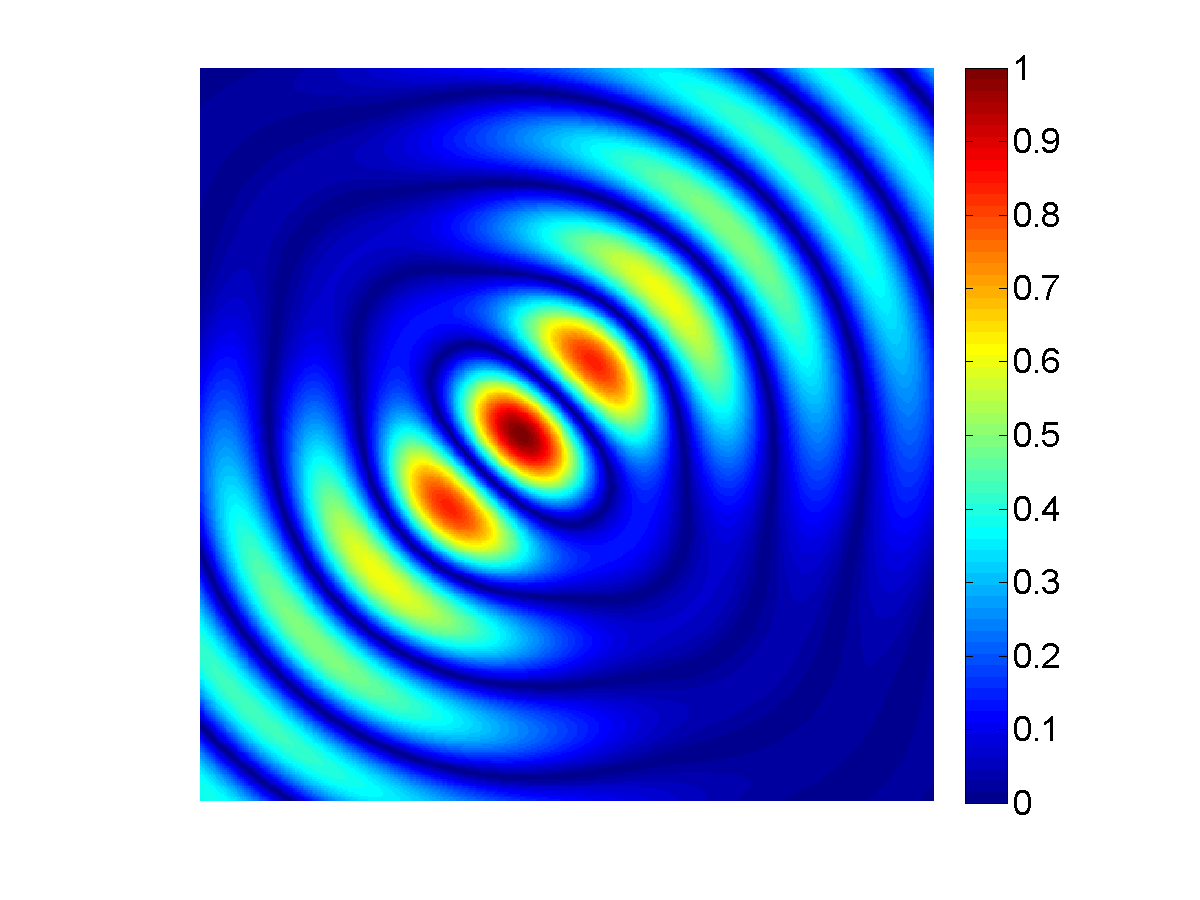} & \includegraphics[trim = 2cm 1cm 2cm 1cm, clip=true,width=5cm]{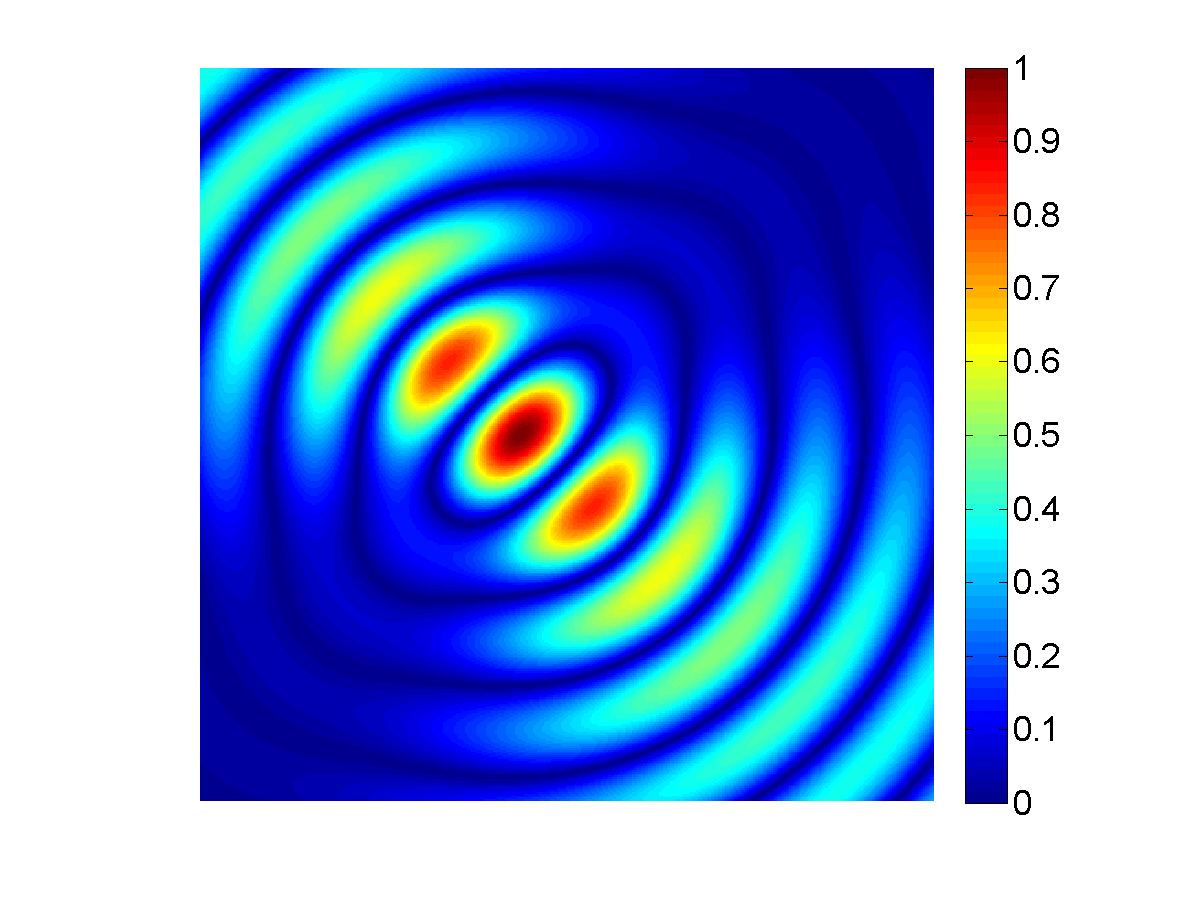} & \includegraphics[trim = 2cm 1cm 2cm 1cm, clip=true,width=5cm]{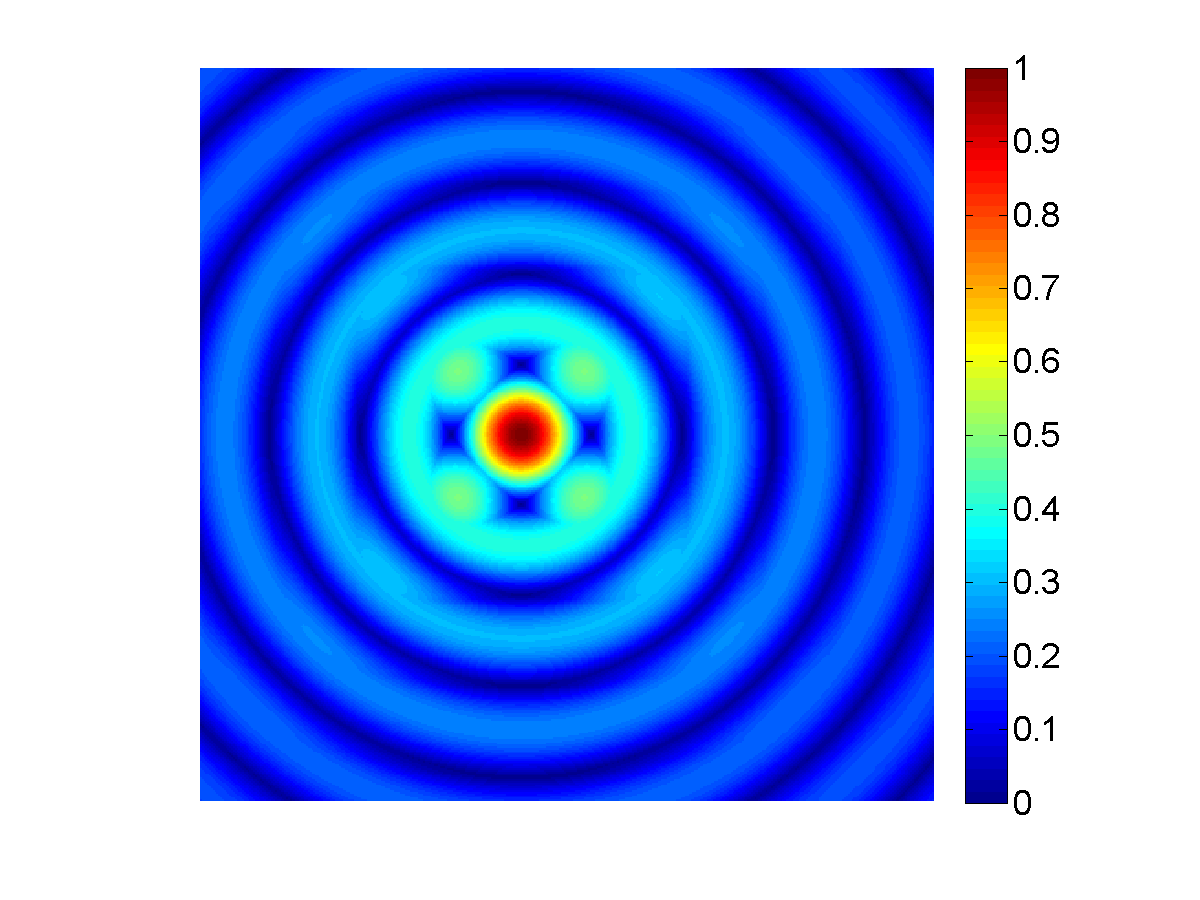}\\
   (a) $|\langle\Phi(x,x_p)p_1,\Phi(x,x_q)p_1\rangle_{L^2(\Gamma)}|$ &
   (b) $|\langle\Phi(x,x_p)p_2,\Phi(x,x_q)p_2\rangle_{L^2(\Gamma)}|$ &
   (c) combination
\end{tabular}
\caption{The cross products for polarizations $p_1$, $p_2$ and the combination with direct
sum over the sampling domain $\widetilde{\Omega}=[-2,2]^2$. The point scatterer
is located at $x_q=(-\frac{1}{4},0)$. The quantities are normalized.} \label{fig:fds2}
\end{figure}

\section{Numerical experiments and discussions}

In this section, we present two- and three-dimensional numerical examples to showcase the features,
e.g., the accuracy and robustness, of the proposed DSM for detecting
scatterers. The examples are designed to illustrate the features of the method and
to validate the theoretical findings in Section 3. Numerical results for both exact and noisy data will be presented.
In all examples, the wavelength $\lambda$ is set to $1$, and the wavenumber $k$ is $2\pi$.
The noisy scattered electric data $E^s_\delta$ is generated pointwise by
\begin{equation*}
  E_\delta^s(x) = E^s(x) + \varepsilon\max_{x\in\Gamma}|E^s(x)|\zeta(x),
\end{equation*}
where $\varepsilon$ denotes the relative noise level, and both real and imaginary parts
of the random variable $\zeta(x)$ follow the standard Gaussian distribution with zero
mean and unit variance. All the computations were performed on a dual-core laptop
using \texttt{MATLAB} R2009a.

\subsection{Two-dimensional examples}

Now we present numerical results for two-dimensional examples. Unless otherwise specified,
two incident fields, i.e., $d_1=\frac{\sqrt{2}}{2}(1,1)^\mathrm{t}$
and $d_2=\frac{\sqrt{2}}{2}(-1,1)^\mathrm{t}$ (accordingly, the polarizations
$p_1=\frac{\sqrt{2}}{2}(1,-1)^\mathrm{t}$ and $p_2=\frac{\sqrt{2}}{2}(1,1)^\mathrm{t}$),
are employed. For each incident field $E^i$, the
scattered electric field $E^s$ is measured at $30$ points uniformly distributed
on a circle of radius $5$ centered at the origin. The sampling domain $\widetilde\Omega$
is fixed at $[-2,2]^2$, which is divided into small squares of equal
side length $h=0.01$. The index function $\Psi$ as an estimate to the scatterer
support will be displayed for each example.

Our first example shows the method for one single scatterer, which
confirms the theoretical analysis of the index $\Psi_c$ in Section 3.2.
\begin{exam}\label{exam:1sc}
The example considers one square scatterer of side length $0.3$ centered at the
point $(-\frac{1}{4},0)$. The true inhomogeneity coefficient $\eta(x)$ of the scatterer is
set to be $1$, namely $n(x)=\sqrt{2}$.
\end{exam}

\begin{figure}[h]
 \centering
  \begin{tabular}{cccc}
    \includegraphics[trim = 3.5cm 1cm 2cm 1cm, clip=true,width=3.7cm]{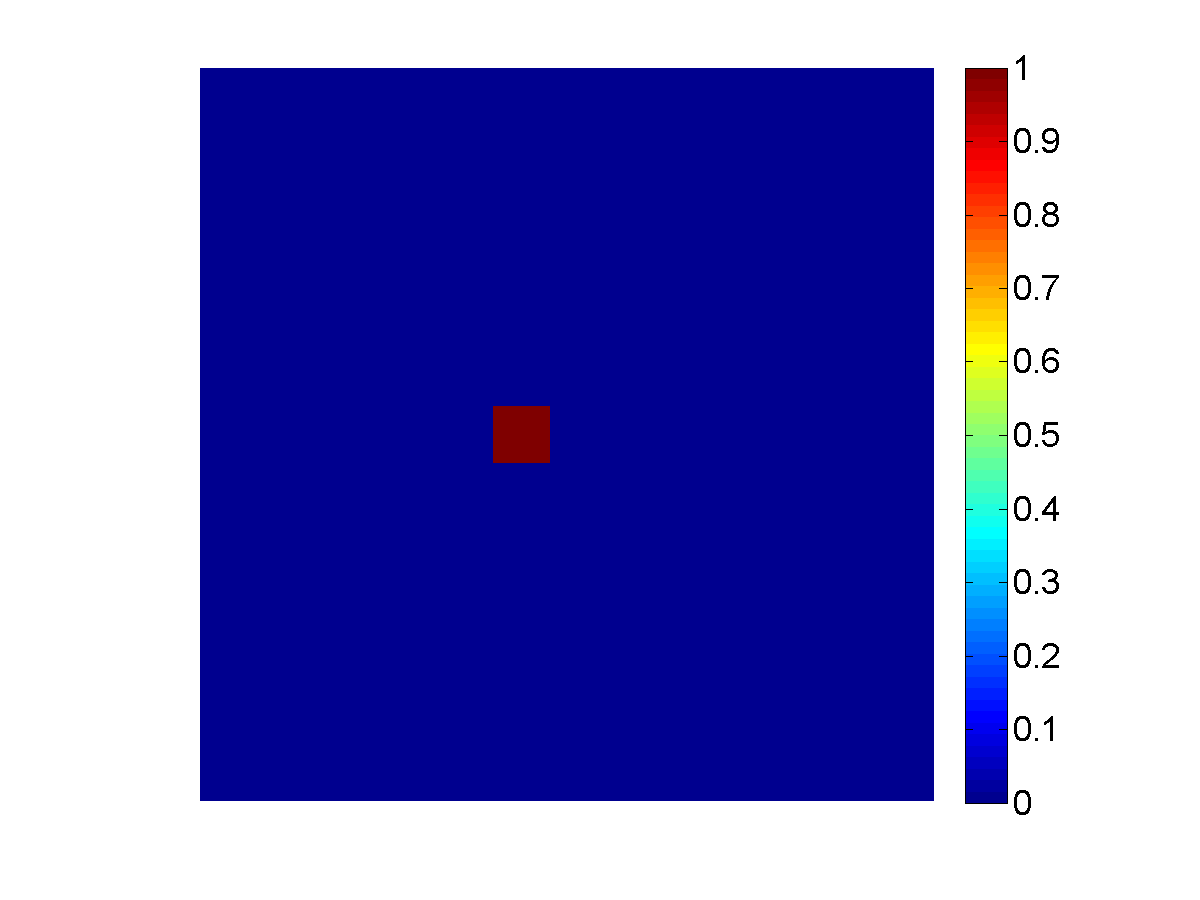} & \includegraphics[trim = 3.5cm 1cm 2cm 1cm, clip=true,width=3.7cm]{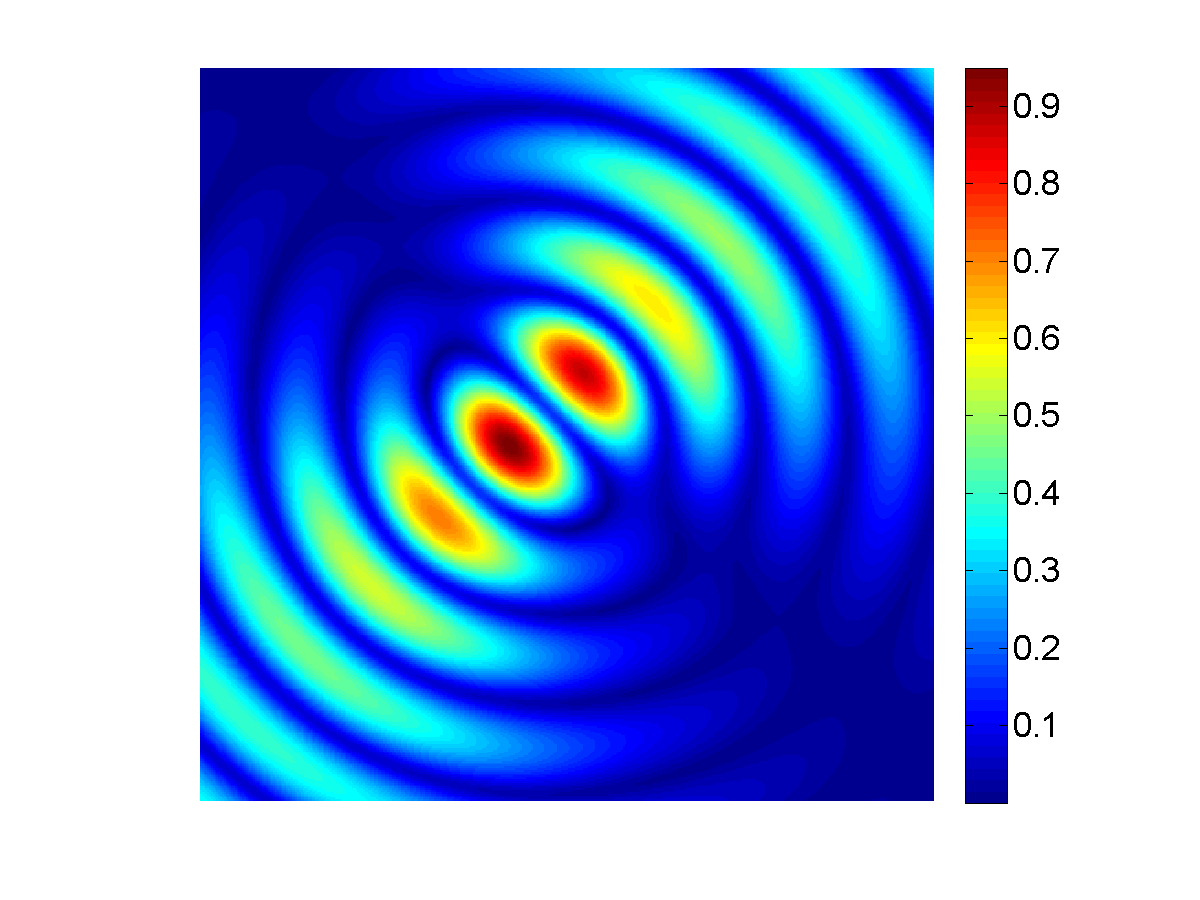}
    & \includegraphics[trim = 3.5cm 1cm 2cm 1cm, clip=true,width=3.7cm]{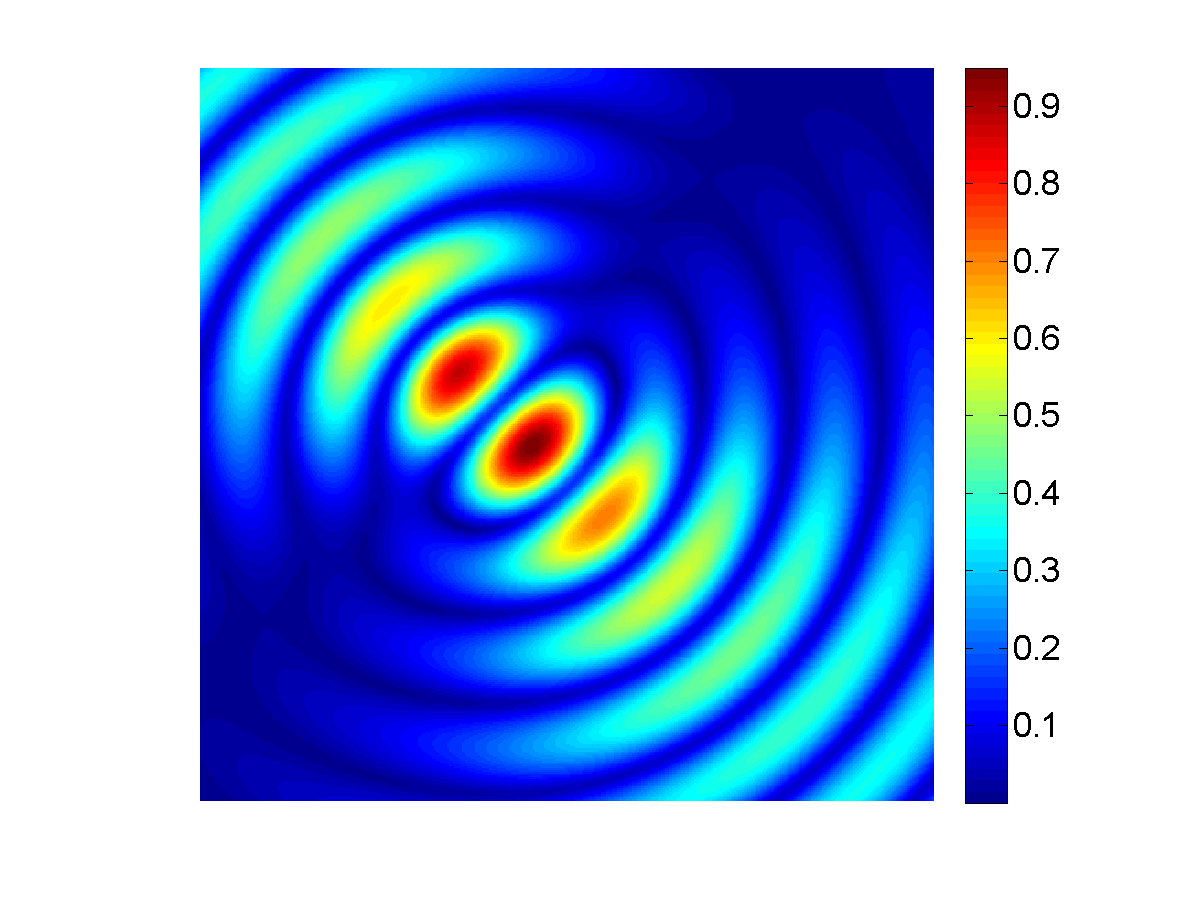} & \includegraphics[trim = 3.5cm 1cm 2cm 1cm, clip=true,width=3.7cm]{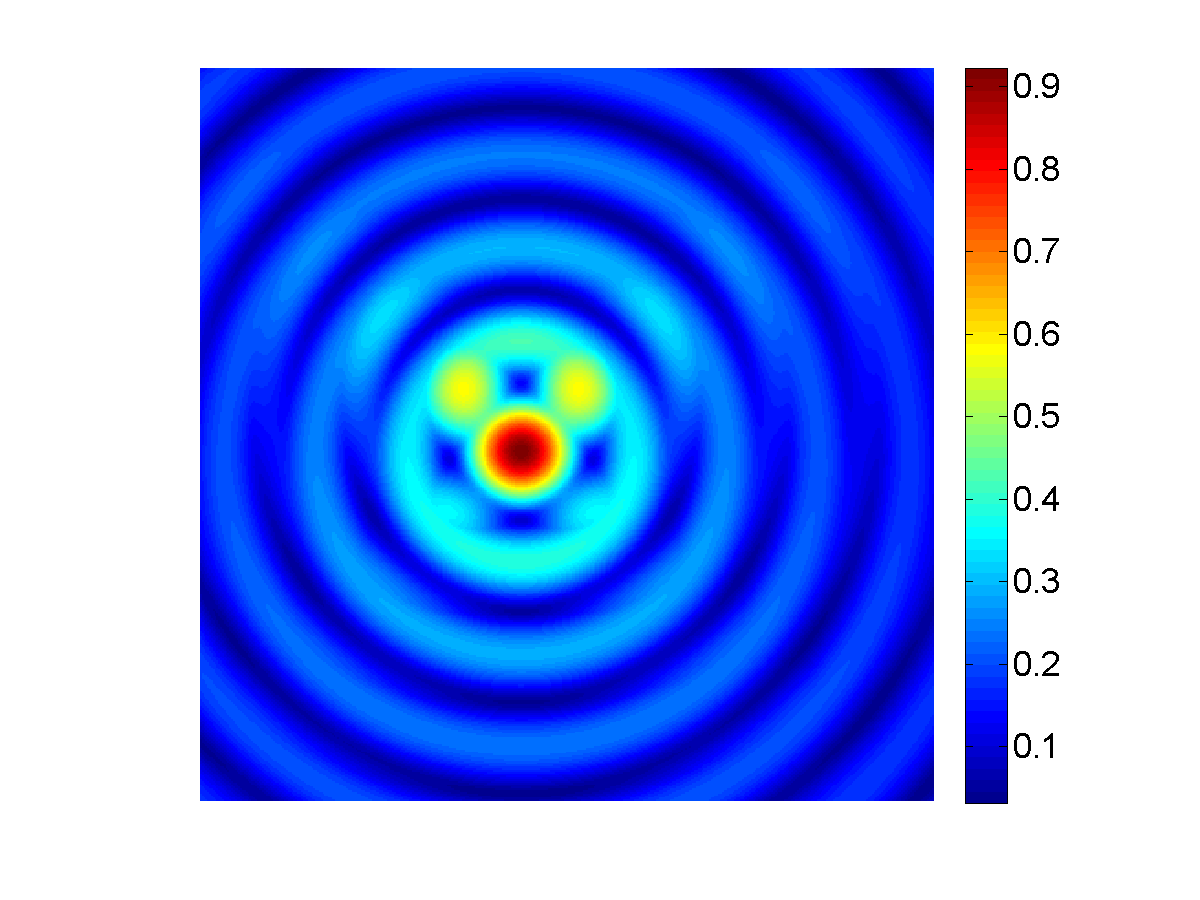}\\
    \includegraphics[trim = 3.5cm 1cm 2cm 1cm, clip=true,width=3.7cm]{ex1_rec_true.png} & \includegraphics[trim = 3.5cm 1cm 2cm 1cm, clip=true,width=3.7cm]{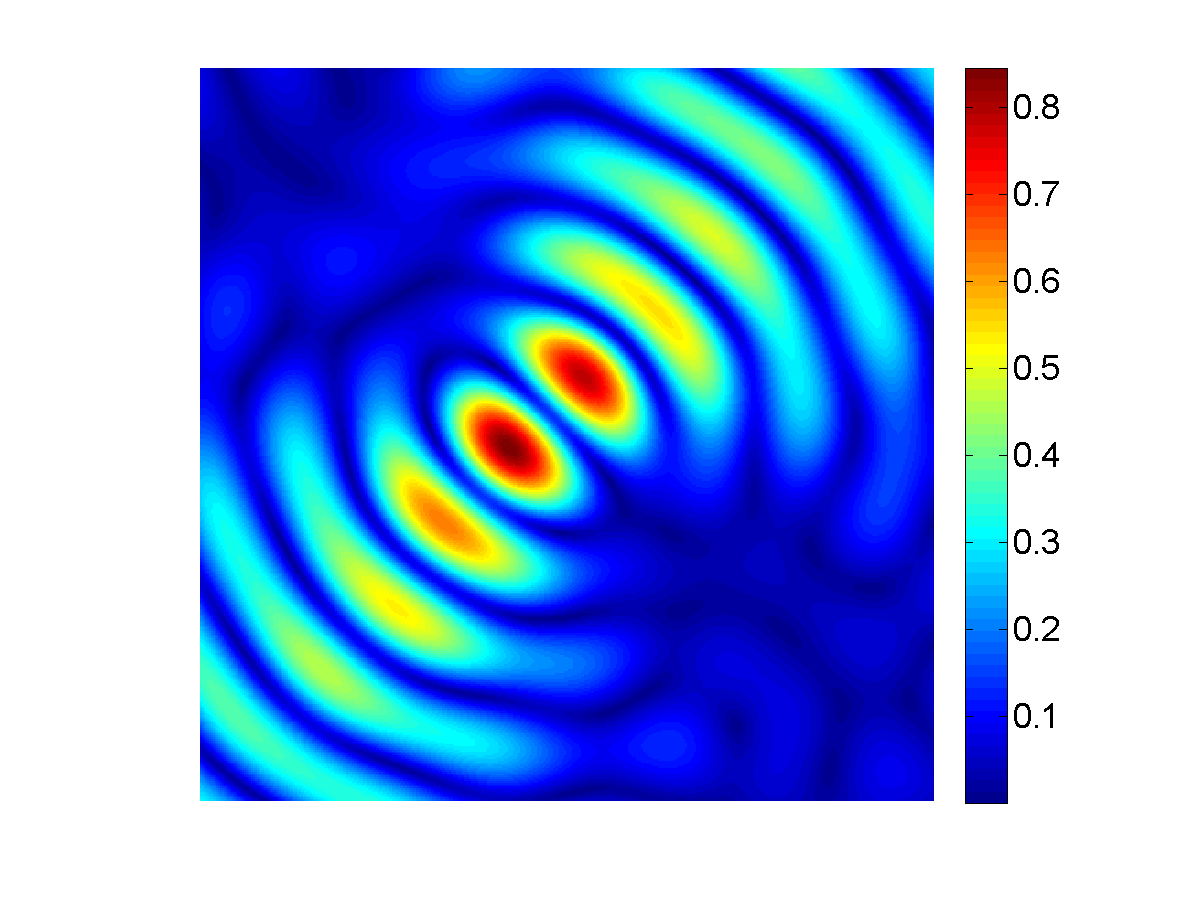}
    & \includegraphics[trim = 3.5cm 1cm 2cm 1cm, clip=true,width=3.7cm]{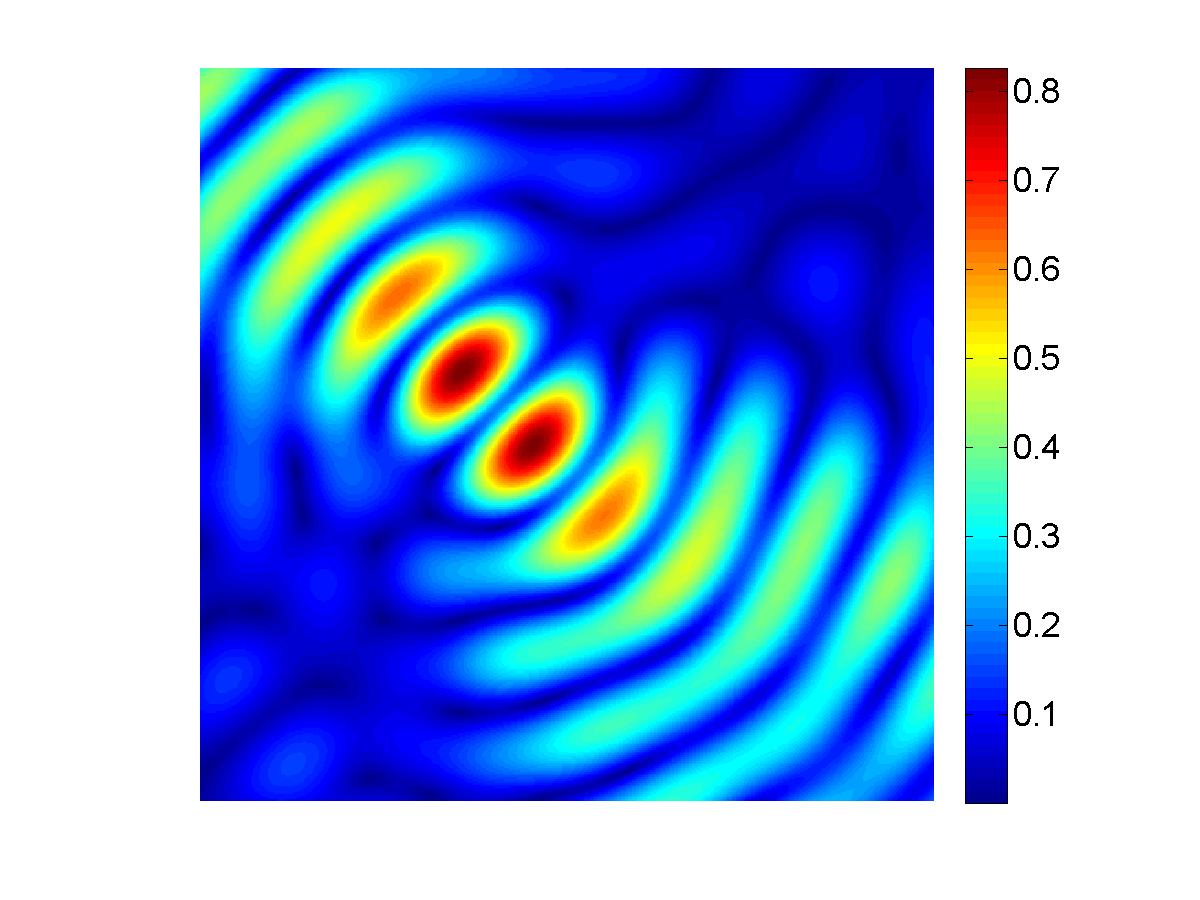} & \includegraphics[trim = 3.5cm 1cm 2cm 1cm, clip=true,width=3.7cm]{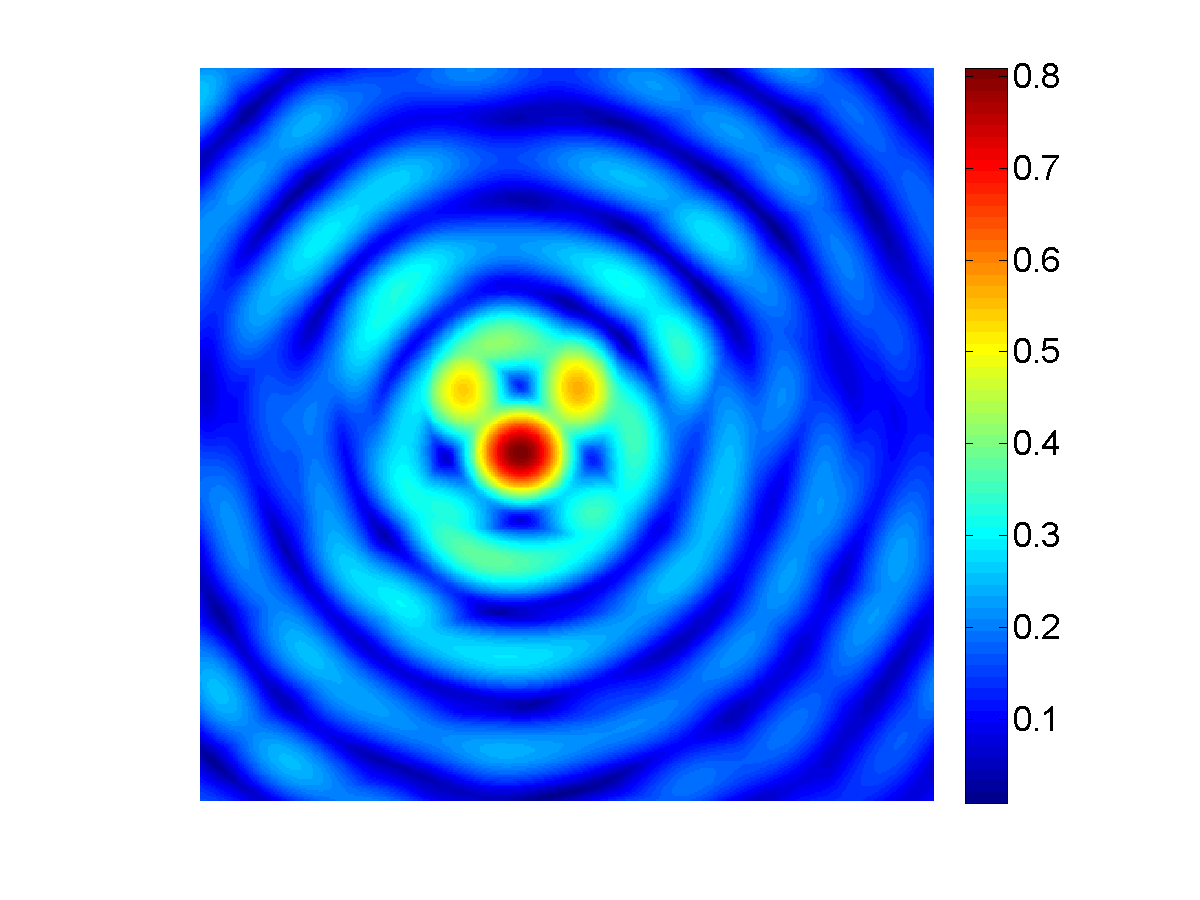}\\
    (a) true scatterer & (b) index $\Psi(x_p;p_1)$ & (c) index $\Psi(x_p;p_2)$ & (d) index $\Psi_c$
  \end{tabular}
  \caption{Numerical results for Example \ref{exam:1sc}. The first and
  second row refers to index functions for the exact data and the noisy
  data with $\varepsilon=20\%$ noise.}\label{fig:1sc}
\end{figure}

The numerical results for Example \ref{exam:1sc} are shown in Fig.\,\ref{fig:1sc}. We observe that for sampling points
 close to the physical scatterer,
the index function $\Psi$ is relatively large; otherwise it takes relatively small values.
Note that the image for one incident field $E^i$ exhibits obvious resonances, with resonance
locations depending on the incident direction. Here, the resonance behavior agree excellently
with the theoretical analysis for one single point scatterer in Section 3.2,
hence it might be removed by applying a suitable postprocessing
procedure. Nonetheless, the use of two incident fields can greatly mitigate the resonances.
Therefore, the index function $\Psi_c$ does provide an
accurate and reliable indicator for the location of the scatterer; see Fig.\,\ref{fig:1sc}(d).
The presence of $\varepsilon=20\%$ noise
in the measured data does not affect the shape of the index function $\Psi_c$,
thereby showing the robustness of the method with respect to noise.

Our second example illustrates the method for two separate scatterers.
\begin{exam}\label{exam:2sc}
We consider two square scatterers, with the inhomogeneity coefficient $\eta$ being $1$
in both scatterers.
The following two cases are investigated:
\begin{itemize}
  \item[(a)] The two scatterers are of side length $0.2$, and located at $(-0.8,-0.7)$ and $(0.3,0.8)$, respectively.
  \item[(b)] The two scatterers are of side length $0.3$, and located at $(-0.45,-0.35)$ and $(0.05,0.15)$, respectively.
\end{itemize}
\end{exam}

\begin{figure}[h!]
  \centering
  \begin{tabular}{cccc}
    \includegraphics[trim = 3.5cm 1cm 2cm 1cm, clip=true,width=3.7cm]{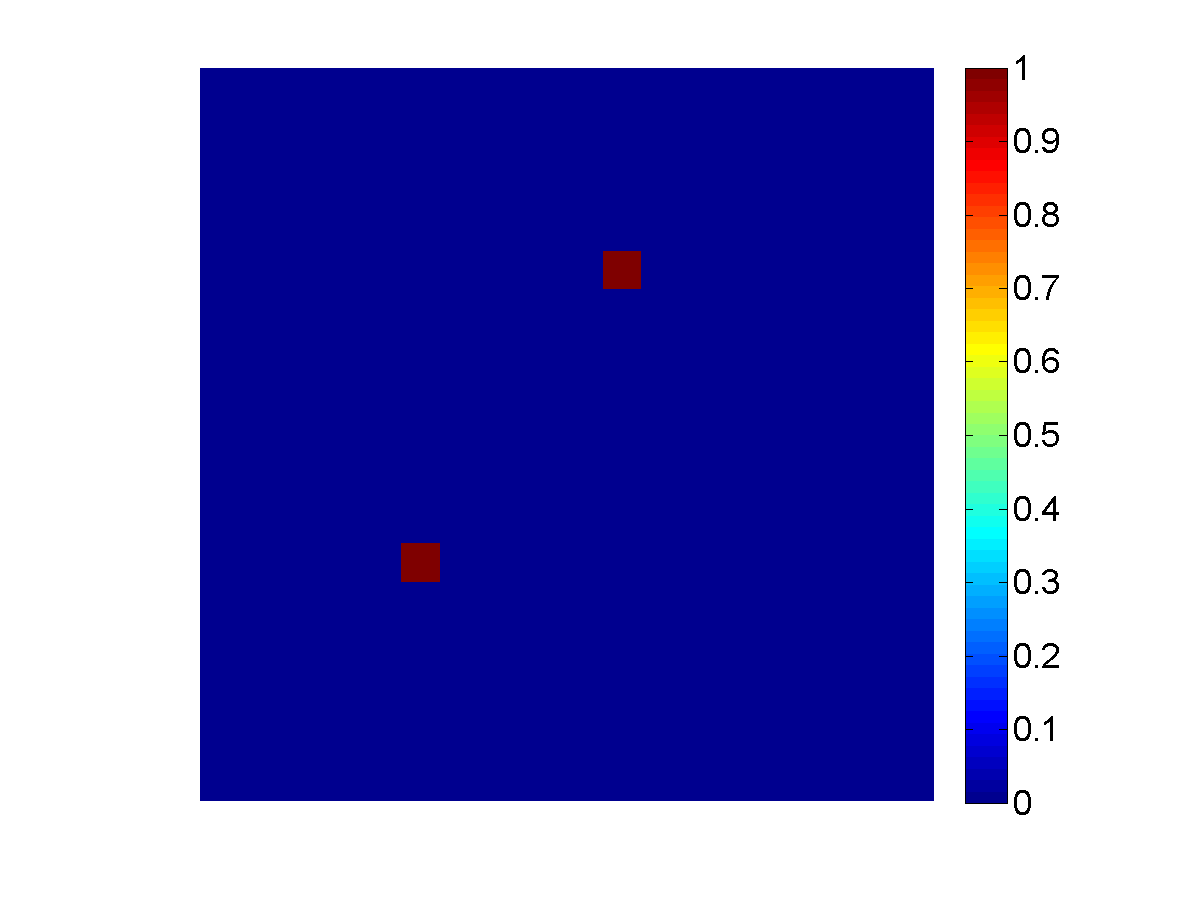} & \includegraphics[trim = 3.5cm 1cm 2cm 1cm, clip=true,width=3.7cm]{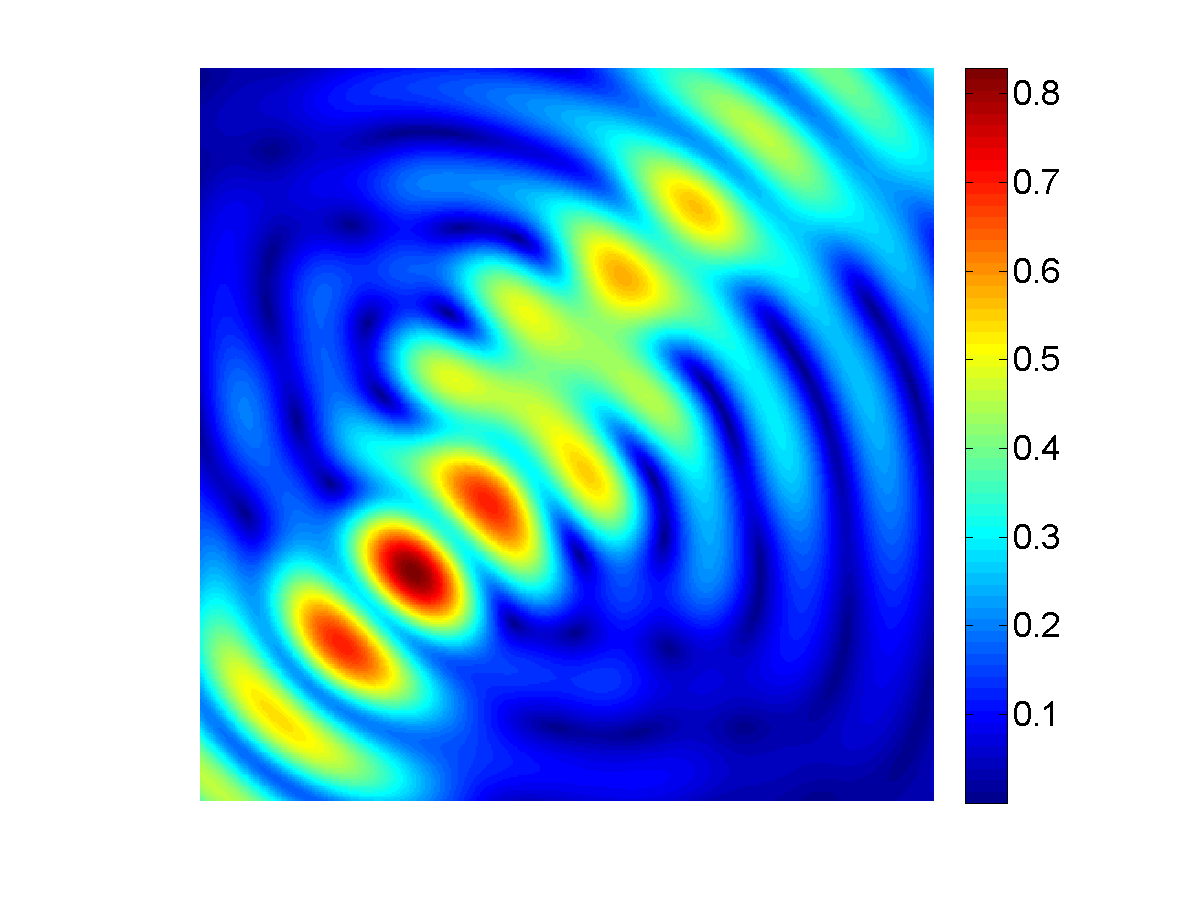}
    & \includegraphics[trim = 3.5cm 1cm 2cm 1cm, clip=true,width=3.7cm]{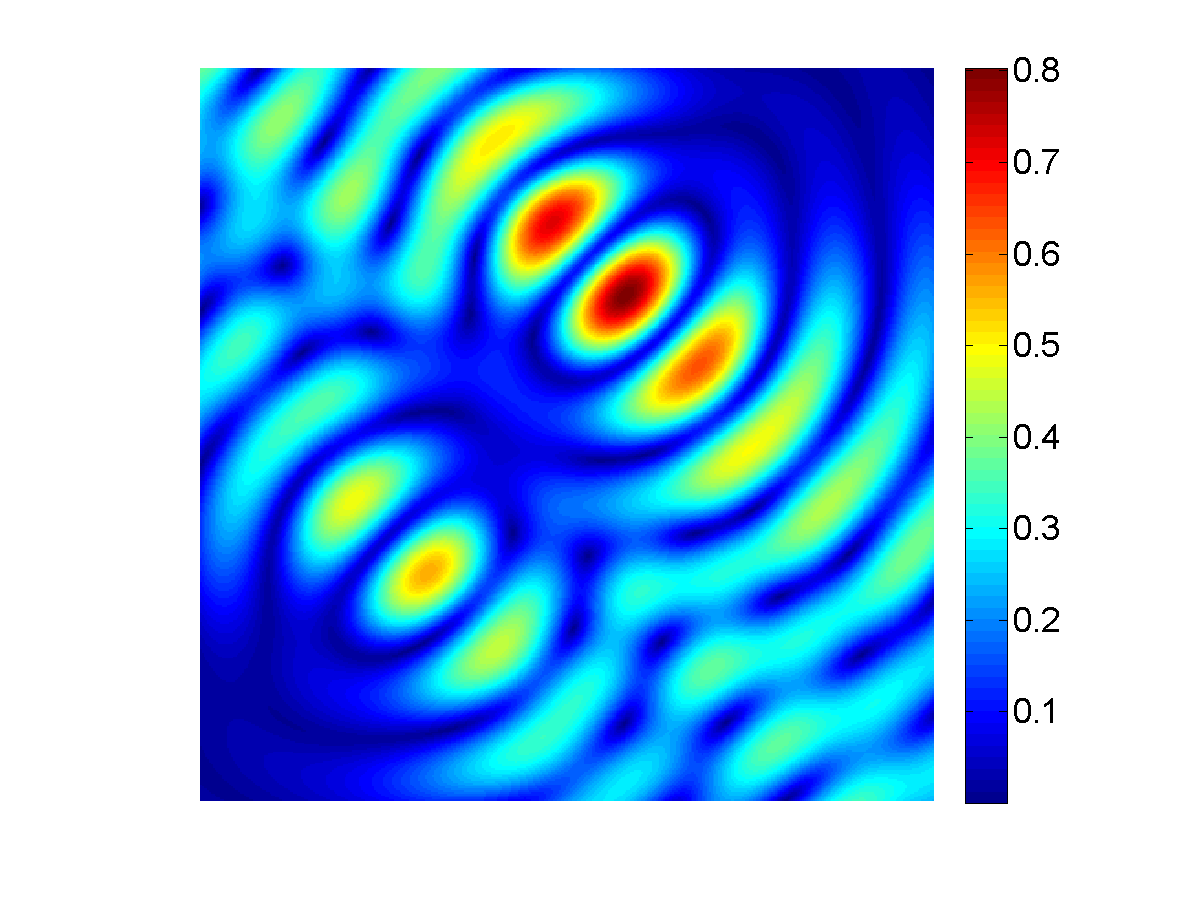} & \includegraphics[trim = 3.5cm 1cm 2cm 1cm, clip=true,width=3.7cm]{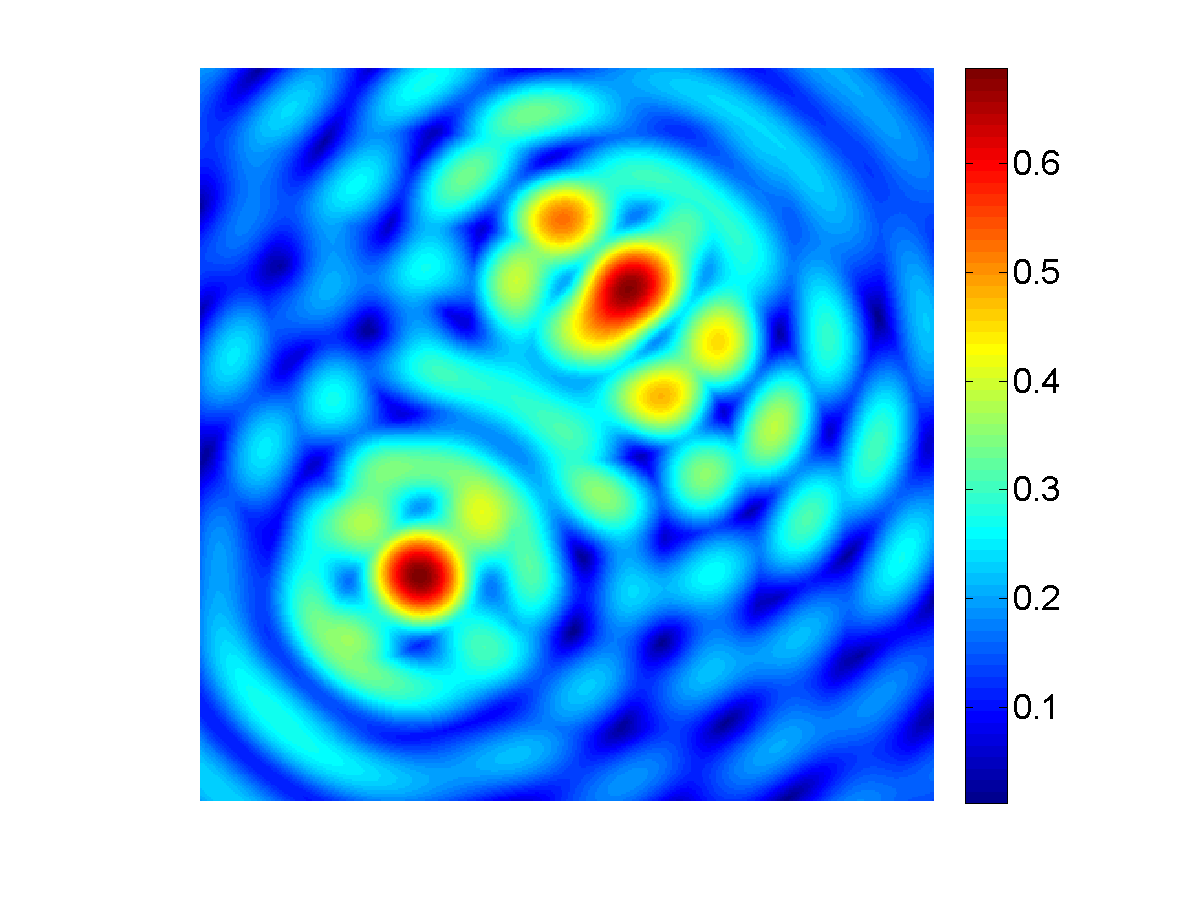}\\
    \includegraphics[trim = 3.5cm 1cm 2cm 1cm, clip=true,width=3.7cm]{ex2_rec_true.png} & \includegraphics[trim = 3.5cm 1cm 2cm 1cm, clip=true,width=3.7cm]{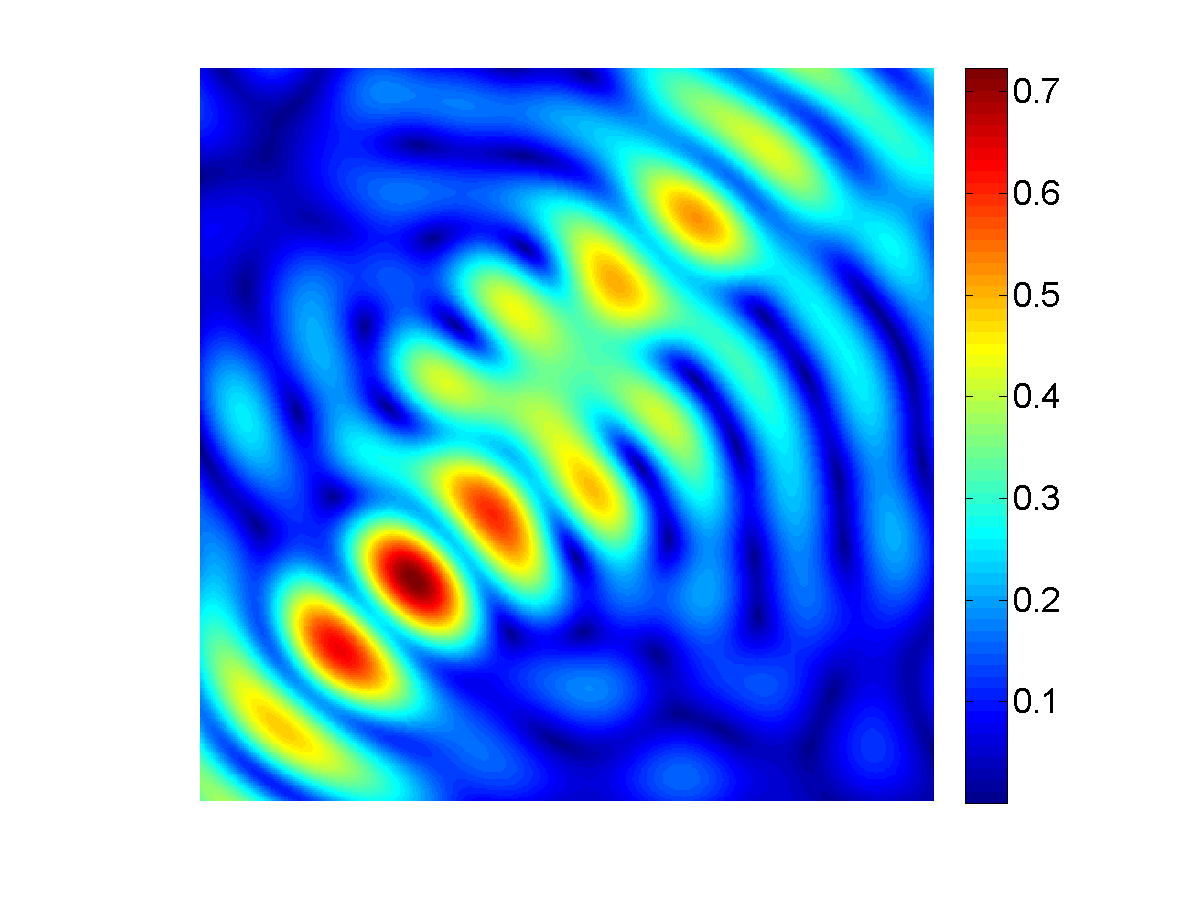}
    & \includegraphics[trim = 3.5cm 1cm 2cm 1cm, clip=true,width=3.7cm]{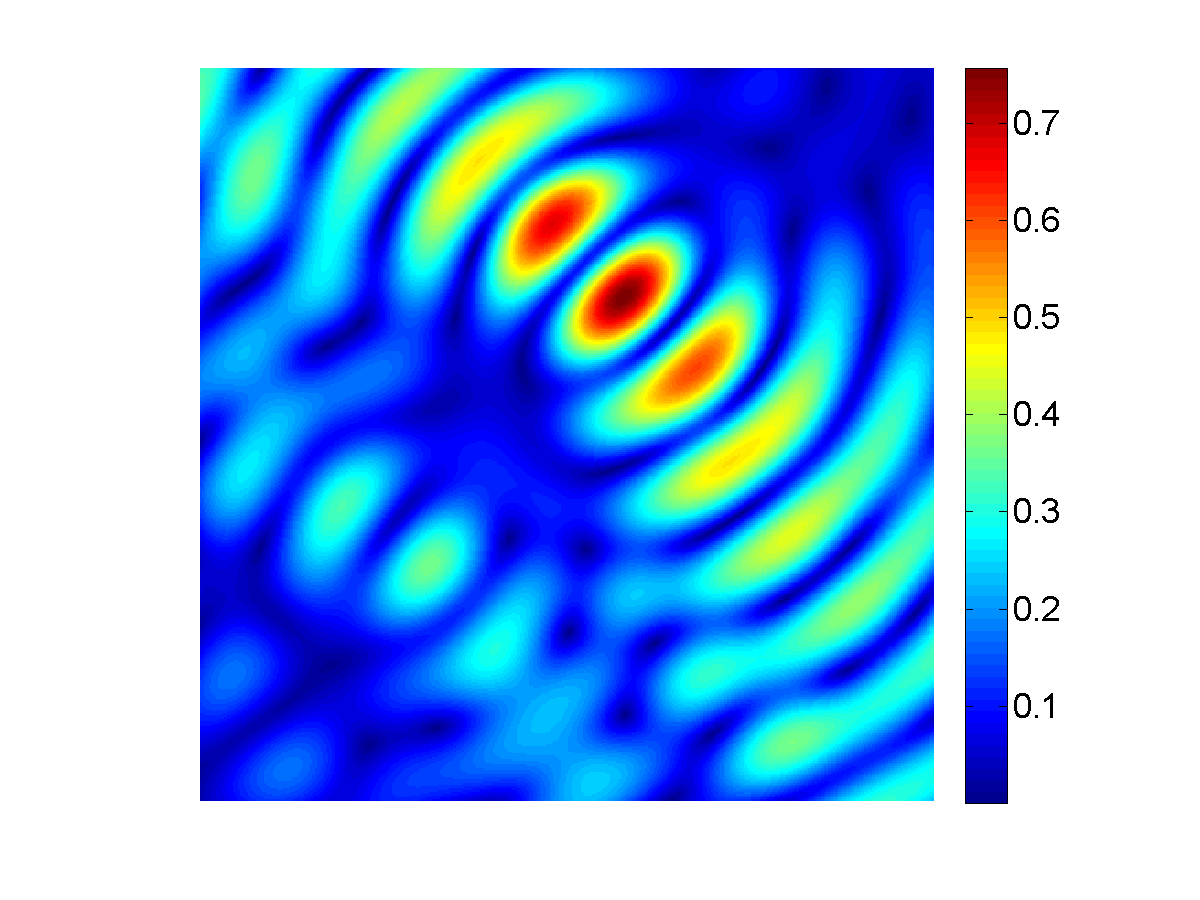} & \includegraphics[trim = 3.5cm 1cm 2cm 1cm, clip=true,width=3.7cm]{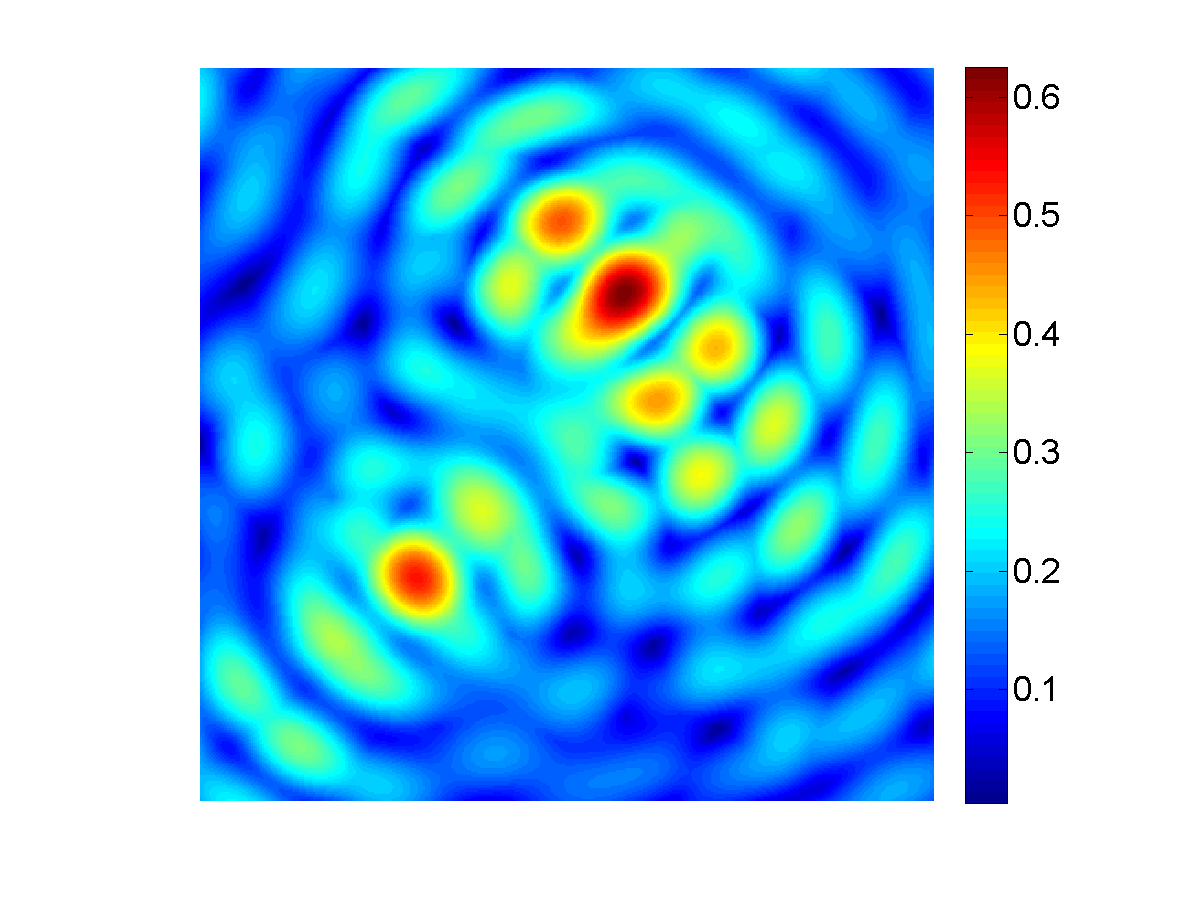}\\
        (a) true scatterers & (b) index $\Psi(x_p;p_1)$ & (c) index $\Psi(x_p;p_2)$ & (d) index $\Psi_c$
  \end{tabular}
  \caption{Numerical results for Example \ref{exam:2sc}(a). The first and
  second row refer to index functions for the exact data and the noisy
  data with $\epsilon=20\%$ noise.}\label{fig:2sca}
\end{figure}

The two scatterers in Example \ref{exam:2sc}(a) are well apart from each other.
In this example, each of the two incident fields tends to highlight only one of the
two scatterers, with the index value for one scatterer being much higher than for
the other; see Figs.\,\ref{fig:2sca}(b)-(c). Since the two scatterers
are well apart, the interactions between them is somehow weak, and the resonance pattern
for the point scatterer is well kept. However, the two incident
fields together give a clear discrimination of the two scatterers, with their
locations satisfactorily recovered for both exact data and the data with $\epsilon=20\%$ noise;
see Fig.\,\ref{fig:2sca}(d).
%


\begin{figure}[h!]
  \centering
  \begin{tabular}{cccc}
    \includegraphics[trim = 3.5cm 1cm 2cm 1cm, clip=true,width=3.7cm]{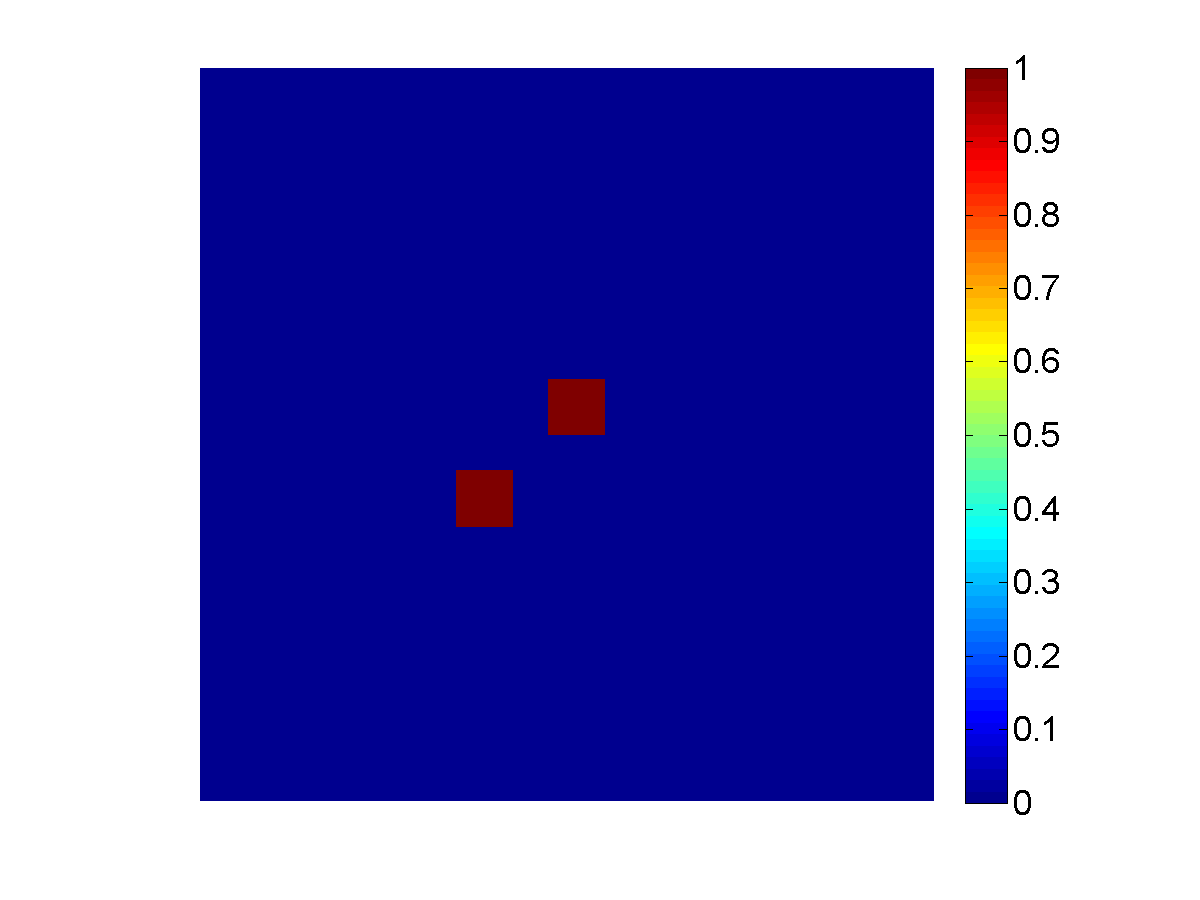} & \includegraphics[trim = 3.5cm 1cm 2cm 1cm, clip=true,width=3.7cm]{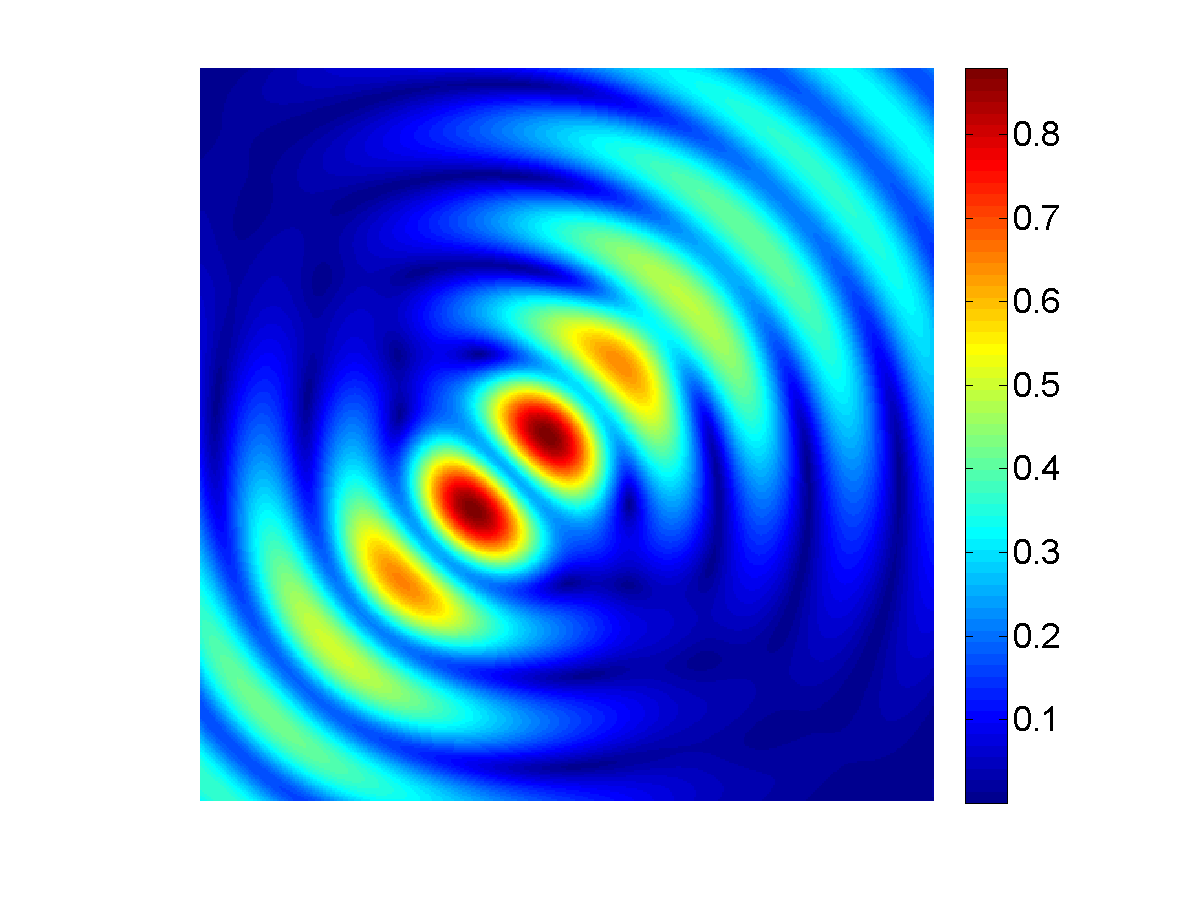}
    & \includegraphics[trim = 3.5cm 1cm 2cm 1cm, clip=true,width=3.7cm]{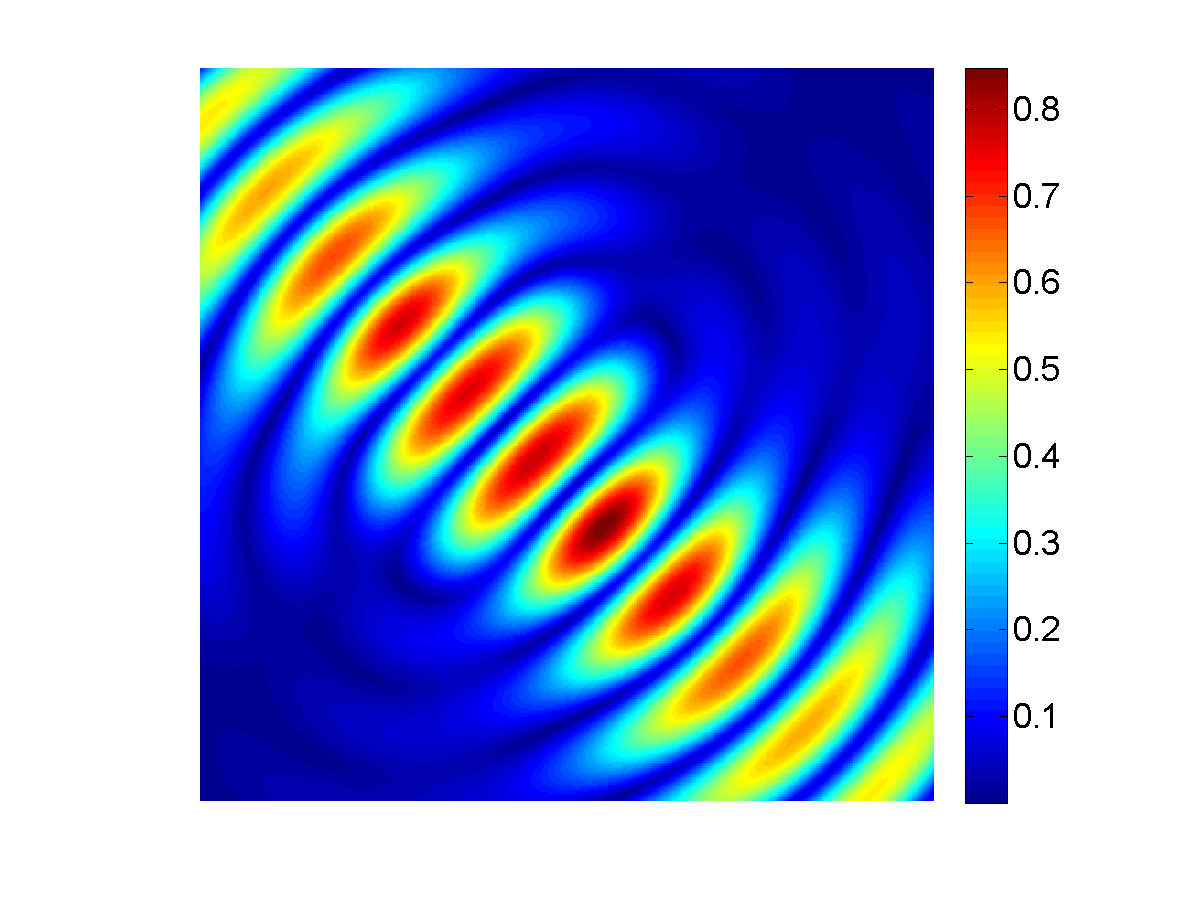} & \includegraphics[trim = 3.5cm 1cm 2cm 1cm, clip=true,width=3.7cm]{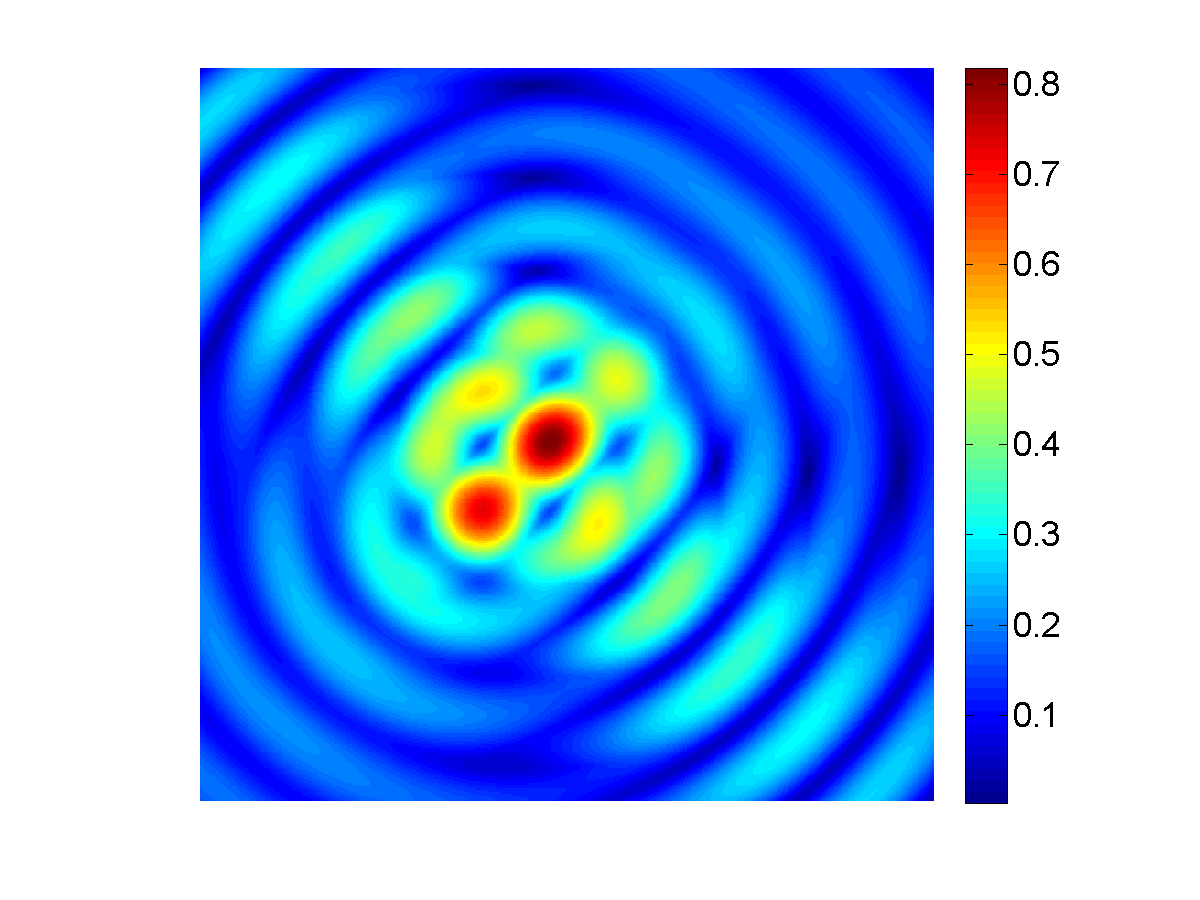}\\
    \includegraphics[trim = 3.5cm 1cm 2cm 1cm, clip=true,width=3.7cm]{ex3_rec_true.png} & \includegraphics[trim = 3.5cm 1cm 2cm 1cm, clip=true,width=3.7cm]{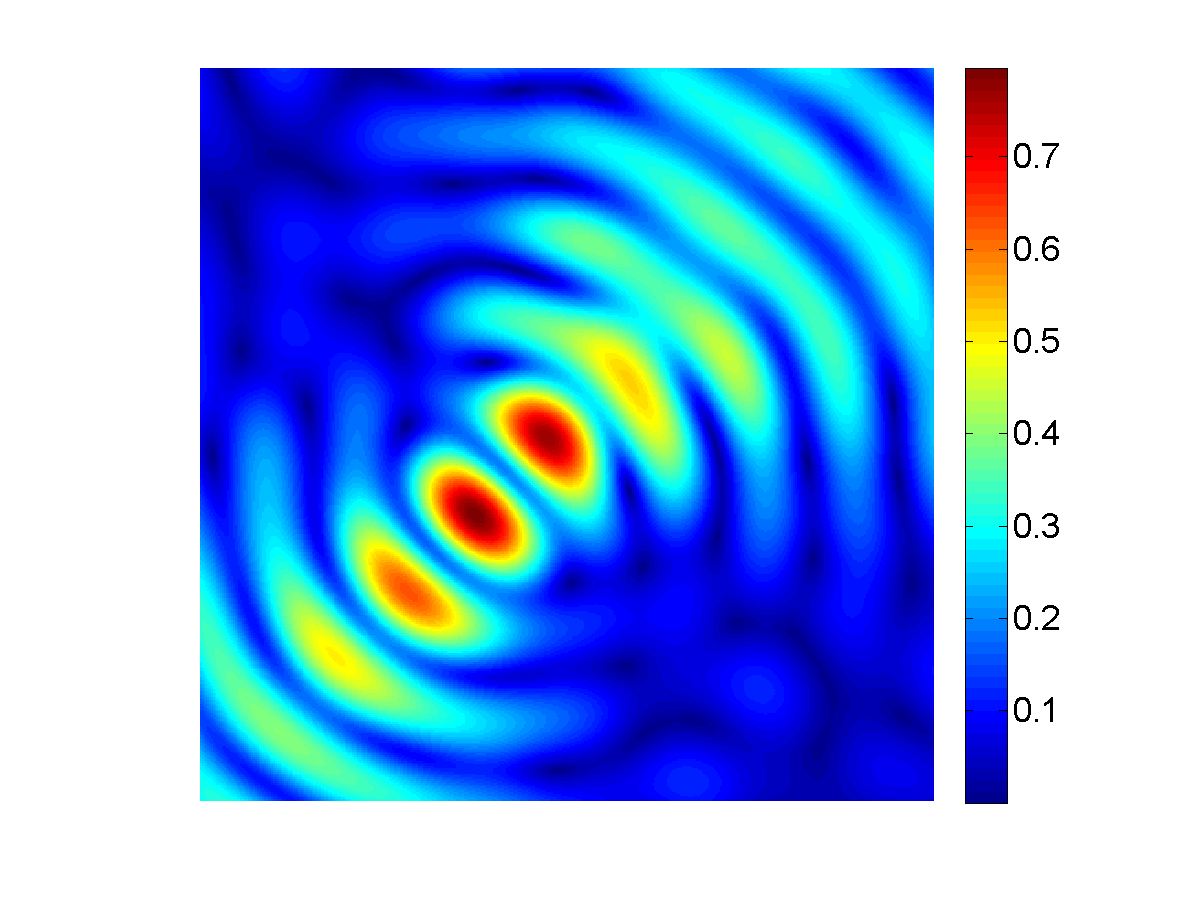}
    & \includegraphics[trim = 3.5cm 1cm 2cm 1cm, clip=true,width=3.7cm]{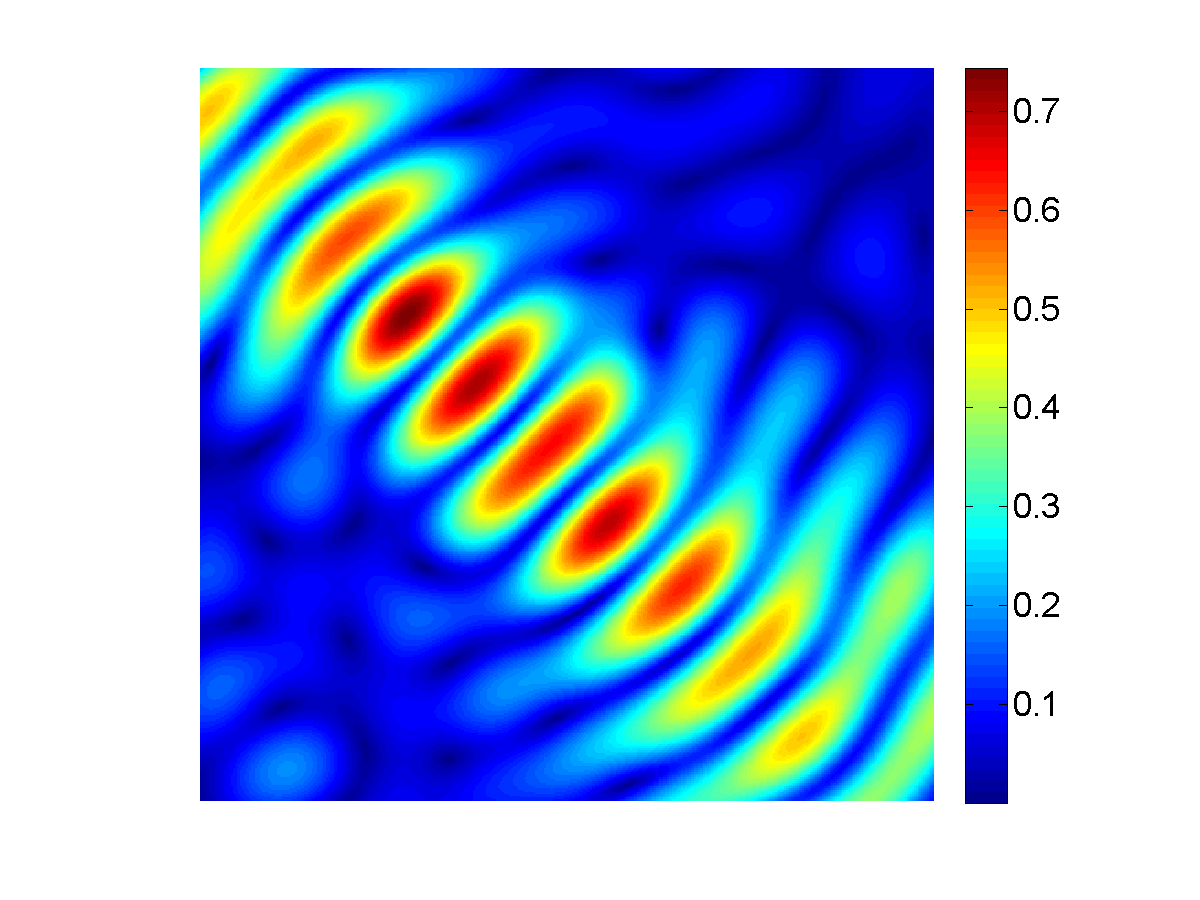} & \includegraphics[trim = 3.5cm 1cm 2cm 1cm, clip=true,width=3.7cm]{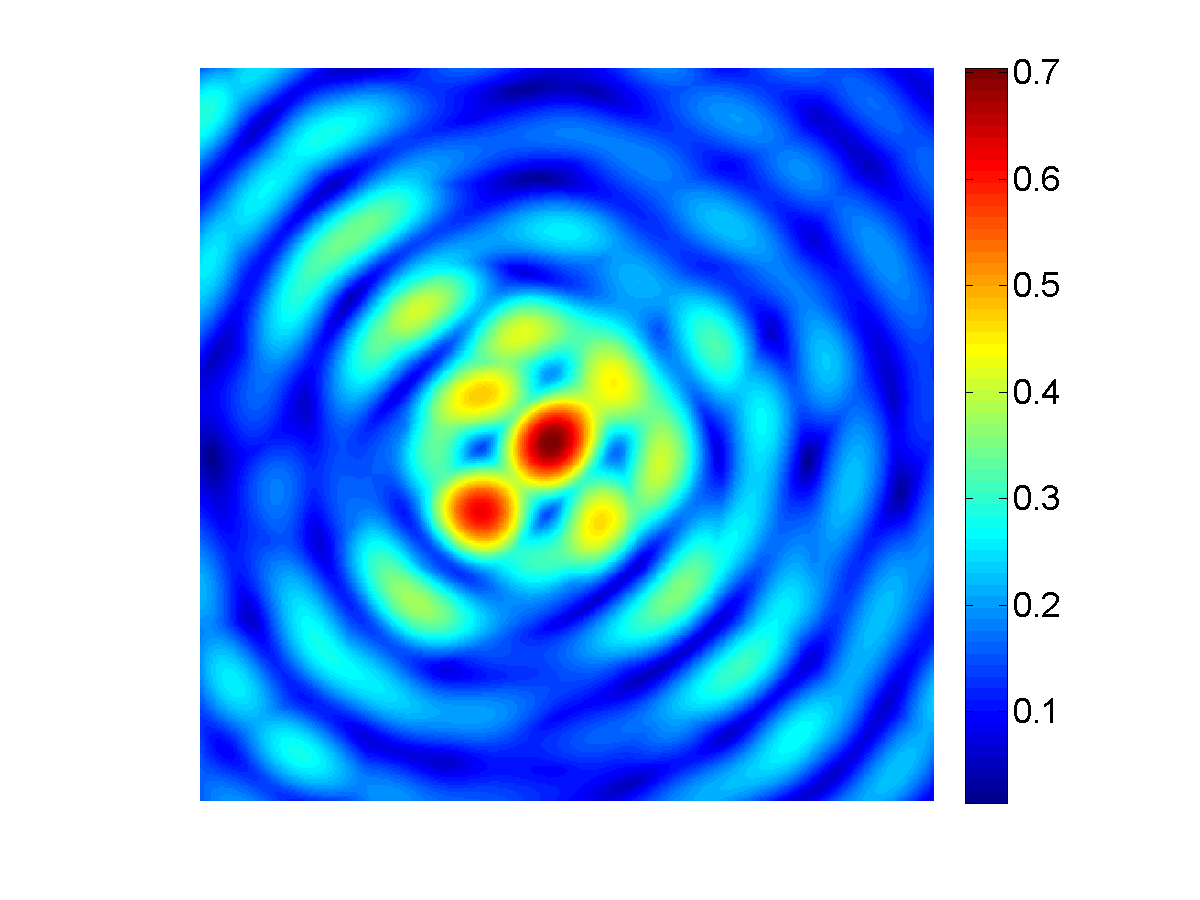}\\
    (a) true scatterers & (b) index $\Psi(x_p;p_1)$ & (c) index $\Psi(x_p;p_2)$ & (d) index $\Psi_c$
  \end{tabular}
  \caption{Numerical results for Example \ref{exam:2sc}(b). The first and
  second row refer to index functions for the exact data and the noisy
  data with $\epsilon=20\%$ noise.}\label{fig:2scb}
\end{figure}

The two scatterers in Example \ref{exam:2sc}(b) stay very close to each other.  We observe that,
for the incident direction $d_2$, apart from the strong resonances, the locations for the
scatterers cannot be directly inferred since the index function $\Psi$ relates the
two scatterers into an elongated ellipse shape.
Nonetheless, the resonances were almost completely removed from the estimate when using two incident fields.
Consequently, the estimate of the locations of the scatterers is very impressive:
the two scatterers are still well separated despite their closeness, with their locations correctly estimated,
for up to $\epsilon=20\%$ noise in the data.
Although not presented, we would like to remark that in the case of very high noise
levels, e.g., $\epsilon=40\%$, the estimate
tends to connect the two scatterers, and also some spurious modes have emerged.

Our next example considers the more challenging case of three neighboring scatterers.
\begin{exam}\label{exam:3sc}
This example consists of three neighboring square scatterers of width $0.15$: one
located at $(-\frac{5}{8},-\frac{5}{8})$, one located at $(-\frac{17}{40},-\frac{17}{40})$, and one located at $(-\frac{21}{40},\frac{1}{8})$. The inhomogeneity
coefficients of all three scatterers are set to be $\eta=1$.
\end{exam}

\begin{figure}[h]
  \centering
  \begin{tabular}{cccc}
    \includegraphics[trim = 3.5cm 1cm 2cm 1cm, clip=true,width=3.7cm]{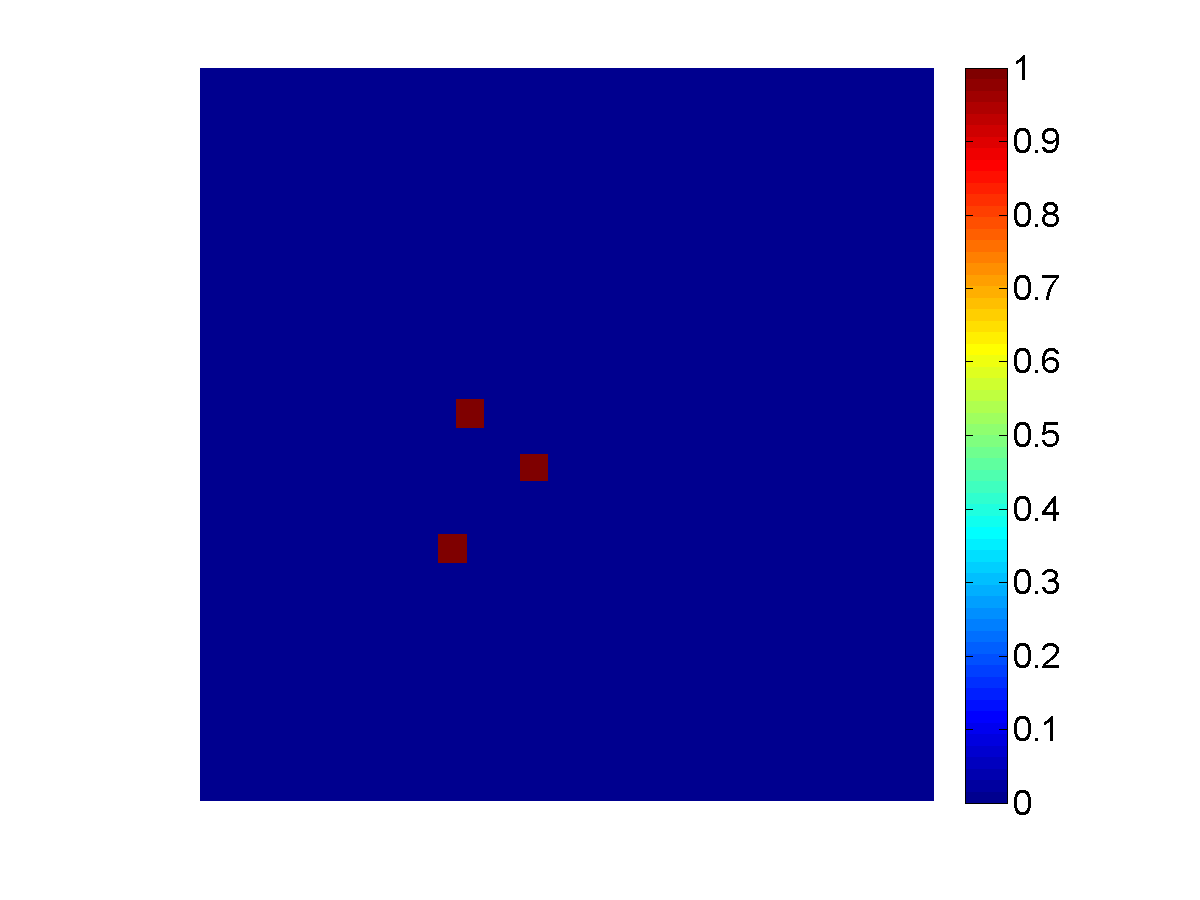} & \includegraphics[trim = 3.5cm 1cm 2cm 1cm, clip=true,width=3.7cm]{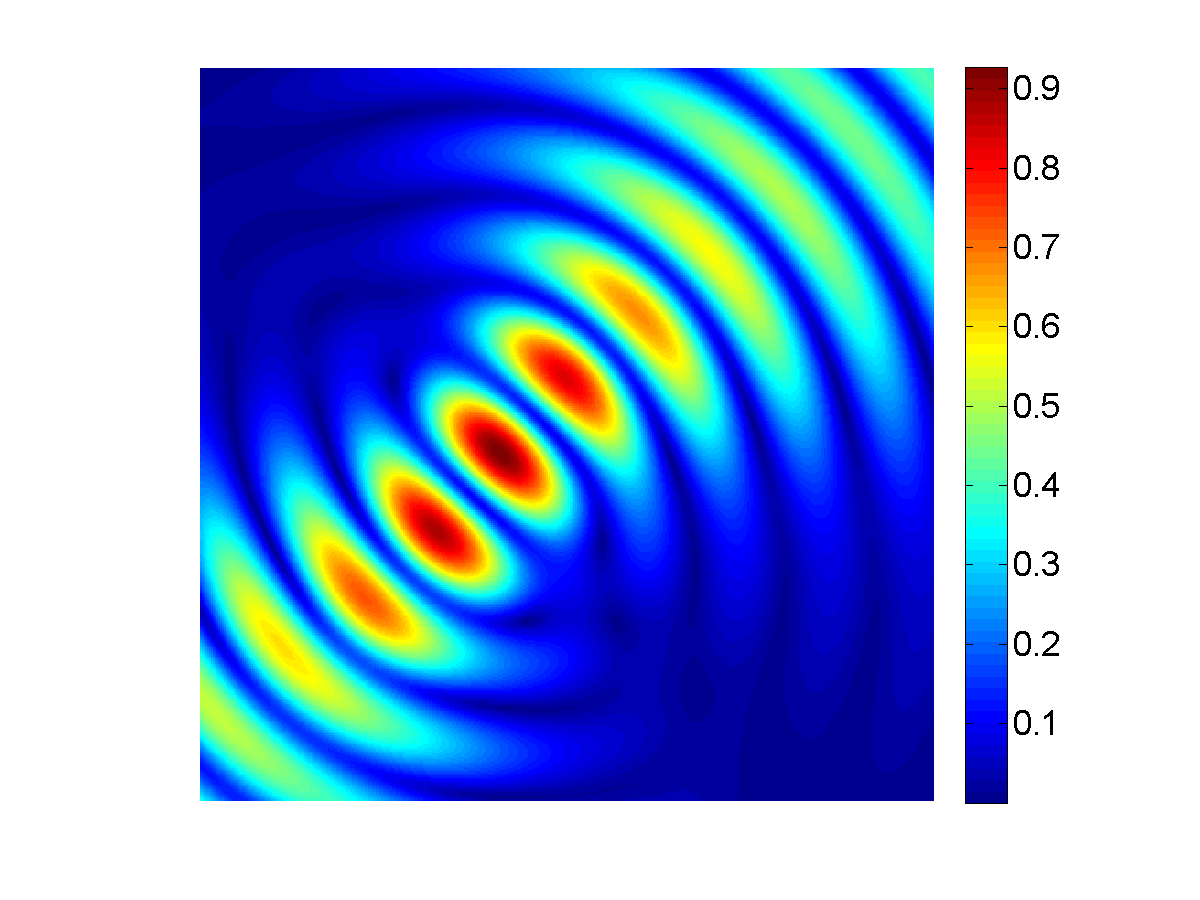}
    & \includegraphics[trim = 3.5cm 1cm 2cm 1cm, clip=true,width=3.7cm]{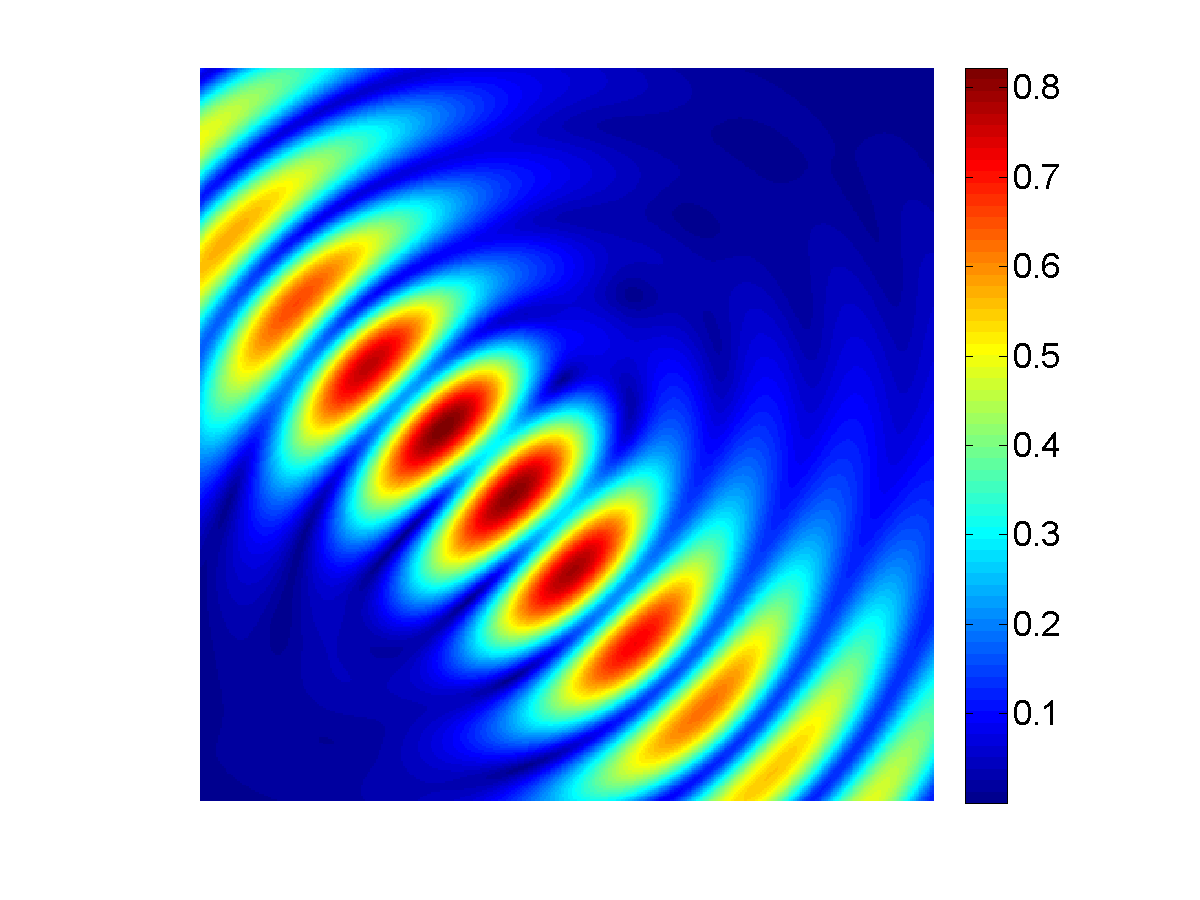} & \includegraphics[trim = 3.5cm 1cm 2cm 1cm, clip=true,width=3.7cm]{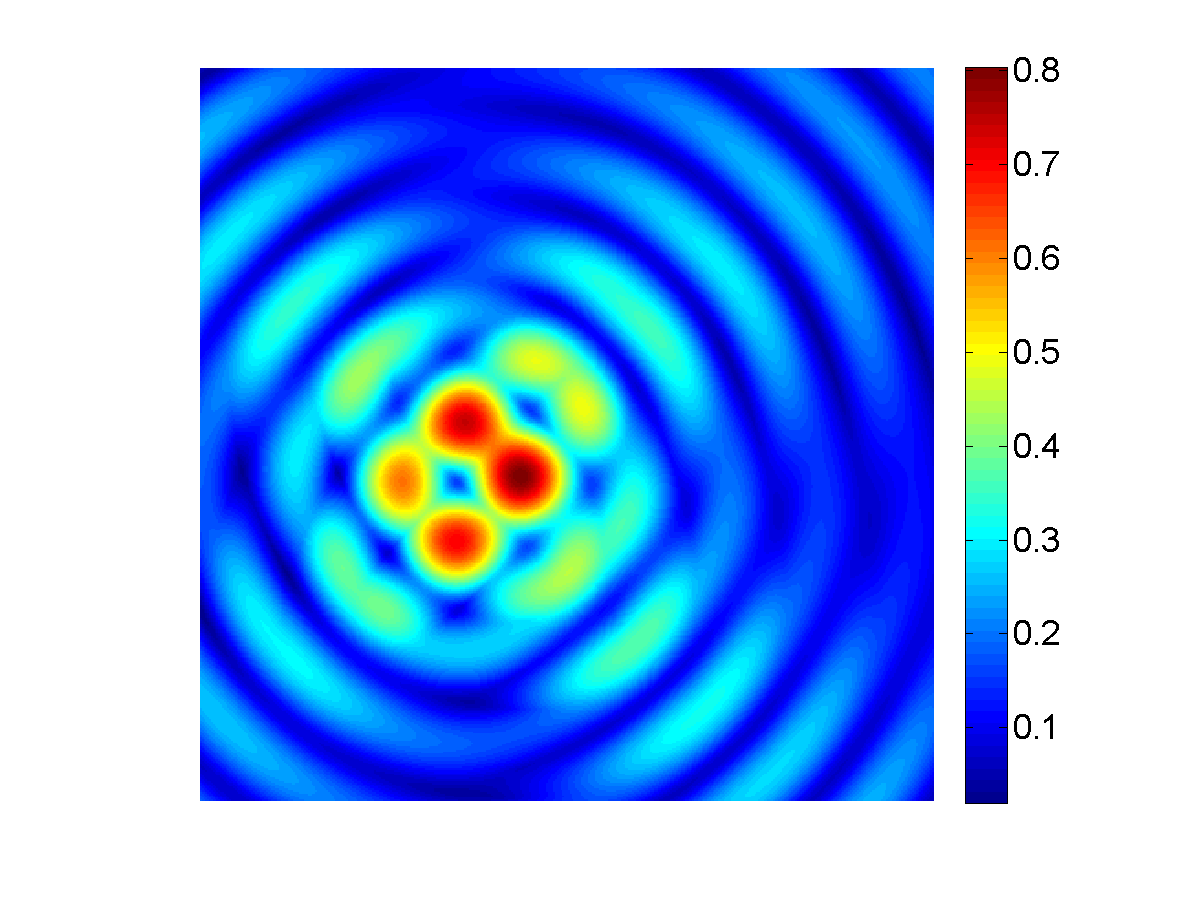}\\
    \includegraphics[trim = 3.5cm 1cm 2cm 1cm, clip=true,width=3.7cm]{ex5_rec_true.png} & \includegraphics[trim = 3.5cm 1cm 2cm 1cm, clip=true,width=3.7cm]{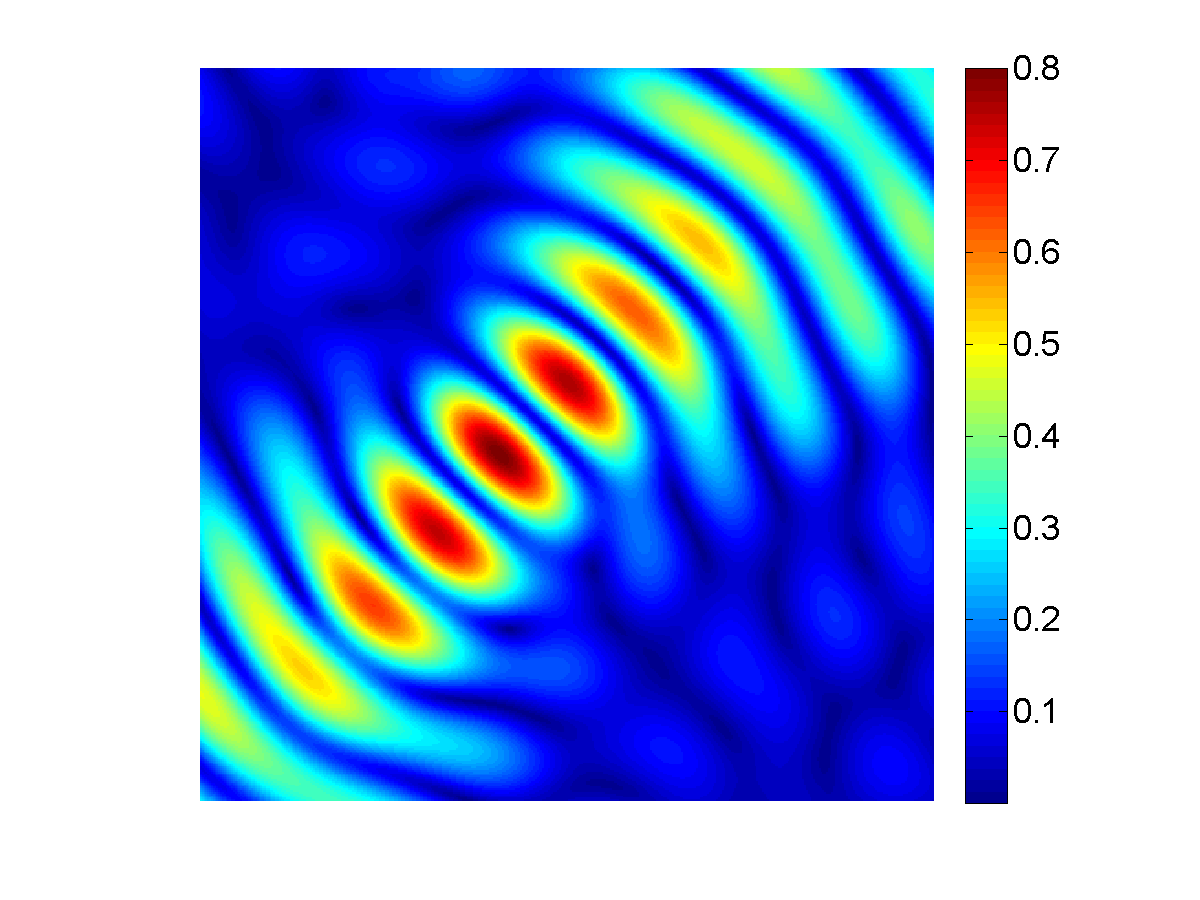}
    & \includegraphics[trim = 3.5cm 1cm 2cm 1cm, clip=true,width=3.7cm]{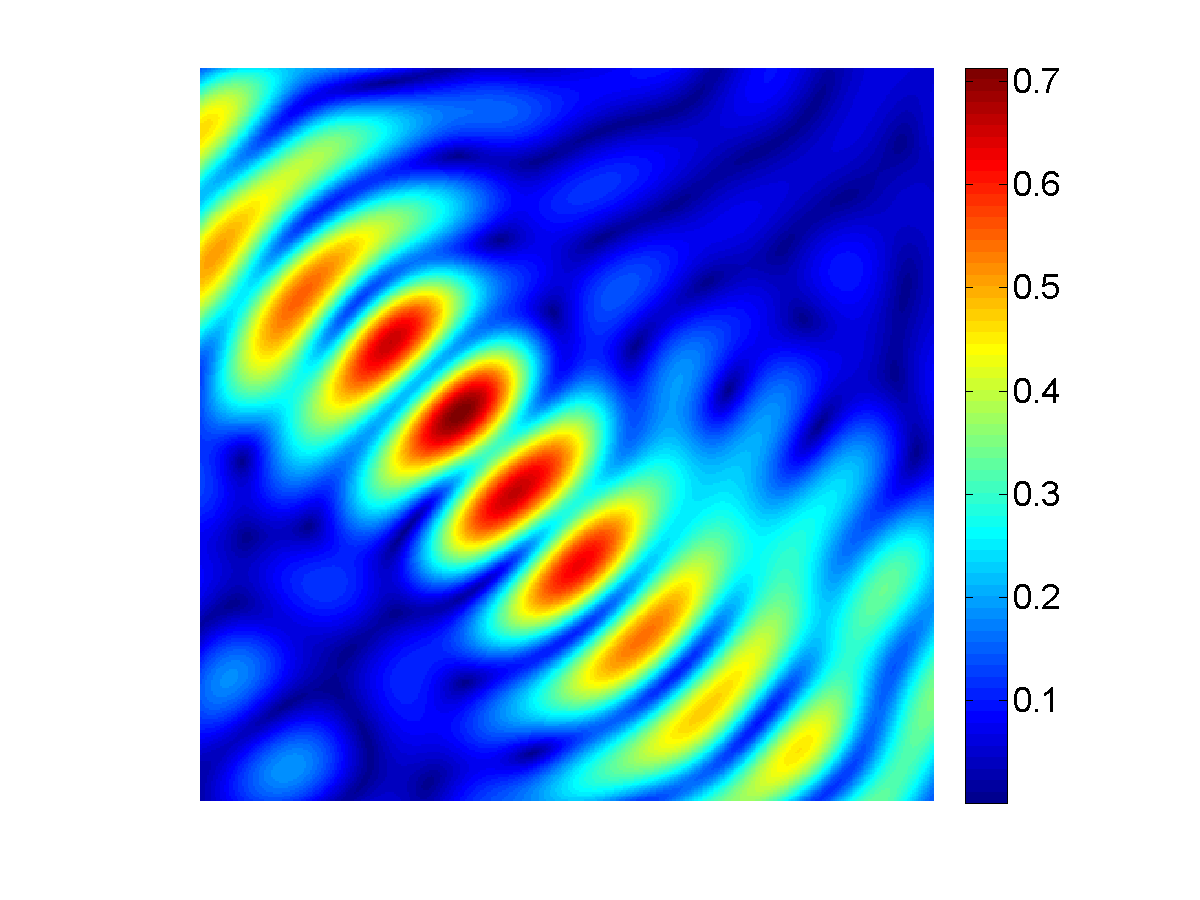} & \includegraphics[trim = 3.5cm 1cm 2cm 1cm, clip=true,width=3.7cm]{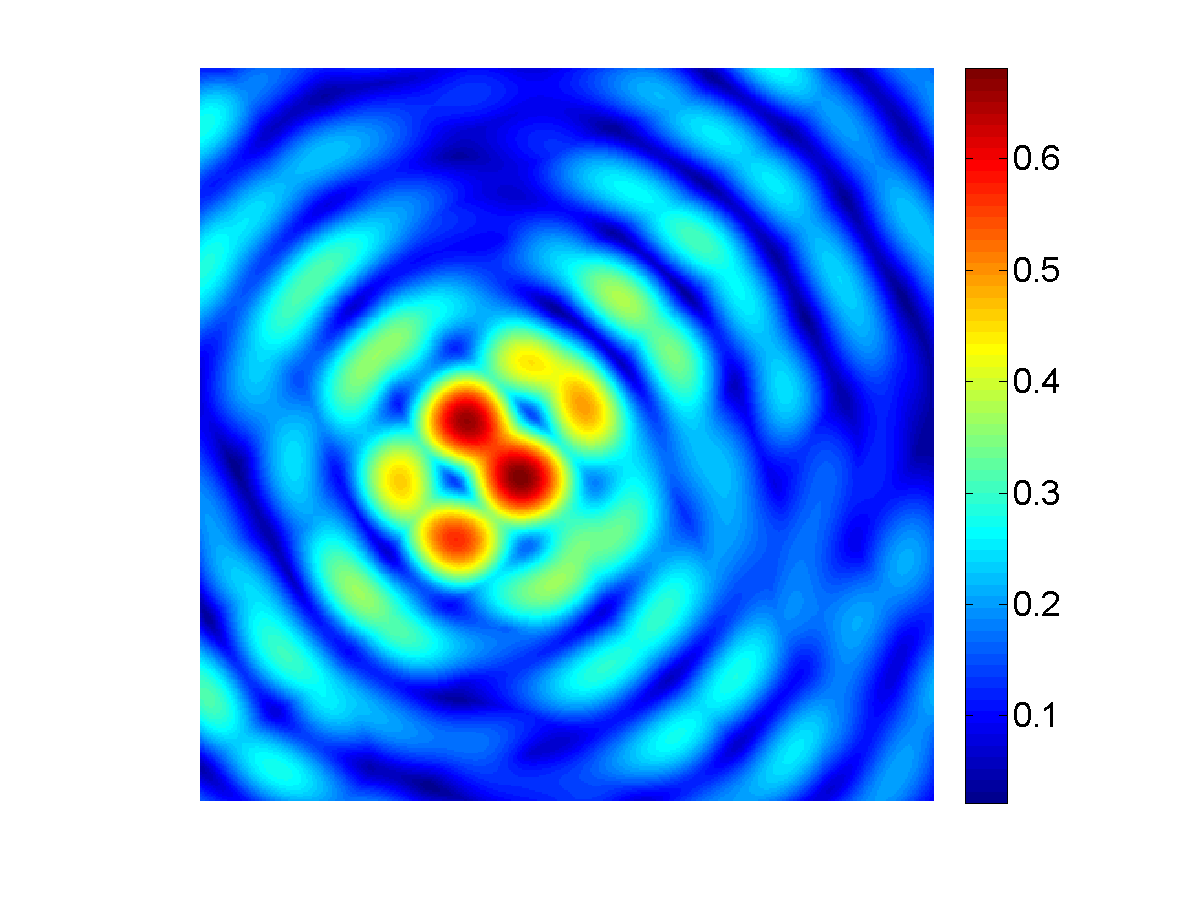}\\
       (a) true scatterers & (b) index $\Psi(x_p;p_1)$ & (c) index $\Psi(x_p;p_2)$ & (d) index $\Psi_c$
  \end{tabular}
  \caption{Numerical results for Example \ref{exam:3sc}. The first and
  second row refer to index functions for the exact data and the noisy
  data with $\epsilon=20\%$ noise.} \label{fig:3sc}
\end{figure}

In this example, the three scatterers stay very close to each other, especially the
upper two, thus it is numerically very challenging to separate them. This is also
reflected in the fact that each individual incident field tends to combine two of the three
scatterers into one larger chunk, which is true
for both the exact data and noisy data; see Figs.\,\ref{fig:3sc} (b)-(c). Therefore,
it is difficult to tell from either Fig.\,\ref{fig:3sc}(b) or Fig.\,\ref{fig:3sc}(c) the number and locations
of the scatterers. The latter is effectively remedied by using two incident fields together;
see Fig.\,\ref{fig:3sc}, where the three scatterers are vividly separately from each
other. However, the interactions between the scatterers focus the
strength on the scatterer to the right, and diminish slightly
the strength of the scatterer to the left.

Next we consider a ring-shaped scatterer.
\begin{exam}\label{exam:ring}
This scatterer is one ring-shaped square scatterer located at the origin, with the
outer and inner side lengths being $0.6$ and $0.4$, respectively. The coefficient $\eta$
of the scatterer is $1$.
\end{exam}

The ring-shaped scatterer represents one of the most challenging objects to resolve, and it
is highly nontrivial even with multiple data sets. The results with the exact data and $\epsilon=20\%$
noise in the data are shown in Fig.\,\ref{fig:ring}. It is observed that with just two incident waves, the
method can provide a quite reasonable estimate of the ring shape, and it remains very  stable
for up to $\epsilon=20\%$ noise in the data.

\begin{figure}[h!]
  \centering
  \begin{tabular}{cccc}
    \includegraphics[trim = 3.5cm 1cm 2cm 1cm, clip=true,width=3.7cm]{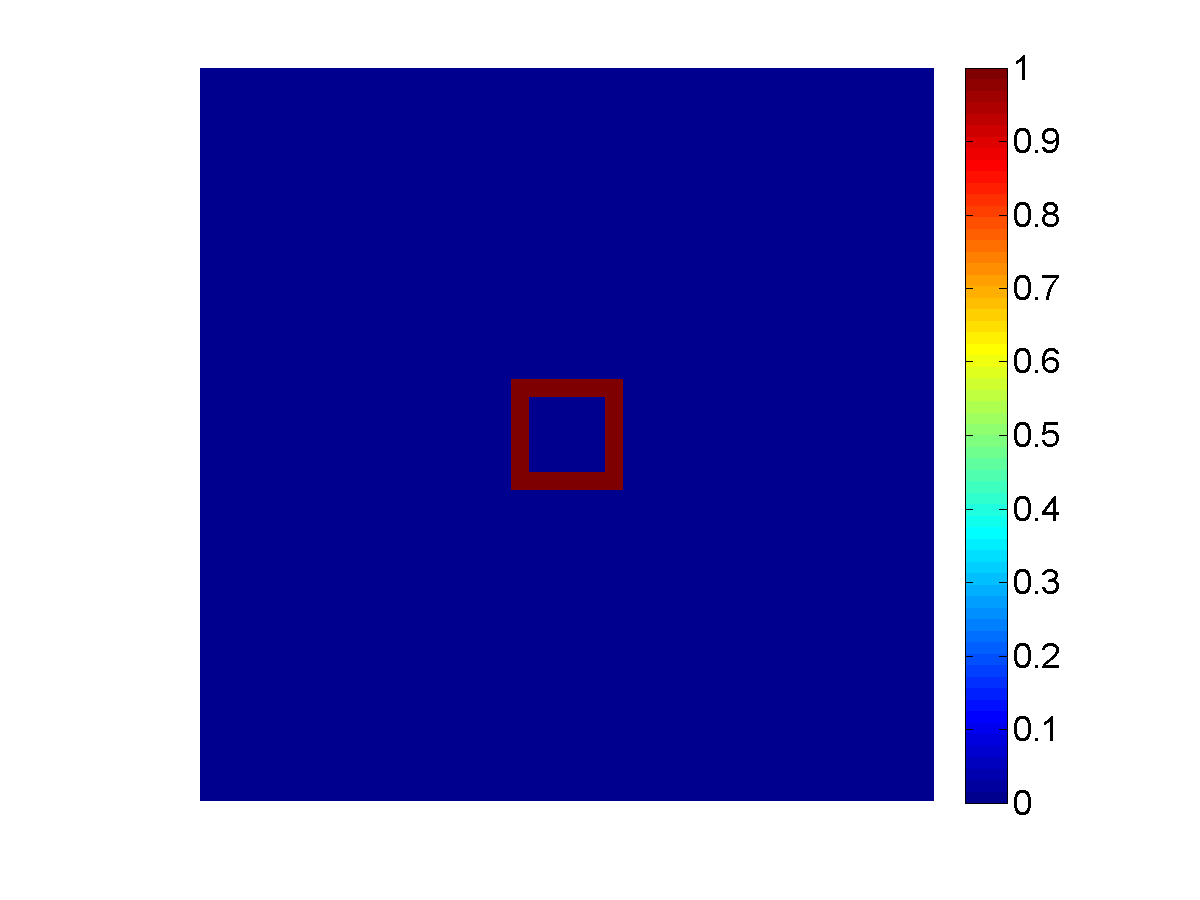} & \includegraphics[trim = 3.5cm 1cm 2cm 1cm, clip=true,width=3.7cm]{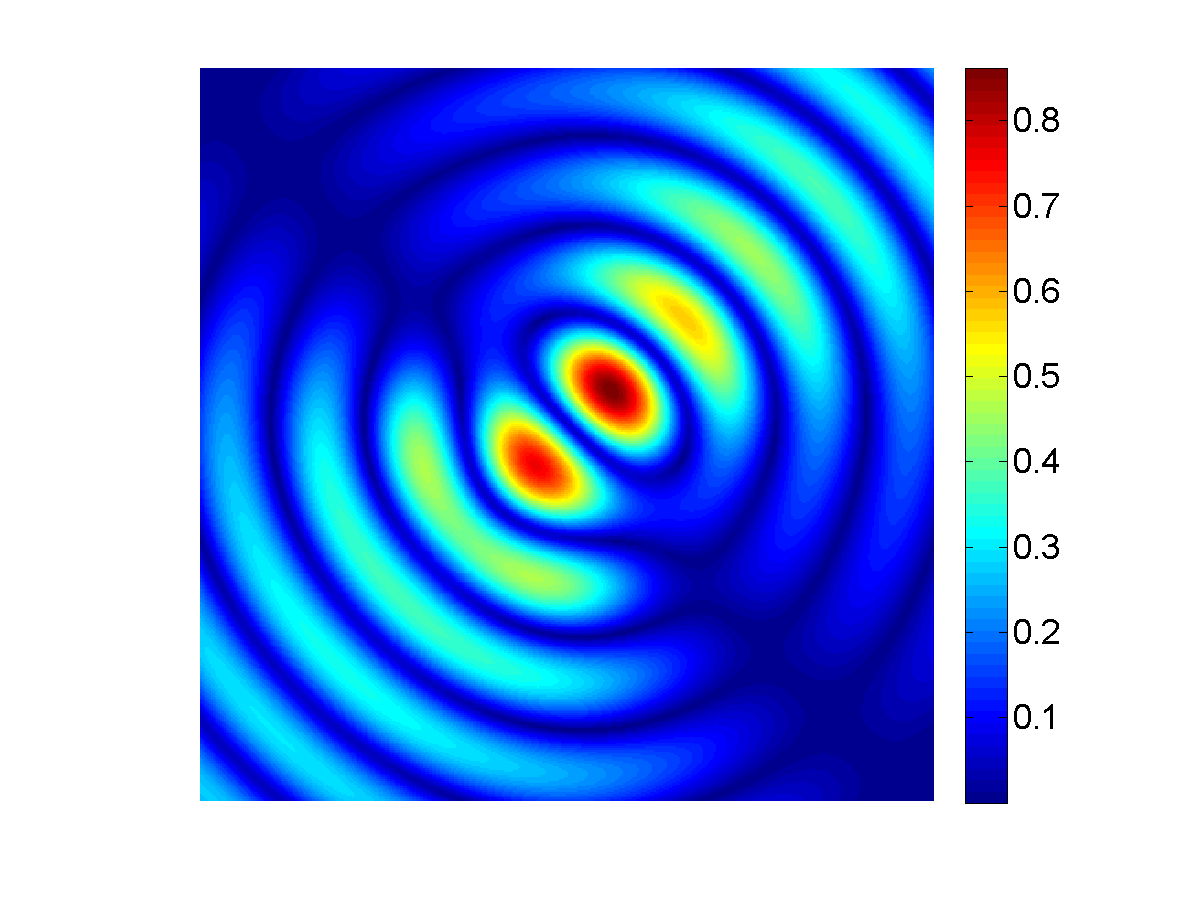}
    & \includegraphics[trim = 3.5cm 1cm 2cm 1cm, clip=true,width=3.7cm]{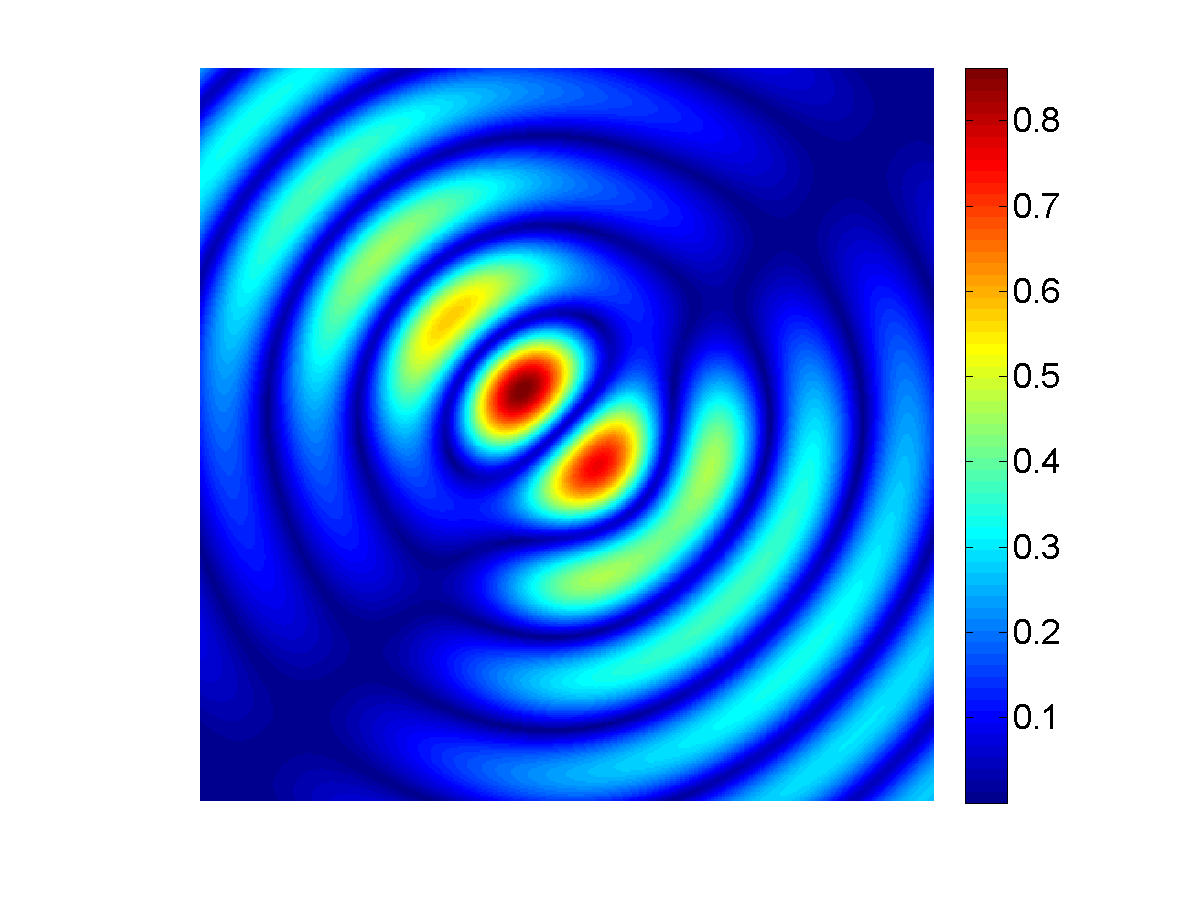} & \includegraphics[trim = 3.5cm 1cm 2cm 1cm, clip=true,width=3.7cm]{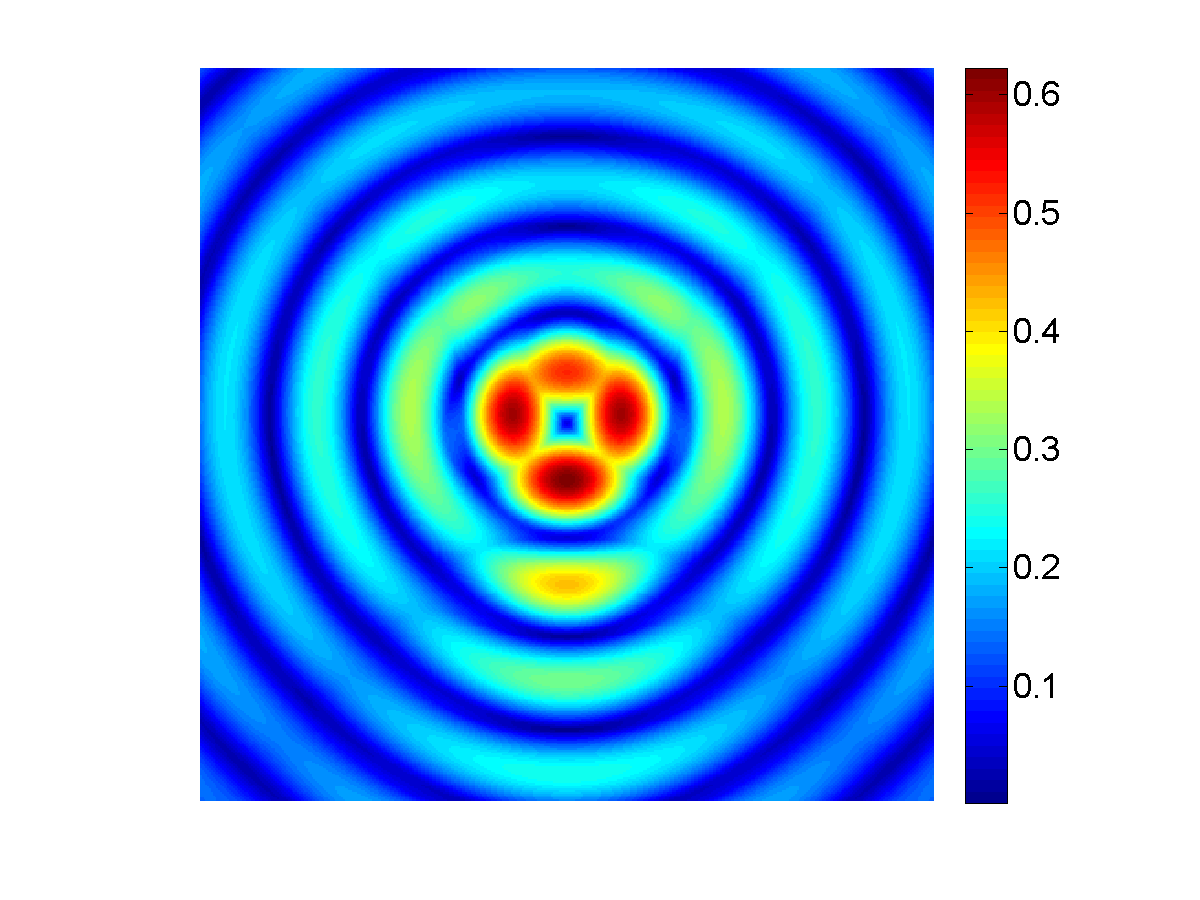}\\
    \includegraphics[trim = 3.5cm 1cm 2cm 1cm, clip=true,width=3.7cm]{ex4_rec_true.png} & \includegraphics[trim = 3.5cm 1cm 2cm 1cm, clip=true,width=3.7cm]{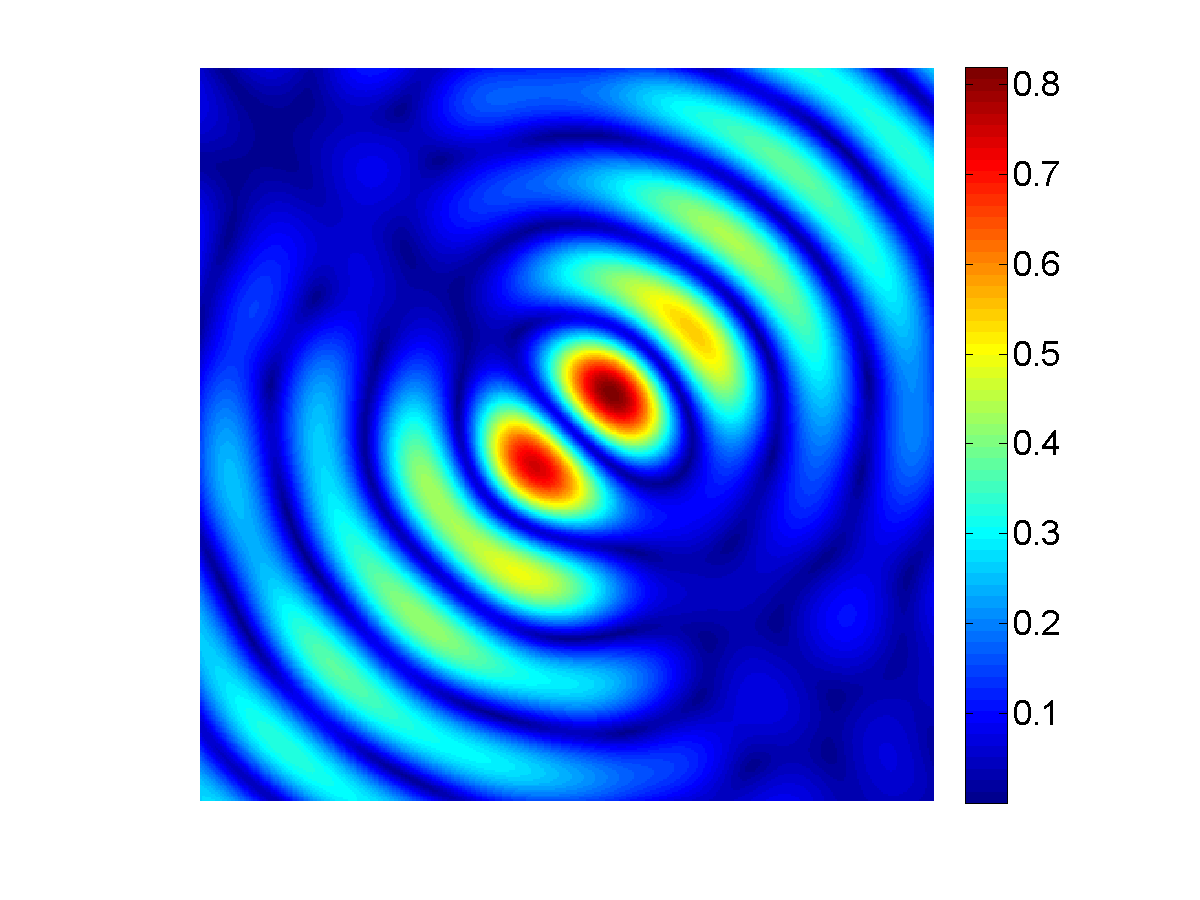}
    & \includegraphics[trim = 3.5cm 1cm 2cm 1cm, clip=true,width=3.7cm]{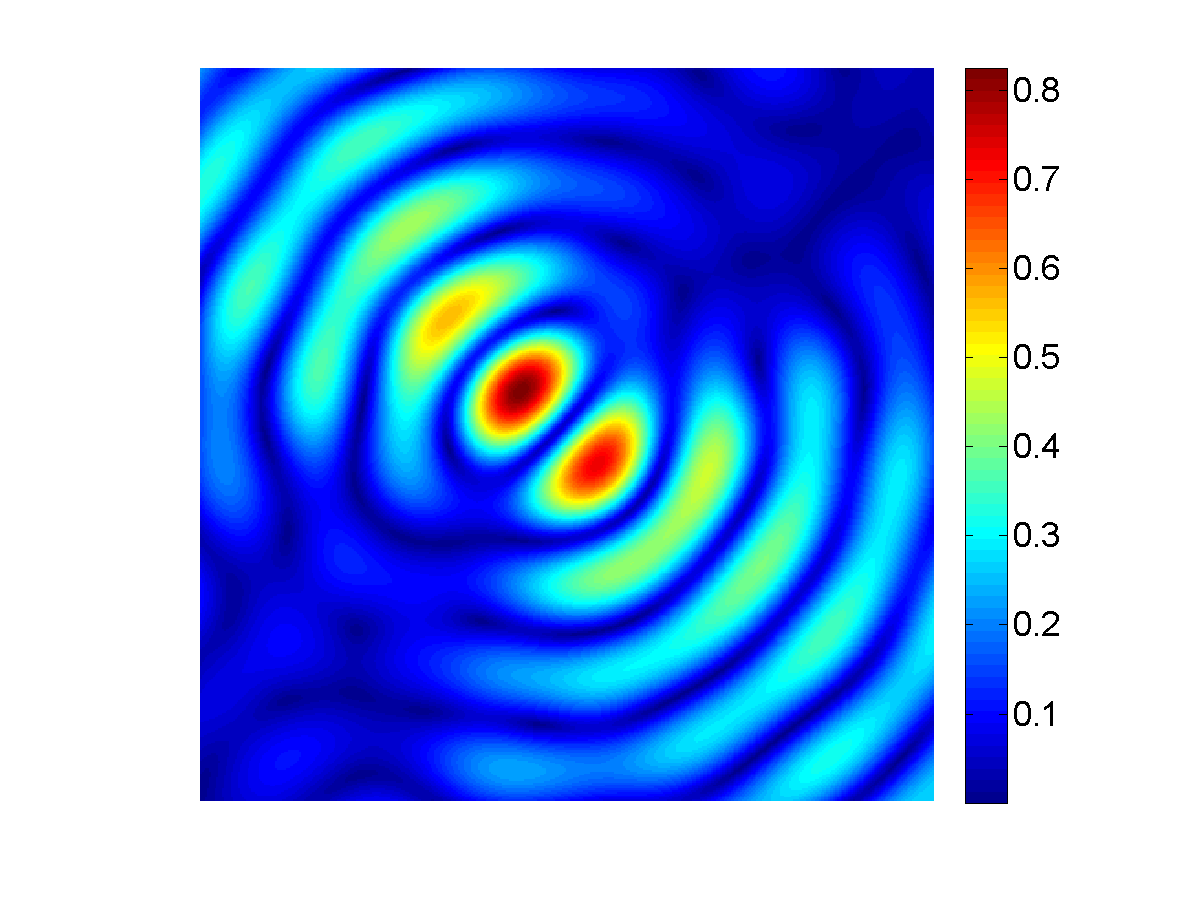} & \includegraphics[trim = 3.5cm 1cm 2cm 1cm, clip=true,width=3.7cm]{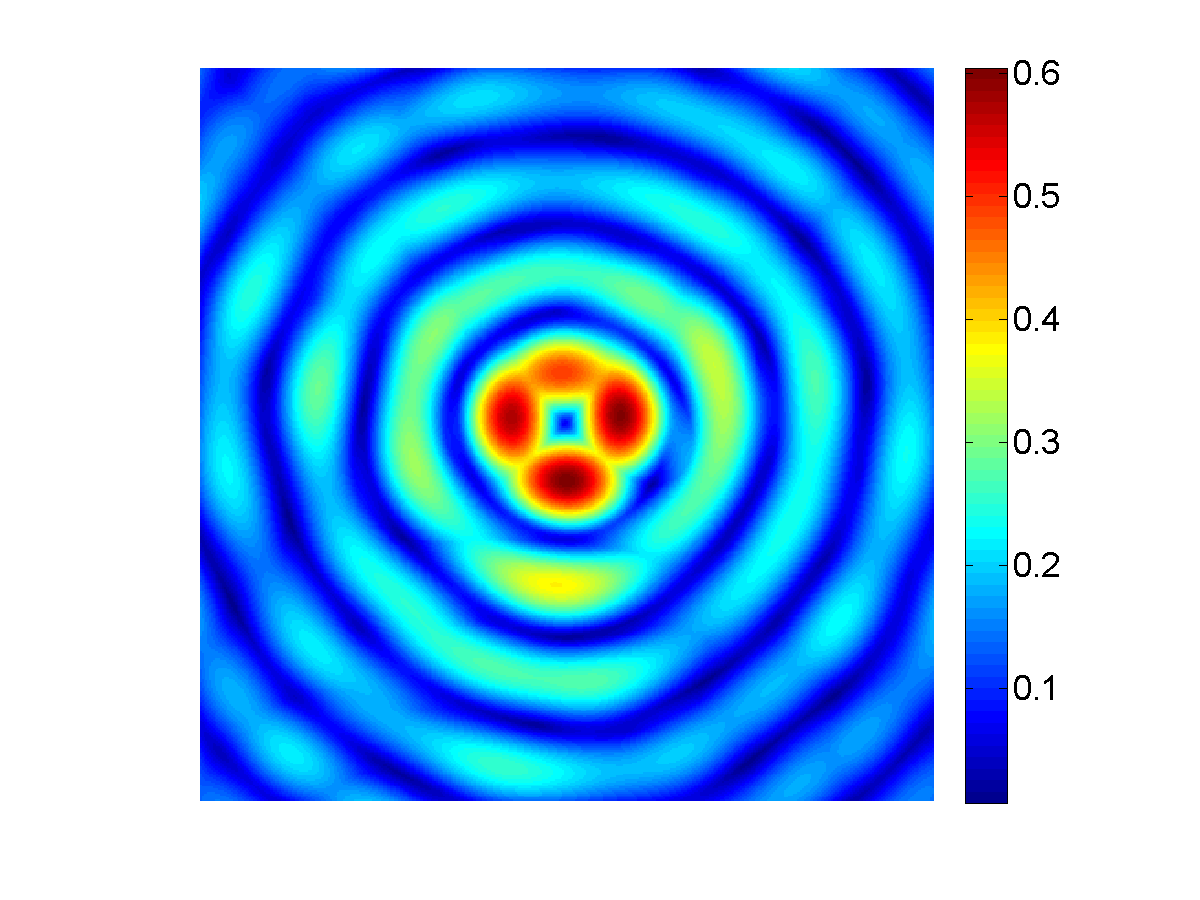}\\
        (a) true scatterers & (b) index $\Psi(x_p;p_1)$ & (c) index $\Psi(x_p;p_2)$ & (d) index $\Psi_c$
  \end{tabular}
  \caption{Numerical results for Example \ref{exam:ring}. The first and
  second row refers to index functions for the exact data and the noisy
  data with $\epsilon=20\%$ noise.} \label{fig:ring}
\end{figure}


\subsection{Three-dimensional example}
The last example shows the feasibility of the method for three-dimensional problems.
\begin{exam}\label{exam:3d}
We consider two cubic scatterers with side length $0.2$: one centered at $(0.4,0.3,0.3)$ and the other at
$(-0.4,0.3,0.3)$, respectively, and the coefficient
$\eta$ in both scatterers is taken to be $1$.
\end{exam}

For this example, we take two incident fields, with the incident directions $d_1=d_2=\frac{1}{\sqrt{3}}
(1,1,1)^\mathrm{t}$ and the polarization vectors
$p_1=\frac{1}{\sqrt{6}}(1,-2,1)^\mathrm{t}$ and $p_2=\frac{1}{\sqrt{6}}(1,1,-2)^\mathrm{t}$.
The scattered field $E^s$ is measured at the points on a uniformly distributed mesh
of $10\times 10$ on each face of the cube of edge length $10$.
The sampling domain $\widetilde{\Omega}$
for evaluating the index functions is taken to be $[-2,2]^3$. The problem is discretized with a mesh size
$0.02$. The numerical results are shown in Fig.\,\ref{fig:3de}. We observe that the support estimated by
the index $\Psi$ agrees very well with the exact one, the magnitude of the index $\Psi$ decreases quickly
away from the boundary of the true scatterers.
The presence of $\epsilon=20\%$
data noise (cf.\,Fig.\,\ref{fig:3dn}), seems to cause no obvious deterioration of the accuracy of the
index $\Psi$ as compared with the exact data.  By examining
the cross-sectional images of the index, we found that
in comparison with the two-dimensional problems, the rippling phenomenon seems less
pronounced for this three-dimensional example.

\begin{figure}[h!]
  \centering
  \begin{tabular}{cccc}
    \includegraphics[width=6.5cm]{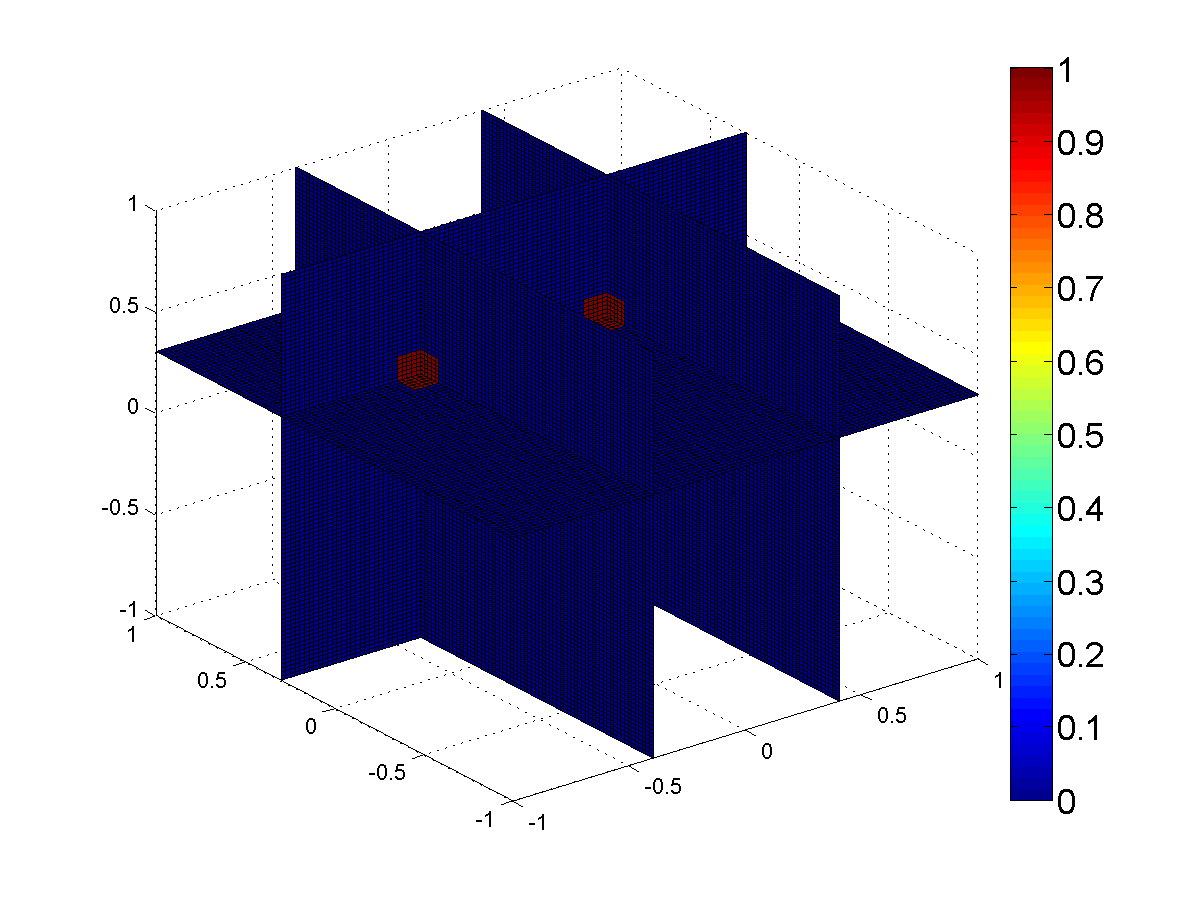} & \includegraphics[width=6.5cm]{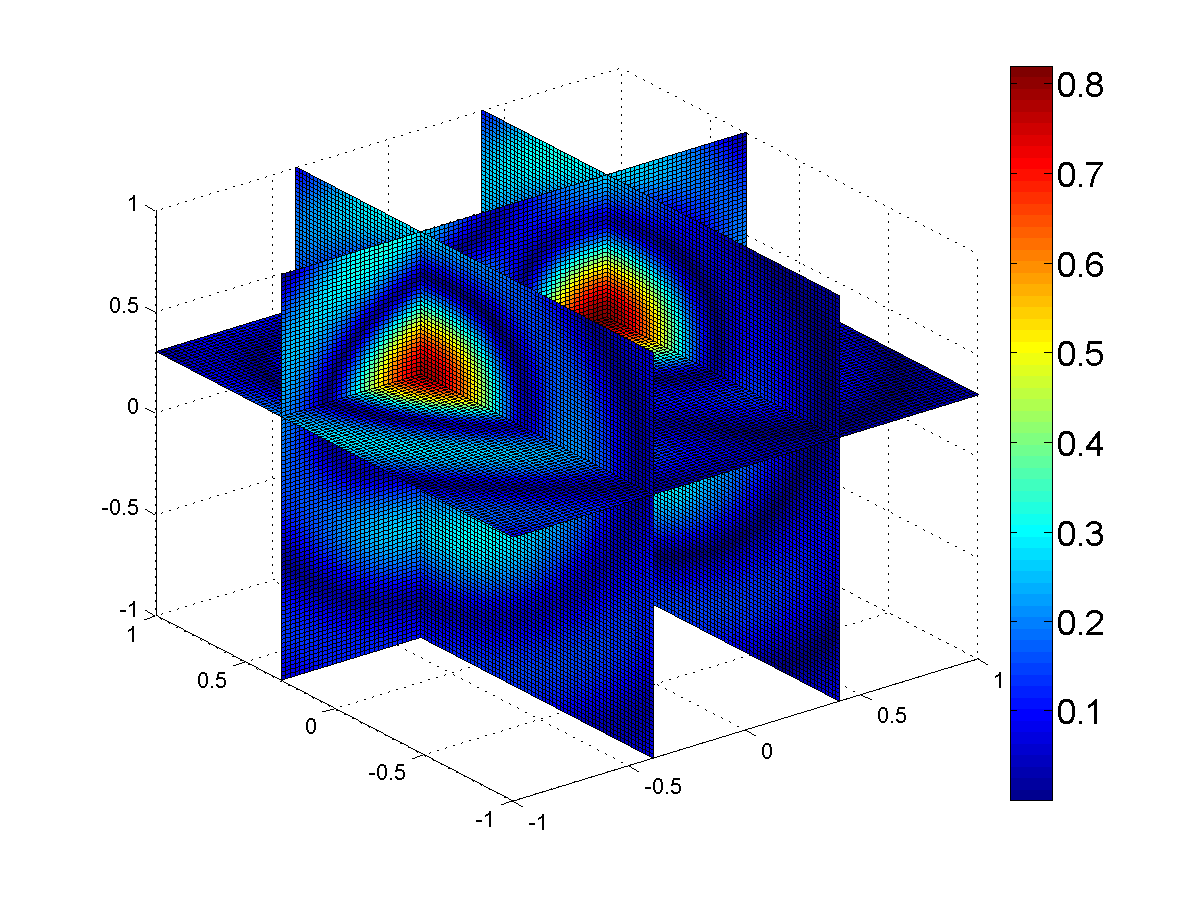}\\
    (a) exact scatterers & (b) index $\Psi(x_p;p_1)$\\
    \includegraphics[width=6.5cm]{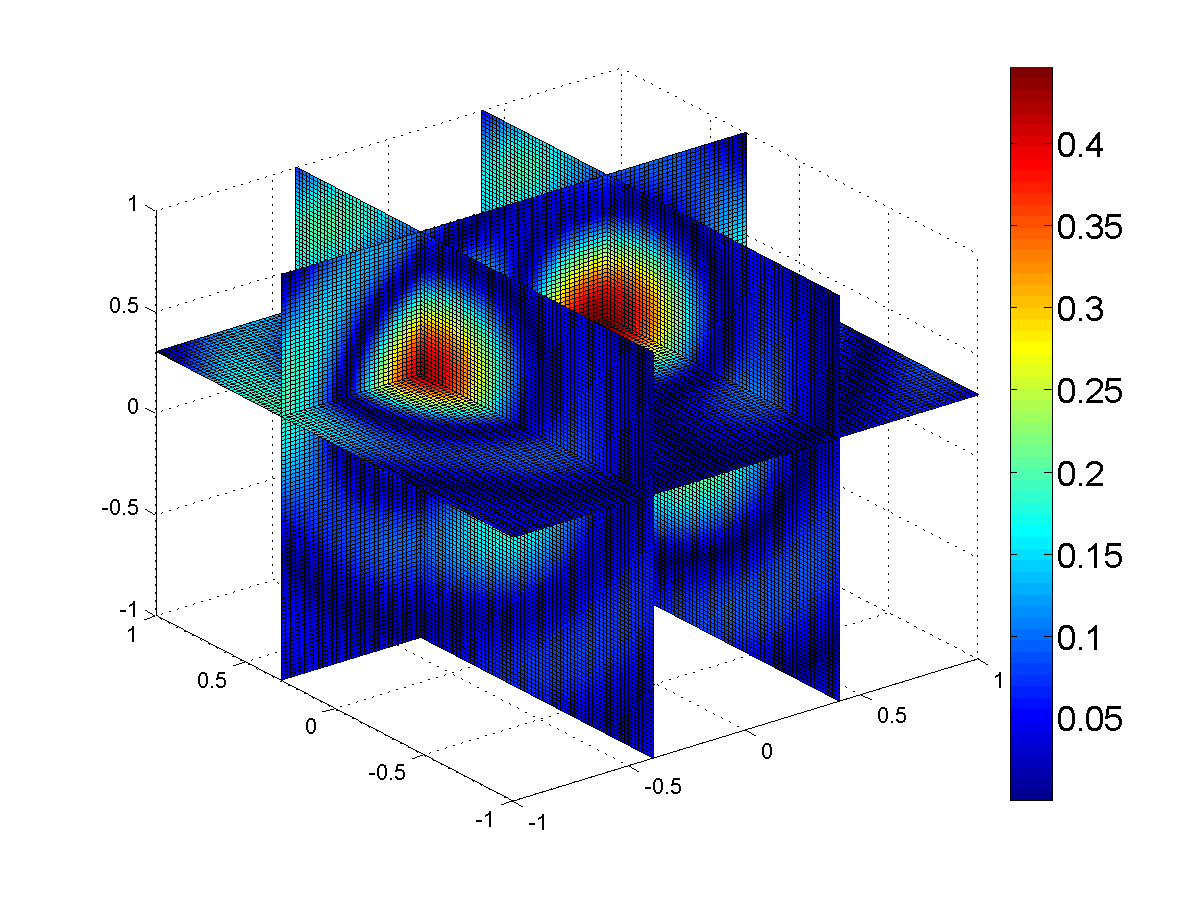} & \includegraphics[width=6.5cm]{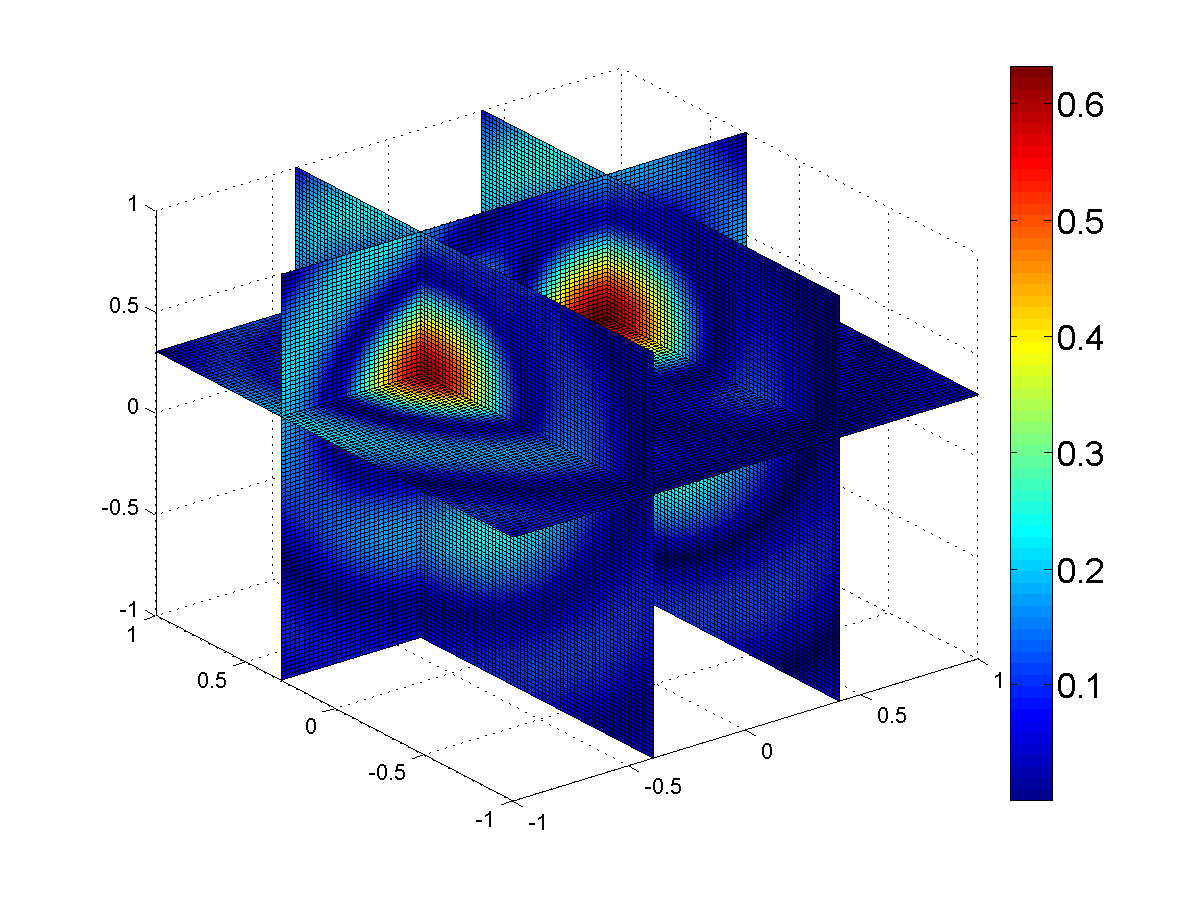}\\
    (c) index $\Psi(x_p;p_2)$ & (d) index $\Psi_c$
  \end{tabular}
  \caption{Numerical results for Example \ref{exam:3d} with exact data.}\label{fig:3de}
\end{figure}

\begin{figure}[h!]
  \centering
  \begin{tabular}{cccc}
    \includegraphics[width=6.5cm]{ex3d_rec_true.png} & \includegraphics[width=6.5cm]{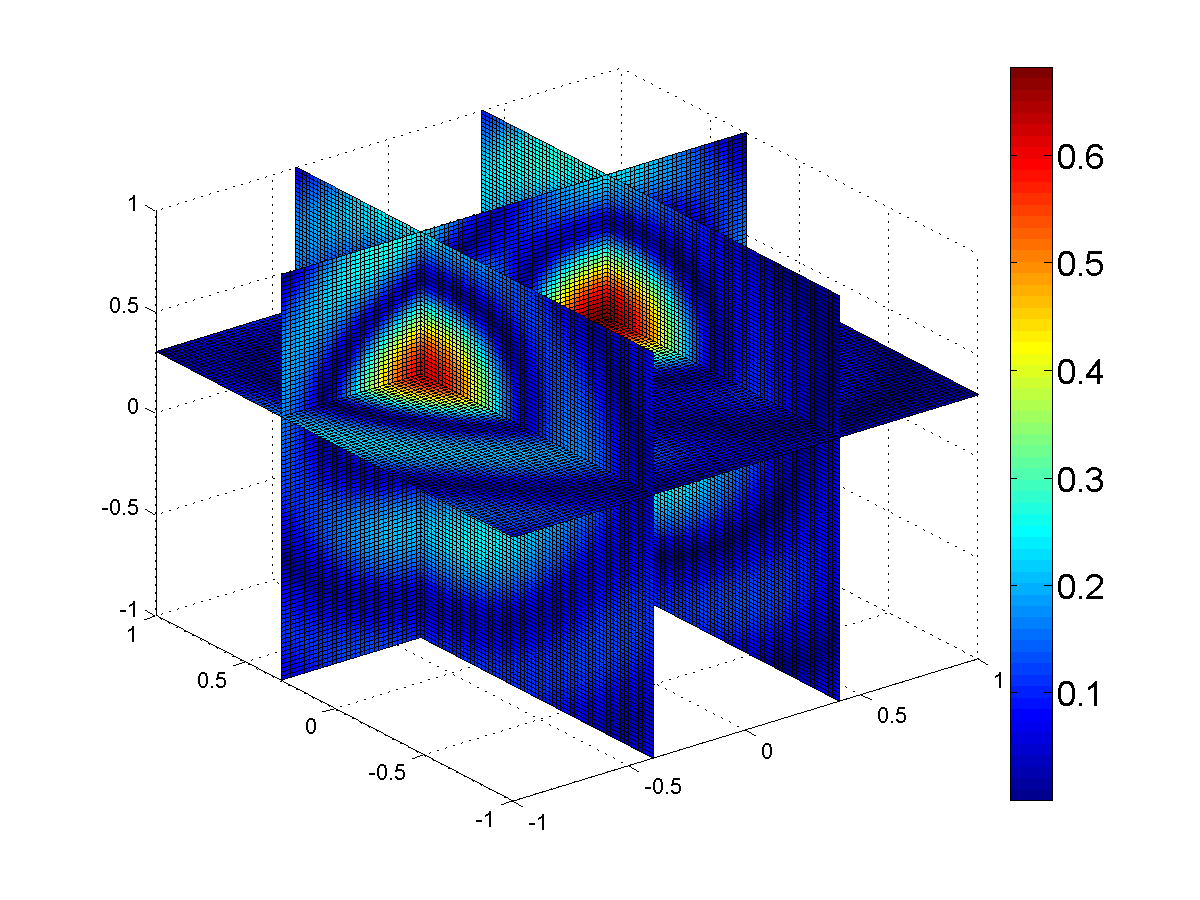}\\
    (a) exact scatterers & (b) index $\Psi(x_p;p_1)$\\
    \includegraphics[width=6.5cm]{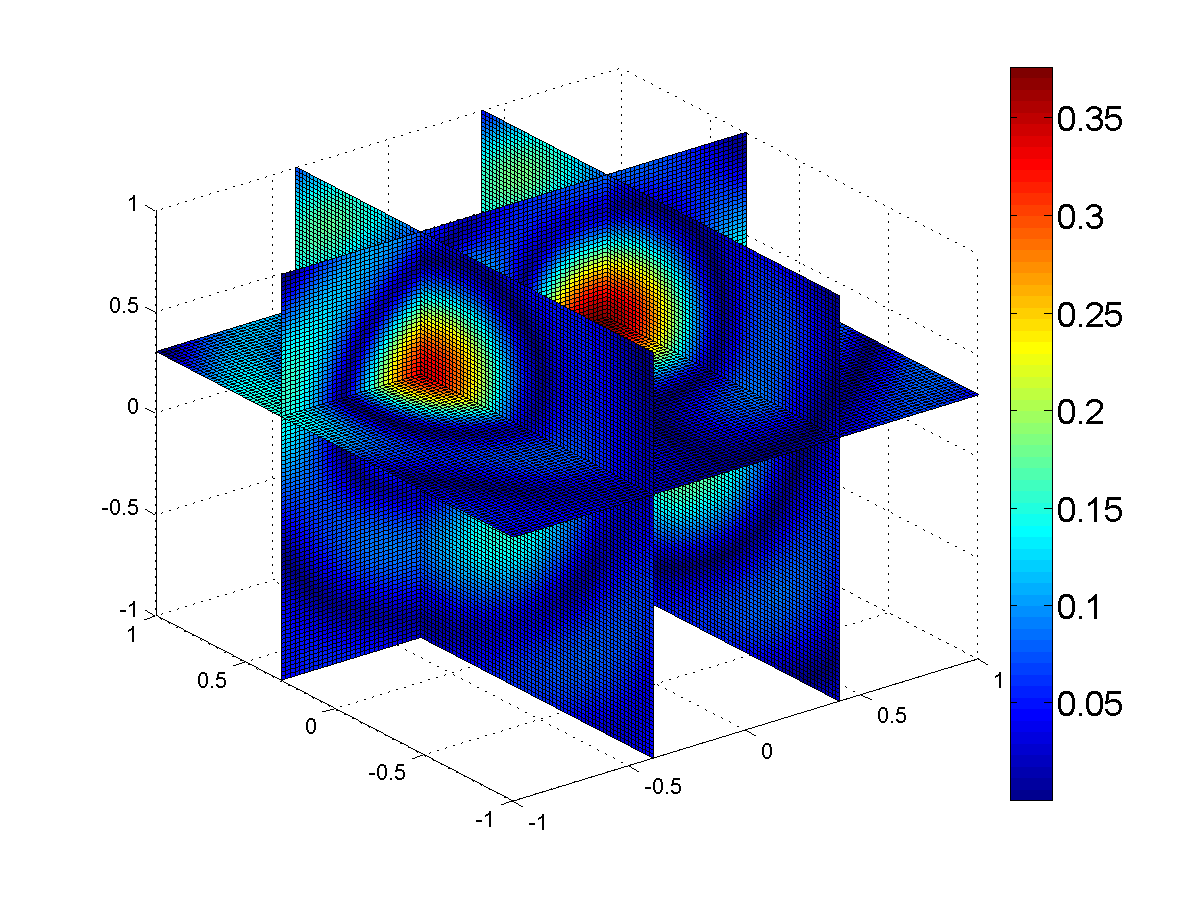} & \includegraphics[width=6.5cm]{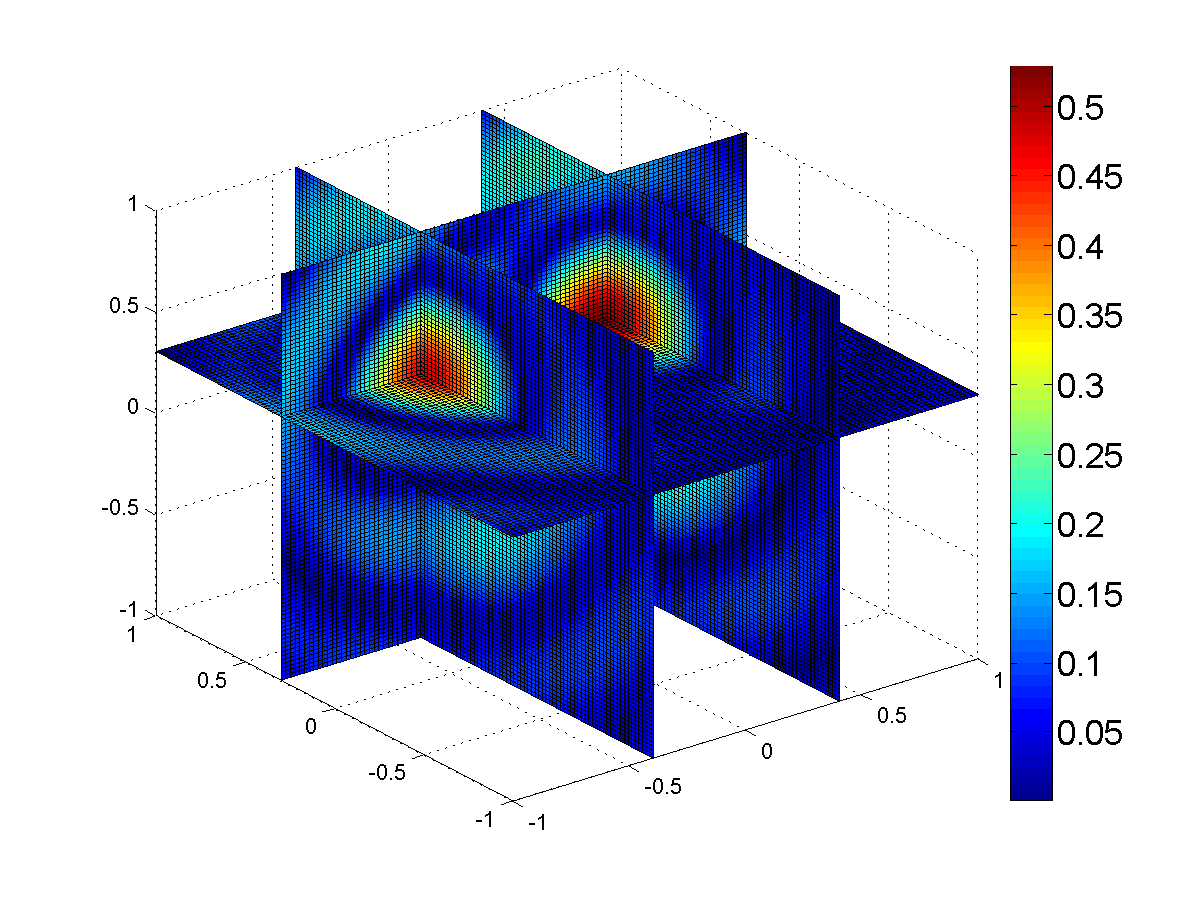}\\
    (c) index $\Psi(x_p;p_2)$ & (d) index $\Psi_c$
  \end{tabular}
  \caption{Numerical results for Example \ref{exam:3d} with $\epsilon=20\%$ data noise.}\label{fig:3dn}
\end{figure}
\section{Conclusions}
We have developed a novel direct sampling method for the inverse electromagnetic medium
scattering problem of estimating the support of inhomogeneities from scattered
electric near-field data. It was derived based on an integral representation of the scattered
field via the fundamental solution, a careful analysis on electromagnetic scattering
and the behavior of the fundamental solutions. The method is particularly attractive in that
it is applicable to a few incident fields. Methodologically, it involves only computing inner
products of the scatted electric field with fundamental solutions located at
the sampling points, hence it is strictly direct, straightforward to implement,
computationally very efficient, and very robust with respect to the presence of data noise.
The experimental results for two- and three-dimensional examples indicate that it can provide,
after thresholding with a suitable cutoff value,
a quite satisfactory estimate of the shapes of the scatterers
from the measured data corresponding to only one or two
incident directions, even in the presence of a large amount of noise.

In the present study, we have focused only on the very cheap
direct sampling method for the reconstructions of
the scatterers. It is natural to enhance the reconstructions
by other more expensive but more accurate methods. 
More specifically, one can use the reconstructions
by the direct sampling method as the initial computational domain
for those more refined methods to retrieve more accurate shapes of the scatterers
and their physical medium properties, e.g., Tikhonov regularization \cite{ItoJinZou:2012jcp}.
Since the indirect imaging methods often involve highly nonlinear optimization processes,
a more accurate and much smaller initial domain can essentially reduce the relevant computational costs.

It is also interesting to see the extensions of the direct sampling method
for other important scattering
scenarios, e.g., scattering from lines (cracks), far-field measurements
and multiple-frequency data, as well as their theoretical justifications.

\section*{Acknowledgements}
The work of BJ was supported by award KUS-C1-016-04, made by
King Abdullah University of Science and Technology (KAUST), and that
of JZ was substantially supported by Hong Kong RGC grants (projects 405110 and 404611).

\appendix

\section{Numerical method for forward scattering}\label{app:int}

In this part we describe our numerical method for the integral equation \eqref{Jeq}
for two-dimensional domains $\Omega$. We denote by $\mathbb{J}$ the index set of
grid points $x_j=(x^1_{j_1},x^2_{j_2}),\ j=(j_1,j_2)\in\mathbb{J},$ of a uniformly distributed
mesh with a mesh size $h>0$ and the square cells $B_j$ given by
\begin{equation*}
  B_j=B_{j_1,j_2}=(x^1_{j_1},x^2_{j_2})+[-\tfrac{h}{2},\tfrac{h}{2}]\times [-\tfrac{h}{2},\tfrac{h}{2}]
\end{equation*}
for every tuple $j=(j_1,j_2)$ in the index set $\mathbb{J}$. Further, we assume that the
set $\cup_{j \in \mathbb{J}} \,B_j$ contains the scatterer support $\Omega$.
We shall approximate the integral operator in \eqref{Jeq} by the mid-point quadrature rule, i.e.,
\begin{equation}\label{eqn:disint}
   J_k +\eta_{k}\, \sum_{j \in \mathbb{I}} G_{k,j} (PJ)_j h^2 = \eta_k \, E^{i}(x_k)
\end{equation}
where $J_k=J(x_k)$ and $\eta_k=\eta(x_k)$, and the off-diagonal entries $G_{k,j}$ and the
diagonal entries $G_{k,k}$ are given by $G_{k,j}=G(x_k,x_j)$ and
\begin{equation*}
G_{k,k}=\frac{1}{h^2} \int_{\left(-\frac{h}{2},\frac{h}{2}\right)^2}  G(x,0)dx,
\end{equation*}
respectively. The diagonal entries $G_{k,k}$ can be accurately computed by a
tensor-product Gaussian quadrature rule. To arrive at a fully discrete scheme,
we further approximate the crucial term $PJ$ in equation \eqref{eqn:disint} by the
central finite difference scheme:
\begin{equation*}
PJ=k^2\left(\begin{array}{cc}J& \\ & J\end{array}\right)+\left(\begin{array}{cc} D_{x_1x_1} & D_{x_1x_2} \\ D_{x_1x_2} & D_{x_2x_2}
\end{array} \right)J,
\end{equation*}
where $D_{x_ix_j}$ refers to taking the second-order derivative with respect to
$x_i$ and $x_j$. In practice, we shall approximate $D_{x_ix_j}$ by central finite
difference scheme, i.e.,
\begin{equation*}
D_{x_1x_1}= H\otimes I,\quad D_{x_2x_2}=I\otimes H,\quad D_{x_1x_2}=D\otimes D,
\end{equation*}
where $\otimes$ is the Kronecker product for matrices, $H$ and $D$ are the
tridiagonal matrices for the second derivative and the first derivative, respectively.
The resulting system can be solved using standard numerical solvers, e.g., Gaussian
elimination, if the cardinality of the index set $\mathbb{J}$ is medium, or
iterative solvers like generalized minimal residual method (GMRES) \cite{Saad:2003}
should be applied if the cardinality of $\mathbb{J}$ is large. Clearly, the extension of the
procedure to 3D problems is straightforward. Our numerical experiences indicate that
tens of GMRES iterations  already yield a very accurate solution to the linear system.

\bibliographystyle{abbrv}
\bibliography{maxwell}

\end{document}